\documentclass[11pt]{article}

\usepackage[left=1in,right=1in,top=1in,bottom=1in]{geometry}

\usepackage{amsfonts, amssymb, amsrefs, amsmath, mathtools}
\usepackage{amsthm}
\usepackage[colorlinks]{hyperref}
\hypersetup{linkcolor=blue!85!black, citecolor=cyan}
\usepackage{mathrsfs}
\usepackage[nameinlink]{cleveref}
\usepackage{enumitem}
\usepackage{chemfig}
\usepackage{authblk}\usepackage[symbol]{footmisc}
\usepackage[textsize=footnotesize, backgroundcolor=gray!25]{todonotes}
\usepackage{tikz, graphicx}
\usetikzlibrary{positioning}
\usetikzlibrary{arrows}
\usetikzlibrary{cd}
\usepackage{pgfplots}
\pgfplotsset{compat=1.18}
\usepackage{orcidlink}
\usepackage{marvosym}
\usepackage{caption}
\usepackage{color, xcolor}
\definecolor{Egraphcolor}{RGB}{0,160,160} 
\definecolor{fluxcolor}{RGB}{219, 48, 122}
\newcommand\fluxcolor[1]{{\textcolor{fluxcolor}{#1}}}

\newtheoremstyle{dotless}{}{}{\itshape}{}{\bfseries}{}{ }{}
\theoremstyle{dotless}
\newtheorem{theorem}{Theorem}[section]
\newtheorem{lemma}[theorem]{Lemma}
\newtheorem{corollary}[theorem]{Corollary}
\newtheorem{proposition}[theorem]{Proposition}

\newtheorem{theoremalphabetic}{Theorem}

\newtheoremstyle{notitalic}{}{}{}{}{\bfseries}{}{ }{}
\theoremstyle{notitalic}
\newtheorem{definition}[theorem]{Definition}
\newtheorem{remark}[theorem]{Remark}

\newtheorem{example}[theorem]{Example}
\AtBeginEnvironment{example}{%
  \pushQED{\qed}%
}
\AtEndEnvironment{example}{\popQED\endexample}

\newcommand{\st}{\mid} 
\newcommand\mc[1]{\mathcal{#1}}  
\newcommand{\rr}{\ensuremath{\mathbb{R}}}   
\renewcommand{\phi}{\varphi}
\newcommand{\zz}{\ensuremath{\mathbb{Z}}}   

\DeclareMathOperator{\spn}{span}
\DeclareMathOperator{\grad}{grad}
\DeclareMathOperator{\sgn}{sgn}
\DeclareMathOperator{\comp}{comp}

\DeclareMathOperator{\relint}{relint}

\newcommand{\kk}{\kappa}
\newcommand{\vv}[1]{{\boldsymbol{#1}}} 
\newcommand{\rrp}{\rr_{\geq 0}}
\newcommand{\rrpp}{\rr_{>0}}
\newcommand{\xx}{\vv x}
\newcommand{\yy}{\vv y}

\newcommand{\Hwr}{\mc{H}^{\operatorname{wr}}}

\newcommand{\Keq}{\mc{K}^{\operatorname{eq}}}
\newcommand{\Kt}{\mc{K}^{\operatorname{t}}}
\newcommand{\Kdt}{\mc{K}^{\operatorname{dt}}}

\newcommand{\Feq}{\mc{F}^{\operatorname{eq}}}
\newcommand{\Ft}{\mc{F}^{\operatorname{t}}}
\newcommand{\Fdt}{\mc{F}^{\operatorname{dt}}}

\newcommand{\0}{\mathsf{0}}
\newcommand{\X}{\mathsf{X}}

\begin{document}

\title{
A flux-based approach for analyzing the\\ disguised toric locus of reaction networks}

\author[1,2]{Bal\'azs Boros}
\author[3,4]{Gheorghe Craciun}
\author[5,6]{Oskar Henriksson}
\author[7]{Jiaxin Jin}
\author[3]{Diego~Rojas~La~Luz}

\affil[1]{\small Bolyai Institute, University of Szeged}
\affil[2]{\small National Laboratory for Health Security, University of Szeged}
\affil[3]{\small Department of Mathematics, University of Wisconsin-Madison}
\affil[4]{\small Department of Biomolecular Chemistry, University of Wisconsin-Madison}
\affil[5]{\small Max Planck Institute of Molecular Cell Biology and Genetics, Dresden}
\affil[6]{\small Center for Systems Biology Dresden}
\affil[7]{\small Department of Mathematics, University of Louisiana at Lafayette}

\date{\vspace{-2em}} 

\maketitle

\footnotetext{\orcidlink{0000-0001-5417-4565} 0000-0001-5417-4565; \Letter\ \href{mailto:boros.balazs@szte.hu}{boros.balazs@szte.hu}}
\footnotetext{\orcidlink{0000-0003-1344-3890} 0000-0003-1344-3890; \Letter\ \href{mailto:craciun@wisc.edu}{craciun@wisc.edu}}
\footnotetext{\orcidlink{0000-0002-2742-0027} 0000-0002-2742-0027; \Letter\ \href{mailto:oskar.henriksson@mpi-cbg.de}{oskar.henriksson@mpi-cbg.de}}
\footnotetext{\orcidlink{0000-0003-0129-7590} 0000-0003-0129-7590; \Letter\ \href{mailto:jiaxin.jin@louisiana.edu}{jiaxin.jin@louisiana.edu}}
\footnotetext{\orcidlink{0009-0005-1626-4857} 0009-0005-1626-4857; \Letter\ \href{rojaslaluz@wisc.edu}{rojaslaluz@wisc.edu}}

\begin{abstract}
\noindent
Dynamical systems with polynomial right-hand sides are very important in various applications, e.g., in biochemistry and population dynamics. The mathematical study of these dynamical systems is challenging due to the possibility of multistability, oscillations, and chaotic dynamics. One important tool for this study is the concept of \emph{reaction systems}, which are dynamical systems generated by reaction networks for some choices of parameter values.  
Among these, {\em disguised toric systems} are remarkably stable: they have a unique attracting fixed point, and \emph{cannot} give rise to oscillations or chaotic dynamics. 
The computation of the set of parameter values for which a network gives rise to disguised toric systems (i.e., the {\em disguised toric locus of the network}) is an important but difficult task. We introduce new ideas based on {\em network fluxes} for studying the disguised toric locus. We prove, under mild assumptions, that the disguised toric locus of any network $G$ is a contractible manifold with boundary, and introduce an associated graph $G^{\max}$ that characterizes its interior. These theoretical tools allow us, for the first time, to compute the full disguised toric locus for many networks of interest. 
\end{abstract}

\vskip3cm

\phantom{.}

\section{Introduction}

Mathematical models for many important questions from biochemistry, ecology, population dynamics, and the study of infectious diseases give rise to complex nonlinear dynamical systems where the variables are (nonnegative) concentrations or populations of interest. It is very common in practice that the right-hand side of these dynamical systems is given by polynomials. Polynomial dynamical systems on the nonnegative orthant can exhibit any of the complex dynamics of general polynomial dynamical systems, such as multiple basins of attraction, periodic trajectories, or chaos, and their mathematical analysis is very challenging. On the other hand, the dynamics of any polynomial dynamical systems on the positive orthant can be obtained by using {\em reaction systems}, which are dynamical systems generated by reaction networks for some choices of rate constant parameters, according to the law of mass action.  

In general, the mathematical analysis of reaction systems can also be very challenging. A remarkable exception is reaction systems that have a {\em vertex-balanced equilibrium}~\cite{horn:jackson:1972}, which are called \emph{toric systems} (also known as \emph{complex-balanced systems} or \emph{vertex-balanced systems}). Toric systems have been introduced in the seminal work of Horn and Jackson, and they enjoy a unique locally stable equilibrium within each linear invariant subspace (that, under certain additional assumptions, is globally asymptotically stable), and cannot exhibit oscillations or chaotic dynamics~\cite{horn:jackson:1972}. This can be seen by showing that for each vertex-balanced equilibrium $\vv x^*$, we have a global strict Lyapunov function, given by the Horn--Jackson function
\begin{align*}
L(\vv x) = \sum_{i=1}^n x_i (\log \tfrac{x_i}{x_i^*} - 1).
\end{align*}

Whether a network $G$ admits a vertex-balanced equilibrium usually depends on the rate constant parameters. The \emph{toric locus} $\Kt(G)$ is the set of rate constants under which the network $G$ admits a vertex-balanced equilibrium. 
In \cite{craciun:dickenstein:shiu:sturmfels:2009}, the set $\Kt(G)$ has been characterized algebraically. For $\Kt(G)$ to be nonempty, the network $G$ needs to be weakly reversible, meaning every component of $G$ is strongly connected. For a weakly reversible network $G$, the codimension of $\Kt(G)$ equals the {\em deficiency} of $G$, a nonnegative integer that measures the affine dependency of the vertices of $G$. 

Yet, most real-world networks have high deficiency or are not weakly reversible, making them ineligible for vertex balancing. However, many of these systems can be \emph{dynamically equal} to a vertex-balanced system (possibly on another network), in which case the stability properties of vertex-balanced systems are shared by the non-vertex-balanced system. This notion first appeared in \cite{craciun:jin:yu:2020}, and gave rise to the notion of \emph{disguised toric locus}, introduced in \cite{brustenga:craciun:sorea:2022} and further analyzed in recent years in~\cites{haque:satriano:sorea:yu:2023, craciun:deshpande:jin:2024a, craciun:deshpande:jin:2024b, craciun:deshpande:jin:2025a, craciun:deshpande:jin:2025b}. For example, it has been shown that the disguised toric locus is path-connected \cite{craciun:deshpande:jin:2024b},  is invariant under affine transformations of the network \cite{haque:satriano:sorea:yu:2023}, and there exist methods for calculating its dimension \cite{craciun:deshpande:jin:2025a}. 
Often, the disguised toric locus of a network is significantly larger than its toric locus. However, explicitly computing parametric or implicit semialgebraic descriptions of the disguised toric locus remains a big challenge. Previously, this has only been achieved for certain families of small networks \cite{brustenga:craciun:sorea:2022}.

In this paper, we propose a new method for studying and computing the disguised toric locus based on the concept of \emph{fluxes}, which are the rates of the reactions evaluated at a particular state. We consider the \emph{disguised toric flux cone} $\Fdt(G)$, which is the set of reaction fluxes on $G$ that are realizable by some vertex-balanced fluxes (on a possibly different network). Unlike the semialgebraic set $\Kdt(G)$, the cone $\Fdt(G)$ is always a polyhedral object, which can be computed through linear programming. It turns out that the topology of $\Kdt(G)$ is fully captured by the topology of $\Fdt(G)$ in the sense of the following result, which extends a homeomorphic map previously constructed in the \emph{toric} setting in \cite{craciun:jin:sorea:2024} to the \emph{disguised toric} setting.

\begin{theoremalphabetic}[{\Cref{thm:homeo_Psi} and \Cref{cor:Kdt_G}}]
\label{thmA:homeomorphism}
Let $G$ be a network with $\Kdt(G)\neq\emptyset$, and denote by $\mc{S}$ the stoichiometric subspace of $G$. Suppose that the kinetic subspace coincides with the stoichiometric subspace for all rate constants.
Then, for any $\xx_0\in\rrpp^n$, the following map is a homeomorphism:
\begin{align*}
\begin{split}
\Psi\colon (\xx_0 + \mc{S})_{>0} \times \Fdt(G) &\to \Kdt(G), \\
\Psi(\vv x, \vv\beta) &= (\beta_{\vv y \to \vv y'} \xx^{-\yy})_{\yy \rightarrow \yy' \in E}.
\end{split}
\end{align*}
In particular, $\Kdt(G)$ is a contractible manifold with boundary of dimension $\dim \mc{S}+\dim\Fdt(G)$.
\end{theoremalphabetic}

A different homeomorphism involving $\Fdt(G)$ and $\Kdt(G)$ has been discussed in \cites{craciun:deshpande:jin:2024a,craciun:deshpande:jin:2025b} and has been used to calculate the dimension of $\Kdt(G)$.
The simpler homeomorphism in \Cref{thmA:homeomorphism} provides both a theoretical framework for analyzing $\Kdt(G)$ through the polyhedral structure of $\Fdt(G)$, and a new computational method that splits a hard quantifier elimination problem into a linear programming part and a new, simpler quantifier elimination. We explore both of these perspectives in this work. 

\begin{example}
\label{ex:running_example}
As a running example throughout the paper, we will consider the \emph{partly reversible square} $(G,\vv\kk)$ shown below, where we have also displayed the mass-action differential equation associated with $(G,\vv\kk)$, as well as the fluxes (in magenta).
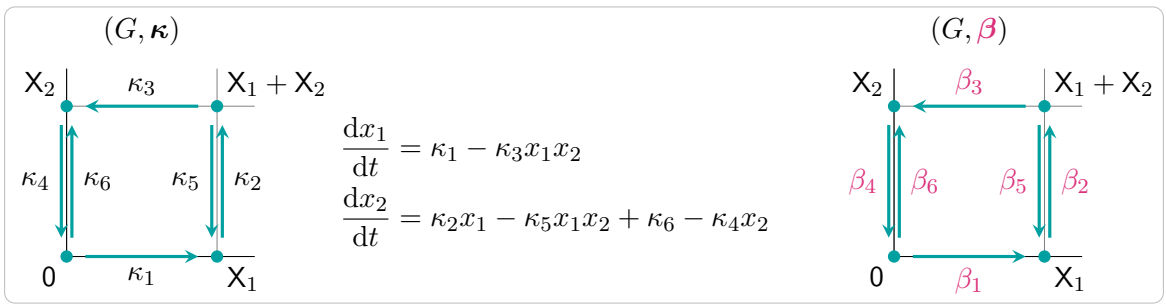
\begin{align*}\label{tikz:running_intro}
\begin{aligned}
\begin{tikzpicture}[scale=2]

\tikzset{mybullet/.style={inner sep=1.5pt,outer sep=5pt,draw,fill,Egraphcolor,circle}};
\tikzset{myarrow/.style={arrows={-stealth},very thick,Egraphcolor}};

\draw [step=1, gray, very thin] (0,0) grid (1.25,1.25);
\draw [ -, black] (0,0)--(1.25,0);
\draw [ -, black] (0,0)--(0,1.25);

\node[mybullet]  (0) at (0,0) {};
\node[mybullet]  (X) at (1,0) {};
\node[mybullet] (XY) at (1,1) {};
\node[mybullet]  (Y) at (0,1) {};

\node[below left]  at  (0) {$\0$};
\node[below right] at  (X) {$\X_1$};
\node[above right] at (XY) {$\X_1 + \X_2$};
\node[above left]  at  (Y) {$\X_2$};

\draw[myarrow]                                 (0) to node[below] {\textcolor{black}{$\kappa_1$}} (X);
\draw[myarrow,transform canvas={xshift=2pt}]   (X) to node[right] {\textcolor{black}{$\kappa_2$}} (XY);
\draw[myarrow,transform canvas={xshift=-2pt}] (XY) to node[left]  {\textcolor{black}{$\kappa_5$}} (X);
\draw[myarrow]                                (XY) to node[above] {\textcolor{black}{$\kappa_3$}} (Y);
\draw[myarrow,transform canvas={xshift=-2pt}]  (Y) to node[left]  {\textcolor{black}{$\kappa_4$}} (0);
\draw[myarrow,transform canvas={xshift=2pt}]   (0) to node[right] {\textcolor{black}{$\kappa_6$}} (Y);

\node at (0.5,1.5) {$(G,\vv\kappa)$};

\node[right] at (1.75,0.5) {$\begin{aligned}
\frac{\mathrm{d}x_1}{\mathrm{d}t} &= \kappa_1 - \kappa_3 x_1x_2 \\
\frac{\mathrm{d}x_2}{\mathrm{d}t} &= \kappa_2 x_1 - \kappa_5 x_1x_2  + \kappa_6 - \kappa_4 x_2 
\end{aligned}$};

\begin{scope}[shift={(5.5,0)}]

\draw [step=1, gray, very thin] (0,0) grid (1.25,1.25);
\draw [ -, black] (0,0)--(1.25,0);
\draw [ -, black] (0,0)--(0,1.25);

\node[mybullet]  (0) at (0,0) {};
\node[mybullet]  (X) at (1,0) {};
\node[mybullet] (XY) at (1,1) {};
\node[mybullet]  (Y) at (0,1) {};

\node[below left]  at  (0) {$\0$};
\node[below right] at  (X) {$\X_1$};
\node[above right] at (XY) {$\X_1 + \X_2$};
\node[above left]  at  (Y) {$\X_2$};

\draw[myarrow]                                 (0) to node[below] {\fluxcolor{$\beta_1$}} (X);
\draw[myarrow,transform canvas={xshift=2pt}]   (X) to node[right] {\fluxcolor{$\beta_2$}} (XY);
\draw[myarrow,transform canvas={xshift=-2pt}] (XY) to node[left]  {\fluxcolor{$\beta_5$}} (X);
\draw[myarrow]                                (XY) to node[above] {\fluxcolor{$\beta_3$}} (Y);
\draw[myarrow,transform canvas={xshift=-2pt}]  (Y) to node[left]  {\fluxcolor{$\beta_4$}} (0);
\draw[myarrow,transform canvas={xshift=2pt}]   (0) to node[right] {\fluxcolor{$\beta_6$}} (Y);

\node at (0.5,1.5) {$(G,\fluxcolor{\vv\beta})$};
    
\end{scope}

\node at (-0.3,0.5) {\phantom{0}};
\draw[rounded corners, lightgray] (current bounding box.south west) rectangle (current bounding box.north east);

\end{tikzpicture}
\end{aligned}
\end{align*}
The connection between the rate constant vector $\vv\kk$ and the flux vector $\vv\beta$ is as follows:
\begin{align*}
\beta_1 = \kk_1, \quad \beta_2 = \kk_2 x_1, \quad \beta_3 = \kk_3 x_1 x_2, \quad \beta_4 = \kk_4 x_2, \quad \beta_5 = \kk_5 x_1 x_2, \quad \beta_6 = \kk_6.
\end{align*}
We find that the toric flux cone $\Ft(G)$, the disguised toric flux cone $\Fdt(G)$, the toric locus $\Kt(G)$, and the disguised toric locus $\Kdt(G)$ are given by
\begin{align*}
\Ft(G) &= \{\vv\beta \in \rrpp^6 \st \beta_1 = \beta_3,\;\; \beta_2 + \beta_6 = \beta_4 + \beta_5,\;\; \beta_4 - \beta_6 = \beta_1 \},\\
\Fdt(G) &= \{\vv\beta \in \rrpp^6 \st \beta_1 = \beta_3,\;\; \beta_2 + \beta_6 = \beta_4 + \beta_5,\;\; |\beta_4 - \beta_6| \leq \beta_1 \leq \beta_4 + \beta_5 \},\\
\Kt(G) &= \{\vv\kk \in \rrpp^6 \st \tfrac{\kk_2\kk_4}{\kk_1\kk_3} = (1+\tfrac{\kk_6}{\kk_1}) (1+\tfrac{\kk_5}{\kk_3})\}, \\
\Kdt(G) &= \{\vv\kk \in \rrpp^6 \st (1-\tfrac{\kk_6}{\kk_1}) (1-\tfrac{\kk_5}{\kk_3}) \leq \tfrac{\kk_2\kk_4}{\kk_1\kk_3} \leq (1+\tfrac{\kk_6}{\kk_1}) (1+\tfrac{\kk_5}{\kk_3})\};
\end{align*}
see \Cref{fig:running_plot} for an illustration of (some slices of) $\Ft(G)$, $\Fdt(G)$, $\Kt(G)$ and $\Kdt(G)$.
Notice that the cones $\Ft(G)$ and $\Fdt(G)$ have dimension three and four, respectively. Similarly, the semialgebraic sets $\Kt(G)$ and $\Kdt(G)$ have dimension five and six, respectively. 
The homeomorphism $\rrpp^2 \times \Fdt(G)\to\Kdt(G)$ from \Cref{thmA:homeomorphism} is given by
\[(\xx,\vv\beta) \mapsto (\beta_1, \beta_2 \tfrac{1}{x_1}, \beta_3 \tfrac{1}{x_1 x_2}, \beta_4 \tfrac{1}{x_2}, \beta_5 \tfrac{1}{x_1 x_2}, \beta_6).\qedhere
\]

\begin{figure}[ht]
    \centering
    \input{tikz/1_running_plot_with_cone}
    \caption{Some slices of the toric flux cone $\Ft(G)$ and the disguised toric flux cone $\Fdt(G)$ (left), as well as the toric locus $\Kt(G)$ and the disguised toric locus $\Kdt(G)$ (right) for the  partly reversible square $G$ from \Cref{ex:running_example}.  
    }
\label{fig:running_plot}
\end{figure}
\end{example}

In principle, it might be that for $\vv\kk, \vv\kk' \in \Kdt(G)$ one must use different graphs (say, $H$ and $H'$) to display a vertex-balanced realization of $(G,\vv\kk)$ and $(G,\vv\kk')$. To capture this phenomenon, we use the notation $\Kdt(G,H)$ for the set of rate constants for which $G$ admits a vertex-balanced realization of $G$ with respect to $H$. It was proven in \cite{craciun:jin:yu:2020}*{Theorem~4.8} that it suffices to consider graphs $H$ that are subgraphs of the complete graph $G^{\comp}$ on the source vertices of $G$, so that
\begin{equation}
\label{eq:G_comp_enough_intro}
    \Kdt(G)=\bigcup_{H\subseteq G^{\comp}}\Kdt(G,H).
\end{equation}
In this paper, we sharpen \eqref{eq:G_comp_enough_intro} by constructing a subgraph of $G^{\comp}$ with this property, which we call the \emph{maximal weakly reversible realization graph} $G^{\max}$. We refer to \Cref{sec:disg_flux} for a precise definition, and \Cref{sec:algorithm} for a discussion on how it can be efficiently computed by solving a linear feasibility problem. Additionally, we show that considering only the realizations with respect to $G^{\max}$ is enough to capture the interior of $\Kdt(G)$. 

\begin{theoremalphabetic}
[\Cref{thm:G_max_enough_K} and \Cref{cor:Kdt_G}]
\label{thmB:Gmax}
Let $G$ be a network, and let $G^{\max}$ be the associated maximal weakly reversible realization graph. Then the following statements hold.
\begin{enumerate}[label={\rm(\alph*)}]
    \item The disguised toric locus is given by $\Kdt(G)=\bigcup_{H\subseteq G^{\max}} \Kdt(G,H)$.
    \item If the kinetic subspace coincides with the stoichiometric subspace for all rate constants, then the manifold interior of $\Kdt(G)$ is $\Kdt(G,G^{\max})$.
\end{enumerate}
\end{theoremalphabetic}

The key to proving part (b) is to first establish the analogous result for the flux cone $\Fdt(G)$, which states that $\Fdt(G,G^{\max})$ is the relative interior of $\Fdt(G)$ (see \Cref{thm:FdtG_is_clFdtGGmax}), and then apply the homeomorphism in \Cref{thmA:homeomorphism} to transfer the result to $\Kdt(G)$.

\begin{example}
\label{ex:running_example_Gmax}
For the network $G$ in \Cref{ex:running_example}, the graphs $G^{\comp}$ and $G^{\max}$ are as follows:
\begin{align*}
\begin{aligned}
\begin{tikzpicture}[scale=2]

\tikzset{mybullet/.style={inner sep=1.5pt,outer sep=5pt,draw,fill,Egraphcolor,circle}};
\tikzset{myarrow/.style={arrows={-stealth},very thick,Egraphcolor}};

\draw [step=1, gray, very thin] (0,0) grid (1.25,1.25);
\draw [ -, black] (0,0)--(1.25,0);
\draw [ -, black] (0,0)--(0,1.25);

\node[mybullet]  (0) at (0,0) {};
\node[mybullet]  (X) at (1,0) {};
\node[mybullet] (XY) at (1,1) {};
\node[mybullet]  (Y) at (0,1) {};

\draw[myarrow]                                 (0) to node {} (X);
\draw[myarrow,transform canvas={xshift=2pt}]   (X) to node {} (XY);
\draw[myarrow,transform canvas={xshift=-2pt}] (XY) to node {} (X);
\draw[myarrow]                                (XY) to node {} (Y);
\draw[myarrow,transform canvas={xshift=-2pt}]  (Y) to node {} (0);
\draw[myarrow,transform canvas={xshift=2pt}]   (0) to node {} (Y);

\node at (0.5,1.5) {$G$};
\node at (-0.5,0) {\phantom{0}};

\begin{scope}[shift={(2,0)}]

\draw [step=1, gray, very thin] (0,0) grid (1.25,1.25);
\draw [ -, black] (0,0)--(1.25,0);
\draw [ -, black] (0,0)--(0,1.25);

\node[mybullet]  (0) at (0,0) {};
\node[mybullet]  (X) at (1,0) {};
\node[mybullet] (XY) at (1,1) {};
\node[mybullet]  (Y) at (0,1) {};

\draw[myarrow,transform canvas={yshift=-2pt}]  (0) to node {} (X);
\draw[myarrow,transform canvas={yshift= 2pt}]  (X) to node {} (0);
\draw[myarrow,transform canvas={xshift=2pt}]   (X) to node {} (XY);
\draw[myarrow,transform canvas={xshift=-2pt}] (XY) to node {}  (X);
\draw[myarrow,transform canvas={yshift= 2pt}]  (XY) to node {} (Y);
\draw[myarrow,transform canvas={yshift=-2pt}]  (Y) to node {} (XY);
\draw[myarrow,transform canvas={xshift=-2pt}]  (Y) to node {} (0);
\draw[myarrow,transform canvas={xshift=2pt}]   (0) to node {} (Y);
\draw[myarrow,transform canvas={xshift=1.4pt,yshift=-1.4pt}]  (0) to node {} (XY);
\draw[myarrow,transform canvas={xshift=-1.4pt,yshift=1.4pt}] (XY) to node {} (0);
\draw[myarrow,transform canvas={xshift=1.4pt,yshift=1.4pt}]  (X) to node {} (Y);
\draw[myarrow,transform canvas={xshift=-1.4pt,yshift=-1.4pt}]  (Y) to node {} (X);

\node at (0.5,1.5) {$G^{\comp}$};

\end{scope}

\begin{scope}[shift={(4,0)}]

\draw [step=1, gray, very thin] (0,0) grid (1.25,1.25);
\draw [ -, black] (0,0)--(1.25,0);
\draw [ -, black] (0,0)--(0,1.25);

\node[mybullet]  (0) at (0,0) {};
\node[mybullet]  (X) at (1,0) {};
\node[mybullet] (XY) at (1,1) {};
\node[mybullet]  (Y) at (0,1) {};

\draw[myarrow]  (0) to node {} (X);
\draw[myarrow,transform canvas={xshift=2pt}]   (X) to node {} (XY);
\draw[myarrow,transform canvas={xshift=-2pt}] (XY) to node {}  (X);
\draw[myarrow]  (XY) to node {} (Y);
\draw[myarrow,transform canvas={xshift=-2pt}]  (Y) to node {} (0);
\draw[myarrow,transform canvas={xshift=2pt}]   (0) to node {} (Y);
\draw[myarrow,transform canvas={xshift=1.4pt,yshift=-1.4pt}]  (0) to node {} (XY);
\draw[myarrow,transform canvas={xshift=-1.4pt,yshift=1.4pt}] (XY) to node {} (0);

\node at (0.5,1.5) {$G^{\max}$};
\node at (1.5,-0.25) {\phantom{0}};
\end{scope}

\draw[rounded corners, lightgray] (current bounding box.south west) rectangle (current bounding box.north east);

\end{tikzpicture}
\end{aligned}
\end{align*}
In accordance with \Cref{thmB:Gmax}(b) (and the analogous statement for the fluxes, see \Cref{thm:FdtG_is_clFdtGGmax}), it holds that
\begin{align*}
\Fdt(G,G^{\max}) &= \{\vv\beta \in \rrpp^6 \st \beta_1 = \beta_3,\;\; \beta_2 + \beta_6 = \beta_4 + \beta_5,\;\; |\beta_4 - \beta_6| < \beta_1 < \beta_4 + \beta_5 \},\\
\Kdt(G,G^{\max}) &= \{\vv\kk \in \rrpp^6 \st (1-\tfrac{\kk_6}{\kk_1}) (1-\tfrac{\kk_5}{\kk_3}) < \tfrac{\kk_2\kk_4}{\kk_1\kk_3} < (1+\tfrac{\kk_6}{\kk_1}) (1+\tfrac{\kk_5}{\kk_3})\}.
\end{align*}
In \Cref{fig:running_plot}, the interiors of the blue regions are slices of $\Fdt(G,G^{\max})$ and $\Kdt(G,G^{\max})$.
\qedhere
\end{example}

By combining the linearization of the problem provided by \Cref{thmA:homeomorphism}, with the fact that $G^{\max}$ can be used to find the whole interior of $\Kdt(G)$, we outline in \Cref{sec:algorithm} a new three-step strategy for computing $\Kdt(G)$: first find $G^{\max}$, then compute $\Fdt(G, G^{\max})$ (and thus, $\Fdt(G)$), and finally obtain $\Kdt(G)$ through quantifier elimination. In \Cref{sec:examples}, we use this strategy to compute the disguised toric locus for several networks that have previously appeared in the literature, and which would have been out of reach with the algorithm from \cite{brustenga:craciun:sorea:2022}*{Section~8}.
Our examples include the reversible Lotka--Volterra autocatalator \cite{simon:1992}, a network with Bogdanov--Takens bifurcation \cite{banaji:boros:hofbauer:2024b}, the basic clock mechanism \cites{johnston:pantea:donnell:2016,leloup:goldbeter:1999}, a tetrahedron network \cite{johnston:siegel:szederkenyi:2013:mathbiosci}, and a four-dimensional network \cite{craciun:pantea:2008}. These examples demonstrate that $\Kdt(G)$ often has positive measure, while  $\Kt(G)$ has measure zero. In particular, for the reversible Lotka--Volterra autocatalator, we obtain an immediate proof of global stability, which is much shorter than the highly intricate original proof in \cite{simon:1992}.

\subsection*{Structure of the paper}
The rest of this paper is organized as follows. In \Cref{sec:prelim}, we collect the necessary background about mass-action systems, vertex-balancing, and the disguised toric locus. In \Cref{sec:G_max}, we introduce the maximal weakly reversible realization graph. In \Cref{sec:disg_flux}, we define the disguised toric flux cones and prove some of their basic properties. In \Cref{sec:topology}, we study some topological aspects pertaining to the disguised toric locus. In \Cref{sec:algorithm}, we discuss how to utilize our findings to construct an efficient procedure to calculate the disguised toric locus. In \Cref{sec:examples}, we present several interesting examples in detail. Finally, in \Cref{sec:discussion} we make some concluding remarks.

\subsection*{Notation and conventions}
We use the notation $\xx^\yy = x_1^{y_1}\cdots x_n^{y_n}$ for $\xx=(x_1,\dots,x_n)\in\rrpp^n$ and $\yy=(y_1,\dots,y_n)\in\rr^n$. For a set $A\subseteq\rr^n$, we write $A_{>0}$ for the positive part $A\cap\rr^n_{>0}$. Unless stated otherwise, all subsets $A\subseteq\rr^n$ are considered with the Euclidean topology, and we write $\overline{A}$ for the closure in $\rr^n$. For graphs $H$ and $G$, we denote the relation of $H$ being a subgraph of $G$ as $H\subseteq G$.

A (closed) \emph{polyhedral cone} is a set of the form $\{\sum_{i=1}^m\alpha_i\vv{v}_i\mid \vv{\alpha}\in\rrp^m\}$, and an \emph{open polyhedral cone} is a set of the form $\{\sum_{i=1}^m\alpha_i\vv{v}_i\mid \vv{\alpha}\in\rrpp^m\}$ for vectors $\vv{v}_1,\ldots,\vv{v}_m\in\rr^n$. We refer to such cones as \emph{pointed} if they do not contain any lines. For a convex set $S\subseteq\rr^n$, we denote its \emph{relative interior} (i.e., the interior with respect to its affine hull) by $\relint S$. 

The notion of \emph{dimension} that is used throughout the paper is the dimension of \emph{semialgebraic sets} \cite{bochnak:coste:roy:1998}*{Section~2.8}, which at all nonsingular points agrees with the usual dimension of manifolds with boundary; an $n$-dimensional manifold with boundary is a topological space where every point has an open neighborhood homeomorphic to an open subset of the closed half-space $\{x\in\rr^n:x_n\geq 0\}$.
\section{Preliminaries}
\label{sec:prelim}

In this section, we introduce the basic objects and terminology of interest, and illustrate these in \Cref{ex:running} below. We start with Euclidean embedded graphs and mass-action systems. Recall from graph theory that a directed graph is said to be \emph{simple} if it has no multiple edges and has no self-loops.

\begin{definition}\label{def:E-graph}
A \textit{Euclidean embedded graph} (or \textit{E-graph} for short) is a finite simple directed graph $G=(V, E)$, where $V \subseteq \mathbb{Z}^n$ is the set of vertices, and $E \subseteq V \times V$ is the set of directed edges. Given an edge $(\yy,\yy')\in E$ we often write $\yy \to \yy'\in E$, and refer to  $\yy$ and $\yy'$ as the \textit{source vertex} and \textit{product vertex} of the edge $\yy \to \yy'$, respectively.
\end{definition}

\begin{definition}\label{def:mass-action}
Let $G=(V,E)$ be an E-graph and let $\vv{\kk}\in\rrpp^{|E|}$ be a labeling of the edges. The \emph{species-formation function} $\vv f_{(G,\vv\kk)}\colon \rrpp^n \to \rr^n$ is defined by
\begin{align}\label{eq:f_(G,k)}
\vv f_{(G,\vv\kk)} (\xx) = \sum_{\yy\to\yy' \in E} \kk_{\yy\to\yy'} \xx^\yy (\yy' - \yy).
\end{align} 
The positive real number $\kk_{\yy\to\yy'}$ is called the \textit{rate constant} corresponding to the reaction $\yy\to\yy'$. The \textit{mass-action system generated by $(G, \vv\kk)$} is the  dynamical system on $\rrpp^{n}$ given by
\begin{align}\label{eq:dxdt_f_(G,k)}
\frac{\mathrm{d}\xx}{\mathrm{d}t} = \vv f_{(G,\vv\kk)}(\xx).
\end{align}
\end{definition}

\begin{remark}
When $V \subseteq \mathbb{Z}_{\geq 0}^n$, \Cref{def:mass-action} corresponds to \emph{classical} mass-action systems. The interpretation of an edge $\yy\to\yy'\in E$ is that the linear combination $\sum_{i=1}^n y_i \X_i$ of the species $\{\X_1,\ldots,\X_n\}$ is transformed to $\sum_{i=1}^n y'_i \X_i$. In this case, the positive orthant $\rrpp^n$ is forward invariant under the mass-action differential equation \eqref{eq:dxdt_f_(G,k)}; no solution can reach the boundary of $\rrpp^n$ in finite time. Further, $\vv{f}_{(G,\vv\kk)}$ in \eqref{eq:f_(G,k)} is a polynomial, which is the case in most practical applications.
\end{remark}

Next, we define the stoichiometric subspace and the positive stoichiometric classes of an E-graph; the discussion following the definition illuminates their relevance. 
\begin{definition}
For an E-graph $G = (V, E)$ in $\zz^n$, the linear subspace $\mc{S}$ of $\rr^n$, defined by $\mc{S}=\spn\{\yy'-\yy \st \yy \to \yy' \in E\}$ is called the \emph{stoichiometric subspace} of $G$. For an $\xx_0 \in \rrpp^n$, the \emph{positive stoichiometric class} through $\xx_0$ is the linear manifold $(\xx_0 + \mc{S})_{>0}$, i.e., the positive part of the coset $\xx_0 + \mc{S}$.
\end{definition}

Note that the solutions of \eqref{eq:dxdt_f_(G,k)} are defined uniquely as long as they are in $\rrpp^n$; and a solution with initial condition $\xx(0) = \xx_0$ remains in the positive stoichiometric class $(\xx_0 + \mc{S})_{>0}$ as long as it exists. In fact, the solutions are confined to a potentially strictly smaller linear submanifold of their positive stoichiometric class.

\begin{definition}
For an E-graph $G = (V, E)$ in $\zz^n$ and a $\vv\kk \in \rrpp^{|E|}$, the linear subspace $\mc{S}^{\vv\kk}$ of $\rr^n$, defined by $\mc{S}^{\vv\kk}=\spn\{\vv{f}_{(G,\vv\kk)}(\xx) \st \xx \in \rrpp^n\}$ is called the \emph{kinetic subspace} of $(G,\vv\kk)$.
\end{definition}

In general, for all $\vv\kk$, the kinetic subspace $\mc{S}^{\kk}$ is a subspace of the stoichiometric subspace $\mc{S}$; and the solution with initial condition $\xx(0) = \xx_0 \in \rrpp^n$ is confined to $(\xx_0 + \mc{S}^{\vv\kk})_{>0}$. In \Cref{subsec:Kdt_G}, we assume that the E-graph in question satisfies $\mc{S}^{\kk} = \mc{S}$ for all $\vv\kk$. Feinberg and Horn identified a large class of E-graphs for which this property holds, see \cite{feinberg:horn:1977}*{Section~6}. Before we state their result, we recall the notion of weak reversibility.

\begin{definition}
An E-graph $G=(V, E)$ is \emph{weakly reversible} if for any $\yy, \yy' \in V$, there exists a directed path from $\yy$ to $\yy'$ if and only if there exists a directed path from $\yy'$ to $\yy$.
\end{definition}

\begin{theorem} \label{thm:FH1977}
Let $G = (V, E)$ be an E-graph for which there exists a directed path between any two vertices that are in the same connected component. Then $\mc{S}^{\kk} = \mc{S}$ for all $\vv\kk \in \rrpp^{|E|}$. In particular, the conclusion holds if $G$ is weakly reversible.
\end{theorem}

Vertex balancing plays a central role in this paper.
\begin{definition}\label{def:vertex-balanced}
Given an E-graph $G=(V,E)$ and $\vv\kk\in\rrpp^{|E|}$, the pair $(G,\vv\kappa)$ is {\em vertex-balanced} if there exists an $\vv x^* \in \rrpp^n$ satisfying the equation
\begin{align*}
\sum_{\yy \to \yy_0\in E}\kappa_{\yy \to \yy_0} (\xx^*)^{\yy} = \sum_{\yy_0 \to \yy'\in E}\kappa_{\yy_0 \to \yy'} (\xx^*)^{\yy_0}
\end{align*}
for every vertex $\yy_0\in V$. When such an $\xx^*$ exists, the mass-action system generated by $(G,\vv\kappa)$ is called a \textit{toric dynamical system}, and $\xx^*$ is called a {\em vertex-balanced equilibrium} (or {\em complex-balanced equilibrium} in the classical theory of mass-action systems).
\end{definition}

It is not hard to see that the pair $(G,\vv\kappa)$ can be vertex-balanced only if $G$ is weakly reversible.

The mass-action system generated by a vertex-balanced pair $(G,\vv\kappa)$ displays remarkably well-behaved dynamical properties, which are summarized in the following theorem. Statements (a) and (b) are due to Horn and Jackson \cite{horn:jackson:1972}. For statement (c), see \cite{johnston:2011}*{Theorem 4.3.4}, \cite{feinberg:2019}*{Theorem 15.2.2 (iii)}, or \cite{boros:muller:regensburger:2020}*{Theorem 8}.
\begin{theorem} \label{thm:nice_dynamics}
Let $(G,\vv\kappa)$ be vertex-balanced, and fix a vertex-balanced equilibrium $\xx^*$. Then the mass-action system generated by $(G,\vv\kappa)$ has the following properties.
\begin{enumerate}[label={\rm(\alph*)}]
\item The set of positive equilibria is given by
\begin{align} \label{eq:qts}
\mathcal{E} = \{\xx \in \rrpp^n \st \log \xx - \log \xx^* \in \mc{S}^\perp\},
\end{align}
every positive stoichiometric class has exactly one positive equilibrium, and every positive equilibrium is vertex-balanced.
\item The Horn--Jackson function $L \colon \rrpp^n \to \rr$, defined by
\begin{align} \label{eq:H-J}
L(\vv x) = \sum_{i=1}^n x_i (\log \tfrac{x_i}{x_i^*} - 1),
\end{align}
satisfies $\langle (\grad L)(\xx), \vv{f}_{(G,\vv\kk)}(\xx) \rangle \leq 0$ for all $\xx \in \rrpp^n$, with equality if and only if $\xx \in \mathcal{E}$. In particular, $L$ is a strict Lyapunov function in every positive stoichiometric class, and there is no periodic solution in $\rrpp^n$.
\item The equilibrium $\xx^*$ is linearly stable relative to its positive stoichiometric class.
\end{enumerate}
\end{theorem}

Asymptotic stability of the vertex-balanced equilibrium relative to its positive stoichiometric class follows from (b). The linear stability stated in (c) is a stronger property. Denoting by $J(\xx^*) \in \rr^{n \times n}$ the Jacobian matrix of a mass-action system, evaluated at a vertex-balanced equilibrium $\xx^* \in \rrpp^n$, each eigenvalue of the linear transformation $J(\xx^*)|_\mc{S}\colon \mc{S} \to \mc{S}$ has a negative real part. In other words, restricting the dynamics to its positive stoichiometric class, $(\xx^*+\mc{S})_{>0}$, we find that $\xx^*$ is linearly stable. This fact plays an important role in the proofs of \Cref{thm:homeo_Psi_H,thm:homeo_Psi} in \Cref{sec:topology} below.

Besides the properties listed in \Cref{thm:nice_dynamics}, we can often also conclude \emph{global} asymptotic stability of a vertex-balanced equilibrium relative to its positive stoichiometric class (when some other tool lets us exclude the possibility of a solution approaching the boundary of $\rrpp^n$). For instance, this includes strongly connected networks
\cites{anderson:2011a,gopalkrishnan:miller:shiu:2014,boros:hofbauer:2020}.

We remark that (4) implies that the set of positive equilibria is a \emph{positive toric variety}, in the sense that it admits a monomial parametrization, and refer to \cite{feliu:henriksson:2024} for a discussion on the algebraic and geometric consequences of this. Moreover, mass-action systems whose set of positive equilibria is as in \eqref{eq:qts}, and for which property (b) in \Cref{thm:nice_dynamics} holds, are called \emph{quasi-thermodynamic} \cite{horn:jackson:1972}. Hence, vertex-balanced mass-action systems are quasi-thermodynamic.

The following four definitions are vital for the rest of this paper. We define the toric locus, dynamical equality, disguised vertex-balanced pairs, and the disguised toric locus.

\begin{definition}\label{def:KtG}
For an E-graph $G=(V,E)$, define the \textit{toric locus} of G as the set
\begin{align*}
\Kt(G) &= \{\vv\kk\in \rrpp^{|E|} \st (G,\vv\kk) \text{ is vertex-balanced}\}.
\end{align*}
\end{definition}

\begin{definition} \label{def:dyn_equil_rateconstant}
Let $G=(V,E)$ and $H=(V_H,E_H)$ be E-graphs and let $\vv\kk\in\rrpp^{|E|}$ and  $\vv\lambda\in\rrpp^{|E_H|}$. The pairs $(G,\vv\kk)$ and $(H,\vv\lambda)$ are said to be \emph{dynamically equal}, denoted $(G,\vv\kk) \triangleq (H,\vv\lambda)$, if $$\vv{f}_{(G,\vv\kk)}(\xx)=\vv{f}_{(H,\vv\lambda)}(\xx)\quad \text{for all $\xx\in\rrpp^n$}.$$
\end{definition}

\begin{definition}
Let $G=(V,E)$ be an E-graph, and let $\vv\kk\in\rrpp^{|E|}$. We say that $(G,\vv\kk)$ is \emph{disguised vertex-balanced} if there exist an E-graph $H$ and a $\vv\lambda \in \Kt(H)$ such that $(G,\vv\kappa) \triangleq (H,\vv\lambda)$.
\end{definition}

\begin{definition} \label{def:KdtG}
For an E-graph $G=(V,E)$, define the \textit{disguised toric locus} of G as the set
\begin{align*}
\Kdt(G)=\{\vv\kk\in\rrpp^{|E|}\st \text{$(G,\vv\kk)$ is disguised vertex-balanced}\}.
\end{align*}
Furthermore, we define the \emph{disguised toric locus of $G$ with respect to $H$} for a fixed E-graph $H$ to be
\begin{align*}
\Kdt(G,H)=\{\vv\kk\in\rrpp^{|E|}\st \text{$(G,\vv\kk)\triangleq (H,\vv\lambda)$ for some $\vv\lambda\in\Kt(H)$}\}.
\end{align*}
\end{definition}

The investigation of the disguised toric locus $\Kdt(G)$ is motivated by the fact that in many examples, $\Kdt(G)$ is significantly larger than $\Kt(G)$, and for every $\vv\kappa \in \Kdt(G)$, the mass-action system generated by $(G,\vv\kk)$ enjoys remarkable dynamical properties, as $(G,\vv\kappa)$ is dynamically equal to a vertex-balanced pair $(H,\vv\lambda)$. For the precise statement, see \Cref{thm:nice_dynamics_disg} below, which is an immediate consequence of \Cref{thm:nice_dynamics}. We remark that in most practical applications, the stoichiometric subspaces of $G$ and $H$ coincide.
In particular, this is true whenever $G$ satisfies the assumptions of \Cref{thm:FH1977} (which includes all the examples in \Cref{sec:examples}).

\begin{theorem} \label{thm:nice_dynamics_disg}
Let $(G,\vv\kappa)$ be disguised vertex-balanced, and let $\mc{S}_H$ denote the stoichiometric subspace of $H$, where the E-graph $H$ is such that $\vv\kk \in \Kdt(G,H)$. Then there exists an $\xx^* \in \rrpp^n$ such that $\vv{f}_{(G,\vv\kk)}(\xx^*)=\vv{0}$, and the mass-action system generated by $(G,\vv\kappa)$ has the following properties for any such fixed $\xx^*$.
\begin{enumerate}[label={\rm(\alph*)}]
\item The set of positive equilibria is given by $\mathcal{E} = \{\xx \in \rrpp^n \st \log \xx - \log \xx^* \in \mc{S}_H^\perp\}$, and the set $(\xx_0+\mc{S}_H)_{>0}$ has exactly one positive equilibrium for every $\xx_0 \in \rrpp^n$.
\item The Horn--Jackson function $L \colon \rrpp^n \to \rr$, defined in \eqref{eq:H-J}, satisfies $\langle (\grad L)(\xx), \vv{f}_{(G,\vv\kk)}(\xx) \rangle \leq 0$ for all $\xx \in \rrpp^n$, with equality if and only if $\xx \in \mathcal{E}$. In particular, $L$ is a strict Lyapunov function in the set $(\xx_0+\mc{S}_H)_{>0}$ for every $\xx_0 \in \rrpp^n$, and there is no periodic solution in $\rrpp^n$.
\item The equilibrium $\xx^*$ is linearly stable relative to the set $(\xx^*+\mc{S}_H)_{>0}$.
\end{enumerate}
\end{theorem}

Further, as for $\vv\kk \in \Kt(G)$, we can sometimes conclude not only local but even global asymptotic stability of a positive equilibrium of $(G,\vv\kk)$, where $\vv\kk \in \Kdt(G)$ \cites{anderson:2011a,gopalkrishnan:miller:shiu:2014,boros:hofbauer:2020}.

We remark that if for a dynamical system $\frac{\mathrm{d}\xx}{\mathrm{d}t} = \vv{g}(\xx)$ on $\rrpp^n$ there exist an E-graph $H$ (in $\zz^n$) and a $\vv\lambda \in \Kt(H)$ such that $\vv{g} = \vv{f}_{(H,\vv\lambda)}$ on $\rrpp^n$ then the dynamical system $\frac{\mathrm{d}\xx}{\mathrm{d}t} = \vv{g}(\xx)$ enjoys all the remarkable dynamical properties that are listed in \Cref{thm:nice_dynamics_disg}. In particular, it is quasi-thermodynamic. The algorithmic aspects of this approach were studied, for example, in \cite{szederkenyi:hangos:2011}.

It has been shown in \cite{brustenga:craciun:sorea:2022} that $\Kdt(G, H)$ and $\Kdt(G)$ are semialgebraic sets, in the sense that each of them is the union of the solution sets of finitely many systems of polynomial equations and inequalities. The key observation is that the condition on $\vv\kk\in\rrpp^{|E|}$ in the definition of $\Kdt(G, H)$ can be formulated in terms of existential quantifiers and polynomial relations that encode dynamic equality and vertex balancing, and $\Kdt(G)$ is a finite union of such sets.
\begin{lemma}\label{lem:semialgebraic_K}
Let $G=(V,E)$ and $H=(V_H,E_H)$ be E-graphs in $\zz^n$, and $\vv{\kk} \in \rrpp^{|E|}$. Then $\vv\kk\in \Kdt(G,H)$ if and only if there exist $\vv\lambda\in\rrpp^{|E_H|}$ and $\xx\in\rrpp^n$ that satisfy the following dynamical equality and vertex-balancing conditions:
\begin{align}
\label{eq:DE}\tag{DE}
\sum_{\vv y_0\to \vv y' \in E} \kk_{\vv y_0\to \vv y'}(\vv y'-\vv y_0)&=\sum_{\vv y_0 \to \vv y'\in E_H} \lambda_{\vv y_0\to \vv y'} (\vv y'-\vv y_0)\:\:\:\text{for every $\vv y_0 \in V\cup V_H$};\\
\label{eq:VB}\tag{VB}
\sum_{\yy \to \yy_0 \in E_H} \lambda_{\yy \to \yy_0} \xx^\yy &= \sum_{\yy_0 \to \yy' \in E_H} \lambda_{\yy_0 \to \yy'} \xx^{\yy_0}\:\:\:\text{for every $\yy_0 \in V_H$}.
\end{align}
\end{lemma}

Next, we introduce the equilibrium locus of an E-graph.
\begin{definition}
For an E-graph $G=(V,E)$ in $\zz^n$, define the \textit{equilibrium locus} of $G$ as the set
\begin{align*}
\Keq(G) &= \{\vv\kk\in \rrpp^{|E|} \st \text{there exists an } \xx \in \rrpp^n \text{ such that } \vv{f}_{(G,\vv\kk)}(\xx) = \vv{0}\}.
\end{align*}
\end{definition}

Provided the E-graph $G=(V, E)$ is weakly reversible, we have $\Kt(G) \neq \emptyset$ (see, e.g., \cite{feinberg:horn:1977}*{Appendix}), and it has been shown in \cite{boros:2019} that $\Keq(G) = \rrpp^{|E|}$. It was proven in \cite{feliu:henriksson:pascual:2024} that $\dim\Keq(G)=|E|$ for any network that has a nondegenerate equilibrium.

In general, for any E-graph $G$ we have the inclusions
\begin{align} \label{eq:chain_K}
\Kt(G) \subseteq\Kdt(G)\subseteq\Keq(G)\subseteq\rrpp^{|E|}.
\end{align}

The following example illustrates the notions introduced in this section.

\begin{example}
\label{ex:running}
We revisit the \emph{partially reversible square} from \Cref{ex:running_example}, which is the E-graph $G = (V, E)$ displayed in \eqref{tikz:running}. It is embedded in $\rr^2$, it has 4 vertices (namely, $V=\{(0,0), (1,0), (1,1), (0,1)\}$) and 6 directed edges. All vertices are both source and target vertices. Notice that $G$ is weakly reversible (and hence, $\Keq(G)=\rrpp^6$), its stoichiometric subspace is $\mc{S} = \rr^2$, and there is only one positive stoichiometric class, namely, the positive quadrant $\rrpp^2$. Since $G$ is weakly reversible, the kinetic subspace $\mc{S}^{\vv\kk}$ equals $\mc{S}$ for all $\vv\kk\in \rrpp^6$.

\begin{align}\label{tikz:running}
\begin{aligned}
\begin{tikzpicture}[scale=2]

\tikzset{mybullet/.style={inner sep=1.5pt,outer sep=5pt,draw,fill,Egraphcolor,circle}};
\tikzset{myarrow/.style={arrows={-stealth},very thick,Egraphcolor}};

\draw [step=1, gray, very thin] (0,0) grid (1.25,1.25);
\draw [ -, black] (0,0)--(1.25,0);
\draw [ -, black] (0,0)--(0,1.25);

\node[mybullet]  (0) at (0,0) {};
\node[mybullet]  (X) at (1,0) {};
\node[mybullet] (XY) at (1,1) {};
\node[mybullet]  (Y) at (0,1) {};

\draw[myarrow]                                 (0) to node[below] {\textcolor{black}{$\kappa_1$}} (X);
\draw[myarrow,transform canvas={xshift=2pt}]   (X) to node[right] {\textcolor{black}{$\kappa_2$}} (XY);
\draw[myarrow,transform canvas={xshift=-2pt}] (XY) to node[left]  {\textcolor{black}{$\kappa_5$}} (X);
\draw[myarrow]                                (XY) to node[above] {\textcolor{black}{$\kappa_3$}} (Y);
\draw[myarrow,transform canvas={xshift=-2pt}]  (Y) to node[left]  {\textcolor{black}{$\kappa_4$}} (0);
\draw[myarrow,transform canvas={xshift=2pt}]   (0) to node[right] {\textcolor{black}{$\kappa_6$}} (Y);

\node at (0.5,1.5) {$(G,\vv\kappa)$};

\node at (0.5,-.85) {$\begin{aligned}
\frac{\mathrm{d}x_1}{\mathrm{d}t} &= \kappa_1 - \kappa_3 x_1x_2 \\
\frac{\mathrm{d}x_2}{\mathrm{d}t} &= \kappa_2 x_1 - \kappa_5 x_1x_2  + \kappa_6 - \kappa_4 x_2 
\end{aligned}$};

\begin{scope}[shift={(3.5,0)}]

\draw [step=1, gray, very thin] (0,0) grid (1.25,1.25);
\draw [ -, black] (0,0)--(1.25,0);
\draw [ -, black] (0,0)--(0,1.25);

\node[mybullet]  (0) at (0,0) {};
\node[mybullet]  (X) at (1,0) {};
\node[mybullet] (XY) at (1,1) {};
\node[mybullet]  (Y) at (0,1) {};

\draw[myarrow]  (0) to node[below] {\textcolor{black}{$\lambda_1$}} (X);
\draw[myarrow]  (X) to node[right] {\textcolor{black}{$\lambda_2$}} (XY);
\draw[myarrow] (XY) to node[above] {\textcolor{black}{$\lambda_3$}} (Y);
\draw[myarrow]  (Y) to node[left]  {\textcolor{black}{$\lambda_4$}} (0);
\draw[myarrow,transform canvas={xshift=1.4pt,yshift=-1.4pt}] (XY) to node[xshift=7pt,yshift=-7pt]  {\textcolor{black}{$\lambda_5$}} (0);
\draw[myarrow,transform canvas={xshift=-1.4pt,yshift=1.4pt}]  (0) to node[xshift=-7pt,yshift=7pt] {\textcolor{black}{$\lambda_6$}} (XY);

\node at (0.5,1.5) {$(H,\vv\lambda)$};

\node at (0.5,-.85) {$\begin{aligned}
\frac{\mathrm{d}x_1}{\mathrm{d}t} &= (\lambda_1+\lambda_6) - (\lambda_3+\lambda_5) x_1x_2 \\
\frac{\mathrm{d}x_2}{\mathrm{d}t} &= \lambda_2 x_1 - \lambda_5 x_1x_2 + \lambda_6 - \lambda_4 x_2  
\end{aligned}$};
    
\end{scope}

\draw[rounded corners, lightgray] (current bounding box.south west) rectangle (current bounding box.north east);

\end{tikzpicture}
\end{aligned}
\end{align}
One finds, for example by the matrix-tree theorem \cite{craciun:dickenstein:shiu:sturmfels:2009}, that the toric locus of $G$ is the codi\-men\-sion-one semialgebraic set
\begin{align*}
    \Kt(G) = \{\vv \kappa \in \rrpp^6 \st \tfrac{\kappa_2\kappa_4}{\kappa_1\kappa_3} = (1+\tfrac{\kappa_6}{\kappa_1}) (1+\tfrac{\kappa_5}{\kappa_3})\}.
\end{align*}
By \Cref{lem:semialgebraic_K}, for the E-graph $H$ in \eqref{tikz:running}, a $\vv\kappa \in \rrpp^6$ is in $\Kdt(G,H)$ if and only if there exist $\vv\lambda \in \rrpp^6$ and $\xx \in \rrpp^2$ such that
\begin{gather}
\kappa_1 = \lambda_1 + \lambda_6, \qquad \kappa_2 = \lambda_2, \qquad \kappa_3 = \lambda_3 + \lambda_5, \qquad \kappa_4 = \lambda_4, \qquad \kappa_5 = \lambda_5, \qquad \kappa_6 = \lambda_6; \tag{DE} \\
\lambda_4 x_2 +\lambda_5 x_1 x_2= \lambda_1 + \lambda_6, \quad\;\; \lambda_1 = \lambda_2 x_1, \quad\;\; \lambda_2 x_1 + \lambda_6 = (\lambda_3 + \lambda_5)x_1 x_2, \quad\;\; \lambda_3 x_1 x_2 = \lambda_4 x_2. \tag{VB}
\end{gather}
Solving this nonlinear quantifier elimination problem, we find that
\begin{align*}
    \Kdt(G,H) &= \{\vv \kappa \in \rrpp^6 \st \tfrac{\kappa_6}{\kappa_1}<1, \tfrac{\kappa_5}{\kappa_3}<1, (1-\tfrac{\kappa_6}{\kappa_1}) (1-\tfrac{\kappa_5}{\kappa_3}) = \tfrac{\kappa_2\kappa_4}{\kappa_1\kappa_3}\},
\end{align*}
a codimension-one semialgebraic set that is disjoint from $\Kt(G)$. In \Cref{subsec:square_parallelogram}, we show that the disguised toric locus $\Kdt(G)$ is given by
\begin{align*}
    \Kdt(G) = \{\vv \kappa \in \rrpp^6 \st (1-\tfrac{\kappa_6}{\kappa_1}) (1-\tfrac{\kappa_5}{\kappa_3}) \leq \tfrac{\kappa_2\kappa_4}{\kappa_1\kappa_3} \leq (1+\tfrac{\kappa_6}{\kappa_1}) (1+\tfrac{\kappa_5}{\kappa_3})\},
\end{align*}
a codimension-zero semialgebraic set in $\rrpp^6$.

Finally, we remark that for some $\vv\kappa \in \rrpp^6 \setminus \Kdt(G)$ the mass-action system generated by $(G,\vv\kappa)$ is \emph{not} quasi-thermodynamic. Indeed, a short calculation shows that when $\kappa_1 = \kappa_3 = \kappa_5 = \kappa_6$, $\kappa_2 = \kappa_4$, and $\kappa_2 > 8\kappa_1$, the Horn--Jackson function (centered at the equilibrium $(x^*_1,x^*_2)=(1,1)$) is \emph{not} a Lyapunov function.
\end{example}

\section{The maximal weakly reversible realization graph}
\label{sec:G_max}

It has been observed in \cite{craciun:jin:yu:2020} that if the pair $(G,\vv\kappa)$ is disguised vertex-balanced, then there exists an E-graph $H$ on the source vertices of $G$ such that $\vv\kappa \in \Kdt(G, H)$. Hence, there are only finitely many E-graphs that one has to consider when exploring the disguised toric locus. Thus, with $G^{\comp}$ denoting the complete simple directed graph on the source vertices of an E-graph $G$, we have
\begin{align}\label{eq:src_vertices_enough_K}
\Kdt(G)=\bigcup_{H\subseteq G^{\comp}} \Kdt(G,H).
\end{align}

It turns out that $\Kdt(G, H) \neq \emptyset$ for a subgraph $H$ of $G^{\comp}$ if and only if $G$ admits a realization with respect to $H$, and $H$ is weakly reversible.

\begin{definition}
Let $G = (V,E)$ and $H = (V_H, E_H)$ be E-graphs. We say that $G$ \emph{admits a realization with respect to} $H$ if there exist $\vv\kappa \in \rrpp^{|E|}$ and $\vv\lambda \in \rrpp^{|E_H|}$ such that $(G,\vv\kappa) \triangleq (H,\vv\lambda)$. Further, let $\Hwr(G)$ denote the set of weakly reversible subgraphs $H$ of $G^{\comp}$ for which $G$ admits a realization with respect to $H$.
\end{definition}

\begin{lemma} \label{lem:Kdt_G_H_nonempty}
Let $G$ and $H \subseteq G^{\comp}$ be E-graphs. Then the following statements hold.
\begin{enumerate}[label={\rm(\alph*)}]
\item We have $\Kdt(G, H) \neq \emptyset$ if and only if $H \in \Hwr(G)$.
\item We have $\Kdt(G) \neq \emptyset$ if and only if $\Hwr(G) \neq \emptyset$.
\end{enumerate}
\end{lemma}
\begin{proof}
To prove the ``only if'' part of (a), assume that $\Kdt(G, H) \neq \emptyset$. Let $\vv\kk \in \Kdt(G, H)$ and $\vv\lambda \in \Kt(H)$ be such that $(G,\vv\kk) \triangleq (H,\vv\lambda)$. Since $\Kt(H)$ can only be nonempty if $H$ is weakly reversible, we find that $H \in \Hwr(G)$.

To prove the ``if'' part of (a), assume that $H \in \Hwr(G)$. Denote by $V$ and $E$ the vertices and the edges of $G$, and by $V_H$ and $E_H$ the vertices and the edges of $H$. Let $\vv\kk \in \rrpp^{|E|}$ and $\vv\lambda  \in \rrpp^{|E_H|}$ be such that $(G,\vv\kk) \triangleq (H,\vv\lambda)$. Since $H$ is weakly reversible, there exists a $\vv\chi \in \rrpp^{|V_H|}$ such that 
\begin{align*}
\sum_{\yy \to \yy_0 \in E_H} \lambda_{\yy \to \yy_0} \chi_\yy &= \sum_{\yy_0 \to \yy' \in E_H} \lambda_{\yy_0 \to \yy'} \chi_{\yy_0}\:\:\:\text{for every $\yy_0 \in V_H$},
\end{align*}
see for example \cite{feinberg:horn:1977}*{Appendix}. Let $\overline{\vv\chi} \in \rrpp^{|V \cup V_H|}$ be the extension of $\vv\chi$ by setting $\overline{\chi}_\yy = 1$ for every $\yy \in V \setminus V_H$. Defining $\widetilde{\vv\kk} \in \rrpp^{|E|}$ and $\widetilde{\vv\lambda} \in \rrpp^{|E_H|}$ by
\begin{align*}
    \widetilde{\kk}_{\yy \to \yy'} = \kk_{\yy \to \yy'}\overline{\chi}_\yy \text{ for }\yy \to \yy' \in E \quad \text{ and } \quad \widetilde{\lambda}_{\yy \to \yy'} = \lambda_{\yy \to \yy'}\chi_\yy \text{ for }\yy \to \yy' \in E_H,
\end{align*}
respectively, we have $(G,\widetilde{\vv\kk}) \triangleq (H,\widetilde{\vv\lambda})$ and $\widetilde{\vv\lambda} \in \Kt(H)$ (with $\xx = \vv{1}$ being a vertex-balanced equilibrium). Hence, $\widetilde{\vv\kk} \in \Kdt(G,H)$, and thereby $\Kdt(G,H) \neq \emptyset$. 

Statement (b) directly follows from (a) and formula \eqref{eq:src_vertices_enough_K}.
\end{proof}

Note that $H_1 \cup H_2 \in \Hwr(G)$ for all $H_1, H_2 \in \Hwr(G)$ (because if $(G,\vv\kk_1) \triangleq (H_1,\vv\lambda_1)$ and $(G,\vv\kk_2) \triangleq (H_2,\vv\lambda_2)$ then $(G,\vv\kk_1+\vv\kk_2) \triangleq (H_1 \cup H_2, \widetilde{\vv\lambda_1} + \widetilde{\vv\lambda_2})$, where $\widetilde{\vv\lambda_1}$ and $\widetilde{\vv\lambda_2}$ are the extensions of $\vv\lambda_1$ and $\vv\lambda_2$, respectively, to the edges of $H_1 \cup H_2$, with zeros). Hence, provided $\Hwr(G)\neq\emptyset$, the finite partially ordered set $(\Hwr(G),\subseteq)$ has a unique maximal element.

\begin{definition}
\label{def:Gmax}
For an E-graph $G$ with $\Hwr(G)\neq\emptyset$, the unique maximal element of $(\Hwr(G), \subseteq)$ is called the \emph{maximal weakly reversible realization graph of} $G$, and is denoted as $G^{\max}$. To simplify the language, we will briefly say ``the maximal graph $G^{\max}$'' in the sequel.
\end{definition}

The distinguished role that $G^{\max}$ plays will become apparent in \Cref{sec:disg_flux,sec:topology}, and we will see in \Cref{sec:algorithm} that it can be efficiently computed by solving a linear feasibility problem. Here, we display an improvement of \eqref{eq:src_vertices_enough_K}; it is an immediate consequence of the definitions and \Cref{lem:Kdt_G_H_nonempty}(a). 

\begin{theorem}
\label{thm:G_max_enough_K}
For any E-graph $G$, we have
\begin{align*}
\Kdt(G) = \bigcup_{H\subseteq G^{\max}} \Kdt(G,H) = \bigcup_{H\in \Hwr(G)} \Kdt(G,H).
\end{align*}
\end{theorem}

We refer to \Cref{ex:running_example_Gmax} for an illustration of $G^{\comp}$ and $G^{\max}$ for the running example of the partly reversible square from \Cref{ex:running_example}.

\section{The disguised toric flux cone}
\label{sec:disg_flux}

The equations in \eqref{eq:VB} in \Cref{lem:semialgebraic_K} are nonlinear in $(\vv\lambda, \xx)$; it is advantageous to hide this nonlinearity temporarily when solving the quantifier elimination problem given by \eqref{eq:DE} and \eqref{eq:VB}. Namely, for fixed $\vv\kappa \in \rrpp^{|E|}$, $\vv\lambda \in \rrpp^{|E_H|}$, $\xx \in \rrpp^n$ define $\vv\beta \in \rrpp^{|E|}$ and $\vv\gamma \in \rrpp^{|E_H|}$ by
\begin{align*}
    \beta_{\yy \to \yy'} = \kappa_{\yy \to \yy'} \xx^\yy \text{ (for } \yy \to \yy' \in E) \quad \text{and} \quad
    \gamma_{\yy \to \yy'} = \lambda_{\yy \to \yy'} \xx^\yy \text{ (for } \yy \to \yy' \in E_H).
\end{align*}
With this, \eqref{eq:DE} can be written as a linear equation in $(\vv\beta,\vv\gamma)$, while \eqref{eq:VB} is linear in $\vv\gamma$. Hence, the nonlinear quantifier elimination problem in \Cref{lem:semialgebraic_K} can be solved in two steps: first, eliminate $\vv\gamma$ and then eliminate $\xx$. We argue in \Cref{sec:topology} that already the solution of the first step (which is a linear problem) gives valuable information about the disguised toric locus. In \Cref{sec:algorithm}, we discuss the whole procedure in more detail. In this section, we concentrate on the linear part of the problem.

\begin{definition}
A \emph{flux vector} (or \emph{flux} for short) of an E-graph $G=(V,E)$ is a vector $\vv{\beta}\in\rrpp^{|E|}$.
\begin{enumerate}[label={\rm(\roman*)}]
\item We say that a flux vector $\vv{\beta}$ is an \emph{equilibrium flux} if
\begin{align*}
\sum_{\yy\to\yy' \in E} \beta_{\yy\to\yy'} (\yy' - \yy) = \vv{0}.
\end{align*}
The set of equilibrium fluxes, denoted as $\Feq(G)$, is called the \emph{equilibrium flux cone}.
\item We say that a flux vector $\vv{\beta}$ is a \emph{vertex-balanced flux} if
\begin{align*}
\sum_{\vv{y}\to \vv{y}_0\in E}\beta_{\vv{y}\to \vv{y}_0} = \sum_{\vv{y}_0\to \vv{y}'\in E} \beta_{\vv{y}_0\to \vv{y}'}\quad \text{for all $\vv{y}_0\in V$}.
\end{align*}
The set of vertex-balanced fluxes, denoted as $\Ft(G)$, is called the \emph{toric flux cone}.
\end{enumerate}
\end{definition}

We note that the definition of a flux varies slightly in the literature; see, e.g., \cites{alradhawi:2023,telek:feliu:2023}.

The word ``cone'' in the names for $\Feq(G)$ and $\Ft(G)$ alludes to the fact that they are pointed polyhedral cones intersected with the positive orthant. The polyhedral structure of $\Ft(G)$ is discussed in \cite{craciun:jin:sorea:2024}*{Section~4}.

Next, with a harmless abuse of notation and terminology, we introduce the flux version of dynamical equality (see \Cref{def:dyn_equil_rateconstant} for the rate constant version).

\begin{definition} \label{def:dyn_equil_flux}
Let $G=(V,E)$ and $H=(V_H,E_H)$ be E-graphs and let $\vv\beta\in\rrpp^{|E|}$ and  $\vv\gamma\in\rrpp^{|E_H|}$ be flux vectors. We say that two pairs $(G,\vv\beta)$ and $(H,\vv\gamma)$ are \emph{dynamically equal}, denoted $(G,\vv\beta) \triangleq (H,\vv\gamma)$, if
\begin{align*}
\sum_{\vv y_0\to \vv y' \in E} \beta_{\vv y_0\to \vv y'}(\vv y'-\vv y_0)&=\sum_{\vv y_0 \to \vv y'\in E_H} \gamma_{\vv y_0\to \vv y'} (\vv y'-\vv y_0)\:\:\:\text{for every $\vv y_0 \in V\cup V_H$}.
\end{align*}
\end{definition}
We now argue why the two slightly different definitions of dynamical equality will not cause any confusion. For fixed E-graphs $G = (V,E)$ and $H = (V_H, E_H)$, let $\vv\varrho \in \rrpp^{|E|}$ and $\vv\sigma \in \rrpp^{|E_H|}$. Then, it is straightforward to see that $(G,\vv\varrho)$ and $(H,\vv\sigma)$ are dynamically equal via \Cref{def:dyn_equil_rateconstant} (where $\vv\varrho$ and $\vv\sigma$ are meant to be rate constants) if and only if $(G,\vv\varrho)$ and $(H,\vv\sigma)$ are dynamically equal via \Cref{def:dyn_equil_flux} (where $\vv\varrho$ and $\vv\sigma$ are meant to be fluxes).

We are now ready to introduce the analogue of \Cref{def:KdtG} for fluxes, and then we obtain the flux versions of \Cref{lem:semialgebraic_K} and the chain \eqref{eq:chain_K}.

\begin{definition}
\label{def:dt_flux}
For an E-graph $G=(V,E)$, define the \emph{disguised toric flux cone} of $G$ as the set
\begin{align*}
    \Fdt(G) = \{\vv\beta \in \rrpp^{|E|} \st \text{there exist an E-graph $H$ and a $\vv{\gamma} \in \Ft(H)$ such that $(G,\vv\beta) \triangleq (H,\vv\gamma)$}\}.
\end{align*}
Furthermore, we define the \emph{disguised toric flux cone of $G$ with respect to $H$} for a fixed E-graph $H$ to be
\begin{align*}
\Fdt(G,H)=\{\vv\beta \in \rrpp^{|E|}\st \text{there exists a }\vv\gamma\in\Ft(H) \text{ such that }(G,\vv\beta) \triangleq (H,\vv\gamma)\}.
\end{align*}
\end{definition}

\begin{lemma}\label{lem:semialgebraic_F}
Let $G=(V,E)$ and $H=(V_H,E_H)$ be E-graphs in $\zz^n$, and $\vv{\beta} \in \rrpp^{|E|}$. Then $\vv\beta \in \Fdt(G,H)$ if and only if there exists a $\vv\gamma\in\rrpp^{|E_H|}$ that satisfy the following dynamical equality and vertex-balancing conditions:
\begin{align}
\label{eq:DE-f}\tag{DE-f}
\sum_{\vv y_0\to \vv y' \in E} \beta_{\vv y_0\to \vv y'}(\vv y'-\vv y_0)&=\sum_{\vv y_0 \to \vv y'\in E_H} \gamma_{\vv y_0\to \vv y'} (\vv y'-\vv y_0)\:\:\:\text{for every $\vv y_0 \in V\cup V_H$};\\
\label{eq:VB-f}\tag{VB-f}
\sum_{\yy \to \yy_0 \in E_H} \gamma_{\yy \to \yy_0} &= \sum_{\yy_0 \to \yy' \in E_H} \gamma_{\yy_0 \to \yy'} \:\:\:\text{for every $\yy_0 \in V_H$}.
\end{align}
\end{lemma}

For any E-graph $G$, we have a flux-analog to the inclusions in \eqref{eq:chain_K}, namely 
\begin{align} \label{eq:chain_F}
\Ft(G) \subseteq \Fdt(G) \subseteq \Feq(G) \subseteq \rrpp^{|E|},
\end{align}
and if $\Feq(G)$ is nonempty, it holds that $\dim \Feq(G)=|E|-\dim\mc{S}$.

Next, we relate the nonemptiness of the disguised toric flux cone and the nonemptiness of the disguised toric locus. Further, the equality of $\Fdt(G)$ and $\Feq(G)$ (see the chain \eqref{eq:chain_F}) is described in terms of the rate constant versions of the same objects.
\begin{proposition} \label{prop:Fdt_iff_Kdt}
Let $G$ and $H$ be E-graphs. Then the following statements hold.
\begin{enumerate}[label={\rm(\alph*)}]
    \item We have $\Fdt(G,H) \neq \emptyset$ if and only if $\Kdt(G,H) \neq \emptyset$. 
    \item We have $\Fdt(G) \neq \emptyset$ if and only if $\Kdt(G) \neq \emptyset$.
    \item We have $\Fdt(G) = \Feq(G)$ if and only if $\Kdt(G) = \Keq(G)$.
\end{enumerate}
\end{proposition}

The flux analogue of \Cref{thm:G_max_enough_K} states that for finding $\Fdt(G)$, it suffices to compute $\Fdt(G,H)$ for $H \subseteq G^{\max}$.

\begin{theorem}
\label{thm:G_max_enough_F}
For any E-graph $G$, we have
\begin{align*}
\Fdt(G) = \bigcup_{H\subseteq G^{\max}} \Fdt(G,H) = \bigcup_{H\in \Hwr(G)} \Fdt(G,H).
\end{align*}
\end{theorem}

Up to this point, every statement in this section was an immediate consequence of the definitions and the statements discussed in \Cref{sec:prelim}. The rest of this section deals with the conic structure of the disguised toric flux cone (\Cref{lem:polyhedral_cones}), the relation of $\Fdt(G,H_1)$ and $\Fdt(G,H_2)$ when $H_1 \subseteq H_2$ (\Cref{lem:H1_subgraph_H2}), and the connection between $\Fdt(G)$ and $\Fdt(G,G^{\max})$ (\Cref{thm:FdtG_is_clFdtGGmax}).

\begin{lemma}
\label{lem:polyhedral_cones}
Let $G$ and $H$ be E-graphs such that $\Fdt(G,H) \neq \emptyset$. Then the following holds.
\begin{enumerate}[label={\rm(\alph*)}]
\item The closure $\overline{\Fdt(G)}$ is a pointed polyhedral cone, and $\Fdt(G)=(\overline{\Fdt(G)})_{>0}$. \\
In particular, $\Fdt(G)$ is convex and is relatively closed in $\rrpp^{|E|}$.
\item The closure $\overline{\Fdt(G,H)}$ is a pointed polyhedral cone, and  $\Fdt(G,H) = \relint(\overline{\Fdt(G,H)})$.\\
In particular, $\Fdt(G,H)$ is convex and is an open polyhedral cone.
\end{enumerate}
\end{lemma}
\begin{proof}
Denote by $E$, $E^{\max}$, and $E_H$ the edge sets of $G$, $G^{\max}$, and $H$, respectively.

We begin by proving (b). Let $\mathcal{X}(G,H)$ be the collection of tuples $(\vv \beta,\vv \gamma) \in \rrp^{|E|} \times \rrp^{|E_H|}$ that satisfy \eqref{eq:DE-f} and \eqref{eq:VB-f} in \Cref{lem:semialgebraic_F} (notice that here both $\vv\beta$ and $\vv\gamma$ are allowed to have vanishing coordinates), and let $\pi \colon \rr^{|E|} \times \rr^{|E_H|} \to \rr^{|E|}$ be the projection to the first factor. Since $\mathcal{X}(G,H)$ is an intersection of a linear subspace and the polyhedral cone $\rrp^{|E|}\times\rrp^{|E_H|}$, it is a polyhedral cone, and so is its projection $\pi(\mathcal{X}(G,H))$. Moreover, the projection is pointed, since it is contained in the pointed polyhedral cone $\rrp^{|E|}$. Since $\mathcal{X}(G,H)$ by assumption contains a positive point, it follows from \cite{rockafellar:1970}*{Corollary~6.5.1} that its relative interior is $\mathcal{X}(G,H)\cap(\rrpp^{|E|}\times\rrpp^{|E_H|})$. Since relative interior commutes with projection by \cite{rockafellar:1970}*{Theorem~6.6}, we conclude that the relative interior of $\pi(\mathcal{X}(G,H))$ is 
\[\Fdt(G,H)=\pi\left(\mathcal{X}(G,H)\cap(\rrpp^{|E|}\times\rrpp^{|E_H|})\right).\] 
This implies that $\overline{\Fdt(G,H)}=\pi(\mathcal{X}(G,H))$, and the desired conclusions follow.

To prove (a), let $\mathcal{X}(G)=\mathcal{X}(G,G^{\max})$. By \Cref{thm:G_max_enough_F}, it holds that $\Fdt(G)=\pi(\mathcal{X}(G))\cap\rrpp^{|E|}$. Hence, $\Fdt(G,G^{\max})\subseteq\Fdt(G)\subseteq\pi(\mathcal{X}(G))$. Taking closures gives $\overline{\Fdt(G)}=\pi(\mathcal{X}(G))$, as $\overline{\Fdt(G,G^{\max})}=\pi(\mathcal{X}(G))$ by the proof of part (b), which concludes the proof.
\end{proof}

Next, we examine the relation between the disguised toric flux cones $\Fdt(G, H_1)$ and $\Fdt(G, H_2)$, where $H_1$ is a subgraph of $H_2$.

\begin{lemma} \label{lem:H1_subgraph_H2}
Let $G$ be an E-graph, and let $H_1, H_2 \in \Hwr(G)$ be such that $H_1 \subseteq H_2$.
Then $\Fdt(G,H_1) \subseteq \overline{\Fdt(G,H_2)}$.
\end{lemma}
\begin{proof}
Denote by $E_{H_1}$ and $E_{H_2}$ the edge sets of $H_1$ and $H_2$, respectively. For any $\vv\beta \in \Fdt(G,H_1)$ there exists a $\vv\gamma_1 \in \rrpp^{|E_{H_1}|}$ such that $(G,\vv\beta) \triangleq (H_1,\vv\gamma_1)$ and $(H_1,\vv\gamma_1)$ is vertex-balanced. For any such pair $(\vv\beta, \vv\gamma_1)$ define $\vv\gamma_2 \in \rrp^{|E_{H_2}|}$ such that $\vv \gamma_1$ and $\vv\gamma_2$ agree on the coordinates referring to edges in $E_{H_1}$, while the coordinates of $\vv\gamma_2$ referring to the edges in $E_{H_2} \setminus E_{H_1}$ vanish. Then $(\vv\beta, \vv\gamma_2) \in \mc{X}(G,H_2)$, where $\mc{X}(G,H_2)$ is the pointed polyhedral cone defined in the proof of \Cref{lem:polyhedral_cones}(b) above. Hence, with $\pi$ denoting the projection from that same proof, it follows that $\vv\beta \in \pi(\mc{X}(G, H_2))$. Since $\overline{\Fdt(G,H_2)} = \pi(\mc{X}(G,H_2))$, this concludes the proof.
\end{proof}

Next, we state and prove the main result of this section, which highlights the distinguished role of the maximal graph $G^{\max}$.

\begin{theorem} \label{thm:FdtG_is_clFdtGGmax}
Let $G = (V, E)$ be an E-graph such that $\Fdt(G)\neq\emptyset$. Then
\begin{align*}
\Fdt(G) = (\overline{\Fdt(G,G^{\max})})_{>0} \quad \text{and} \quad \relint\Fdt(G) = \Fdt(G,G^{\max}).
\end{align*}
In particular, $\dim\Fdt(G) = \dim\Fdt(G,G^{\max})$.
\end{theorem}
\begin{proof}
The inclusion $\Fdt(G) \subseteq (\overline{\Fdt(G,G^{\max})})_{>0}$ follows from \Cref{thm:G_max_enough_F} and \Cref{lem:H1_subgraph_H2}, where the latter is applied with $H_2 = G^{\max}$. The converse inclusion $(\overline{\Fdt(G,G^{\max})})_{>0} \subseteq \Fdt(G)$ also follows, because $\Fdt(G,G^{\max}) \subseteq \Fdt(G)$ by the definitions, and $\Fdt(G)$ is relatively closed in $\rrpp^{|E|}$ by \Cref{lem:polyhedral_cones}(a). Finally, the equality $\relint\Fdt(G) = \Fdt(G,G^{\max})$ is a consequence of \Cref{lem:polyhedral_cones}(b) and $\Fdt(G) = (\overline{\Fdt(G,G^{\max})})_{>0}$.
\end{proof}

We end the section by describing how the boundary of $\Fdt(G)$ is given by subgraphs of $G^{\max}$. The following lemma shows that each subgraph $H\subseteq G^{\max}$ can only contribute to the relative interior of a single face of the polyhedral cone $\overline{\Fdt(G)}$. 
(By a \emph{face} of a closed polyhedral cone $\mathcal{P}$ we mean either $\mathcal{P}$ itself or the intersection of $\mathcal{P}$ with a supporting hyperplane.) 

\begin{lemma}
\label{lem:each_H_contributed_to_one_face}
Let $G$ be an E-graph. For any graph $H\subseteq G^{\max}$ such that $\Fdt(G,H)\neq\emptyset$ and any face $\sigma$ of $\overline{\Fdt(G)}$, it either holds that $\Fdt(G,H)\cap\relint(\sigma)=\emptyset$ or that $\Fdt(G,H)\subseteq\relint(\sigma)$. 
\end{lemma}

\begin{proof}
Let $H$ be as in the statement, and let $\sigma$ be any face of $\overline{\Fdt(G)}$ with $\Fdt(G,H)\cap\relint(\sigma)\neq\emptyset$.  
Furthermore, let $\Pi$ be a supporting hyperplane of $\overline{\Fdt(G)}$ for the face $\sigma$ (i.e., $\overline{\Fdt(G)}\cap \Pi=\sigma$), and denote by $\Pi^+$ the closed half-space determined by $\Pi$ such that $\Fdt(G) \subseteq \Pi^+$.
Then, since the relatively open convex set $\Fdt(G,H)$ intersects $\Pi$ and is contained in $\Pi^+$, it holds that $\Fdt(G,H) \subseteq \Pi$.
Since $\Fdt(G,H)\subseteq \overline{\Fdt(G)}$, we conclude that $\Fdt(G,H)\subseteq \overline{\Fdt(G)}\cap \Pi=\sigma$.
Furthermore, $\Fdt(G,H) \subseteq\relint(\sigma)$, since $\Fdt(G,H)$ is relatively open and intersects $\relint(\sigma)$, and therefore cannot intersect any proper face of $\sigma$.
\end{proof}

For each face $\sigma$ of $\overline{\Fdt(G)}$ that intersects the positive orthant, we now consider, in analogy with \Cref{sec:G_max}, the set
\[\Hwr_\sigma(G)=\{H\subseteq G^{\max}\mid\emptyset\neq\Fdt(G,H)\subseteq\relint(\sigma)\}.\]
Note that by \Cref{thm:G_max_enough_F} and \Cref{lem:each_H_contributed_to_one_face},  $\Hwr_{\sigma}(G)$ is nonempty. Furthermore, similarly to $\Hwr(G)$, we see that $\Hwr_{\sigma}(G)$ is closed under graph union. Indeed, for $H_1, H_2 \in \Hwr_\sigma(G)$ let $\vv\beta_1, \vv\beta_2 \in \relint(\sigma)$, $\vv\gamma_1 \in \Ft(H_1)$, $\vv\gamma_2 \in \Ft(H_2)$ be such that $(G,\vv\beta_1) \triangleq (H_1, \vv\gamma_1)$ and $(G,\vv\beta_2) \triangleq (H_2, \vv\gamma_2)$. Then $(G,\vv\beta_1 + \vv\beta_2) \triangleq (H_1 \cup H_2, \widetilde{\vv\gamma_1} + \widetilde{\vv\gamma_2})$ and $\widetilde{\vv\gamma_1} + \widetilde{\vv\gamma_2} \in \Ft(H_1 \cup H_2)$, where $\widetilde{\vv\gamma_1}$ and $\widetilde{\vv\gamma_2}$ are the extensions of $\vv\gamma_1$ and $\vv\gamma_2$, respectively, to the edges of $H_1 \cup H_2$, with zeros. Since $\vv\beta_1 + \vv\beta_2 \in \relint(\sigma)$, the claim $H_1 \cup H_2 \in \Hwr_\sigma(G)$ is therefore proven. Hence, $\Hwr_{\sigma}(G)$ has a unique inclusion-maximal element, which we denote by $G_\sigma^{\max}$. 

Letting $\sigma$ range over all faces of $\overline{\Fdt(G)}$ that intersect the positive orthant, the sets $\Fdt(G,G_\sigma^{\max})$ together recover the standard face stratification of $\Fdt(G)$, as shown by the following result.

\begin{theorem}
Let $G$ be an E-graph such that $\Fdt(G)$ is nonempty. Then, for each face $\sigma$ of $\overline{\Fdt(G)}$ such that $\sigma\cap\rrpp^{|E|}\neq\emptyset$, it holds that
\begin{align*}
\sigma = \overline{\Fdt(G,G^{\max}_\sigma)} \quad \text{and} \quad \relint(\sigma) = \Fdt(G,G^{\max}_\sigma).
\end{align*}
\end{theorem}

\begin{proof}
We start by proving the first equality. By definition, $\sigma\supseteq\Fdt(G,G^{\max}_\sigma)$, and taking closures gives $\sigma\supseteq\overline{\Fdt(G,G^{\max}_\sigma)}$. For the converse inclusion, 
first note that, by \Cref{lem:polyhedral_cones}(a), we have $\relint(\sigma) \subseteq \Fdt(G)$. Then, by \Cref{thm:G_max_enough_F}, for any $\vv\beta\in\relint(\sigma)$ there exists a subgraph $H\subseteq G^{\max}$ such that $\vv\beta\in\Fdt(G,H)$. By \Cref{lem:each_H_contributed_to_one_face}, $\Fdt(G,H) \subseteq \relint(\sigma)$, so $H \in \Hwr_\sigma(G)$. Then, by \Cref{lem:H1_subgraph_H2}, $\Fdt(G,H)\subseteq\overline{\Fdt(G,G^{\max}_\sigma)}$, so $\vv\beta\in \overline{\Fdt(G,G^{\max}_\sigma)}$. Hence, $\relint(\sigma)\subseteq\overline{\Fdt(G,G^{\max}_\sigma)}$, and $\sigma \subseteq \overline{\Fdt(G,G^{\max}_\sigma)}$ follows by taking the closure of both sides.
Finally, the second equality in the statement follows from the first equality together with \Cref{lem:polyhedral_cones}(b).
\end{proof}

\begin{example}
We return to the partly reversible square from \Cref{ex:running_example}, which is the graph $G$ on the left of \eqref{tikz:running_flux}. A straightforward calculation shows that
\begin{align*}
    \Feq(G) &= \{\vv\beta \in \rrpp^6 \st \beta_1 = \beta_3 \text{ and }\beta_2 + \beta_6 = \beta_4 + \beta_5 \},\\
    \Ft(G) &= \{\vv\beta \in \Feq(G) \st \beta_1 = \beta_4 - \beta_6 \}.
\end{align*}
The maximal weakly reversible realization graph $G^{\max}$ is given on the right in \eqref{tikz:running_flux}, and it holds (see \Cref{subsec:square_parallelogram} for details) that
\begin{align*}
\Fdt(G,G^{\max}) &= \{\vv\beta \in \Feq(G) \st |\beta_4 - \beta_6|< \beta_1 < \beta_4 + \beta_5 \}, \\
\Fdt(G)= (\overline{\Fdt(G,G^{\max})})_{>0} &= \{\vv\beta \in \Feq(G) \st |\beta_4 - \beta_6|\leq \beta_1 \leq \beta_4 + \beta_5 \}.
\end{align*}
\begin{align}\label{tikz:running_flux}
\begin{aligned}
\begin{tikzpicture}[scale=2]

\tikzset{mybullet/.style={inner sep=1.5pt,outer sep=5pt,draw,fill,Egraphcolor,circle}};
\tikzset{myarrow/.style={arrows={-stealth},very thick,Egraphcolor}};

\draw [step=1, gray, very thin] (0,0) grid (1.25,1.25);
\draw [ -, black] (0,0)--(1.25,0);
\draw [ -, black] (0,0)--(0,1.25);

\node[mybullet]  (0) at (0,0) {};
\node[mybullet]  (X) at (1,0) {};
\node[mybullet] (XY) at (1,1) {};
\node[mybullet]  (Y) at (0,1) {};

\draw[myarrow]  (0) to node[below] {\fluxcolor{$\beta_1$}} (X);
\draw[myarrow,transform canvas={xshift=2pt}]   (X) to node[right] {\fluxcolor{$\beta_2$}} (XY);
\draw[myarrow,transform canvas={xshift=-2pt}] (XY) to node[left]  {\fluxcolor{$\beta_5$}} (X);
\draw[myarrow] (XY) to node[above] {\fluxcolor{$\beta_3$}} (Y);
\draw[myarrow,transform canvas={xshift=-2pt}]  (Y) to node[left]  {\fluxcolor{$\beta_4$}} (0);
\draw[myarrow,transform canvas={xshift=2pt}]   (0) to node[right] {\fluxcolor{$\beta_6$}} (Y);

\node at (0.5,1.5) {$(G,\fluxcolor{\vv\beta})$};
\node at (-0.5,0.5) {\phantom{0}};

\begin{scope}[shift={(1.7,0)}]

\draw [step=1, gray, very thin] (0,0) grid (1.25,1.25);
\draw [ -, black] (0,0)--(1.25,0);
\draw [ -, black] (0,0)--(0,1.25);

\node[mybullet]  (0) at (0,0) {};
\node[mybullet]  (X) at (1,0) {};
\node[mybullet] (XY) at (1,1) {};
\node[mybullet]  (Y) at (0,1) {};

\draw[myarrow]  (0) to node[below] {} (X);
\draw[myarrow]  (X) to node[right] {} (XY);
\draw[myarrow] (XY) to node[above] {} (Y);
\draw[myarrow]  (Y) to node[left]  {} (0);
\draw[myarrow,transform canvas={xshift=-1.4pt,yshift=1.4pt}] (XY) to node[xshift=7pt,yshift=-7pt]  {} (0);
\draw[myarrow,transform canvas={xshift=1.4pt,yshift=-1.4pt}]  (0) to node[xshift=-7pt,yshift=7pt] {} (XY);

\node at (0.5,1.5) {$H_1$};
    
\end{scope}

\begin{scope}[shift={(3.4,0)}]

\draw [step=1, gray, very thin] (0,0) grid (1.25,1.25);
\draw [ -, black] (0,0)--(1.25,0);
\draw [ -, black] (0,0)--(0,1.25);

\node[mybullet]  (0) at (0,0) {};
\node[mybullet]  (X) at (1,0) {};
\node[mybullet] (XY) at (1,1) {};
\node[mybullet]  (Y) at (0,1) {};

\draw[myarrow,transform canvas={xshift=2pt}]   (X) to node {} (XY);
\draw[myarrow,transform canvas={xshift=-2pt}] (XY) to node {}  (X);
\draw[myarrow,transform canvas={xshift=-2pt}]  (Y) to node {} (0);
\draw[myarrow,transform canvas={xshift=2pt}]   (0) to node {} (Y);
\draw[myarrow,transform canvas={xshift=1.4pt,yshift=-1.4pt}]  (0) to node {} (XY);
\draw[myarrow,transform canvas={xshift=-1.4pt,yshift=1.4pt}] (XY) to node {} (0);

\node at (0.5,1.5) {$H_2$};
\node at (1.5,0.5) {\phantom{0}};

\end{scope}

\begin{scope}[shift={(5.1,0)}]

\draw [step=1, gray, very thin] (0,0) grid (1.25,1.25);
\draw [ -, black] (0,0)--(1.25,0);
\draw [ -, black] (0,0)--(0,1.25);

\node[mybullet]  (0) at (0,0) {};
\node[mybullet]  (X) at (1,0) {};
\node[mybullet] (XY) at (1,1) {};
\node[mybullet]  (Y) at (0,1) {};

\draw[myarrow]  (0) to node {} (X);
\draw[myarrow,transform canvas={xshift=2pt}]   (X) to node {} (XY);
\draw[myarrow,transform canvas={xshift=-2pt}] (XY) to node {}  (X);
\draw[myarrow]  (XY) to node {} (Y);
\draw[myarrow,transform canvas={xshift=-2pt}]  (Y) to node {} (0);
\draw[myarrow,transform canvas={xshift=2pt}]   (0) to node {} (Y);
\draw[myarrow,transform canvas={xshift=1.4pt,yshift=-1.4pt}]  (0) to node {} (XY);
\draw[myarrow,transform canvas={xshift=-1.4pt,yshift=1.4pt}] (XY) to node {} (0);

\node at (0.5,1.5) {$G^{\max}$};
\node at (1.5,0.5) {\phantom{0}};

\end{scope}

\draw[rounded corners,lightgray] (current bounding box.south west) rectangle (current bounding box.north east);

\end{tikzpicture}
\end{aligned}
\end{align}

The polyhedral cone $\overline{\Fdt(G)}$ is 4-dimensional.
Besides the full cone itself, it has three proper faces that intersect $\rrpp^6$, all of which are 3-dimensional (all remaining proper faces are contained in coordinate hyperplanes).
The maximal realization graphs $G_{\sigma}^{\max}$ for these three faces are, respectively, given by $G$ itself, $H_1$ and $H_2$ in \eqref{tikz:running_flux}. More specifically, 
\begin{align*}
\Fdt(G,G) &=\{\vv\beta \in \Feq(G) \st \beta_1 = \beta_4 - \beta_6 \},\\
\Fdt(G,H_1) &= \{\vv\beta \in \Feq(G) \st \beta_1 = \beta_4 + \beta_5 \},\\
\Fdt(G,H_2) &= \{\vv\beta \in \Feq(G) \st \beta_1 = \beta_6 - \beta_4 \}.
\end{align*}
We remark that, for this particular example, it holds that $\Ft(G) = \Fdt(G,G)$; however, for an arbitrary E-graph $G$, only the inclusion $\Ft(G) \subseteq \Fdt(G,G)$ is guaranteed to hold. 
\end{example}
\section{The topology of the disguised toric locus}
\label{sec:topology}

In this section, we study some topological properties of the disguised toric locus by exhibiting a homeomorphism to a product space formed from the disguised flux cone. In \Cref{subsec:Kdt_G_H}, we focus on $\Kdt(G,H)$, and in \Cref{subsec:Kdt_G}, we apply an analogous analysis to $\Kdt(G)$.

\subsection{The disguised toric locus $\Kdt(G,H)$}
\label{subsec:Kdt_G_H}

For this section, we fix the E-graphs $G=(V,E)$ and $H=(V_H,E_H)$ with stoichiometric subspaces $\mc{S}$ and $\mc{S}_{H}$, respectively, and assume that $\Kdt(G,H)\neq\emptyset$. We also fix an $\xx_0 \in \mathbb{R}^n_{>0}$, and thereby the positive stoichiometric class $(\xx_0 + \mc{S}_{H})_{>0}$ of $H$. Our goal is to show that $\Kdt(G, H)$ is homeomorphic to the direct product of $(\xx_0+\mc{S}_H)_{>0}$ and $\Fdt(G, H)$. To this end, we define the map 
\begin{align*}
\begin{split}
\Psi_H\colon (\xx_0 + \mc{S}_{H})_{>0} \times \Fdt(G,H) &\to \Kdt(G,H), \\
\Psi_H(\vv x, \vv\beta) &= (\beta_{\vv y \to \vv y'} \xx^{-\yy})_{\yy \rightarrow \yy' \in E}.
\end{split}
\end{align*}
It is immediate from the definitions of the sets $\Fdt(G,H)$ and $\Kdt(G,H)$ that the rate constant vector $(\beta_{\vv y \to \vv y'} \xx^{-\yy})_{\yy \rightarrow \yy' \in E}$ is indeed in $\Kdt(G,H)$ for all $(\vv x, \vv\beta) \in (\xx_0 + \mc{S}_{H})_{>0} \times \Fdt(G,H)$. Further, whenever $\xx$ and $\vv\kappa$ are related via $\Psi_H(\xx,\vv\beta)= \vv\kk$ (for some $\vv\beta$), the point $\xx$ is the unique positive equilibrium of $(G,\vv\kk)$ in $(\xx_0 + \mc{S}_{H})_{>0}$. Since for every $\vv\kk \in \Kdt(G,H)$ there exists a unique positive equilibrium $\xx^*_{\vv\kk}$ of $(G,\vv\kk)$ in $(\xx_0 + \mc{S}_{H})_{>0}$ by \Cref{thm:nice_dynamics_disg}(a), we conclude that the map $\Psi_H$ is a bijection, and its inverse is given by
\begin{align*}
\begin{split}
\Psi_H^{-1}\colon \Kdt(G,H) &\to  (\xx_0 + \mc{S}_{H})_{>0} \times \Fdt(G,H), \\
\Psi_H^{-1}(\vv\kk) &= (\xx^*_{\vv\kk},(\kk_{\yy\to\yy'}(\xx^*_{\vv\kk})^{\yy})_{\yy \rightarrow \yy' \in E}).
\end{split}
\end{align*}
The following result tells us that the map $\Psi_H$ is, in fact, a homeomorphism. As a direct consequence of this, we obtain the main result of this section, \Cref{cor:Kdt_G_H} below, which states that the disguised toric locus $\Kdt(G, H)$ is a contractible (and hence, in particular, simply connected) manifold, whose dimension can be read off from $\Fdt(G,H)$.

\begin{theorem} \label{thm:homeo_Psi_H}
The map $\Psi_H$ is a homeomorphism between $(\xx_0 + \mc{S}_{H})_{>0} \times \Fdt(G,H)$ and $\Kdt(G,H)$.
\end{theorem}
\begin{proof}
Clearly, $\Psi_H$ is continuous. Further, as a consequence of the implicit function theorem, the map $\phi_H\colon \Kdt(G,H) \to (\xx_0 + \mc{S}_H)_{>0}$, defined by $\phi(\vv\kappa) = \xx^*_{\vv\kappa}$, is continuous, because $\xx^*_{\vv\kappa}$ is a nondegenerate equilibrium within $(\xx_0+\mc{S}_H)_{>0}$ by \Cref{thm:nice_dynamics_disg}(c). Hence, the continuity of the inverse map $\Psi_H^{-1}$ also follows.
\end{proof}

\begin{corollary}
\label{cor:Kdt_G_H}
Assume that $\Kdt(G,H)\neq\emptyset$. Then the following statements hold.
\begin{enumerate}[label={\rm(\alph*)}]
\item The disguised toric locus $\Kdt(G,H)$ is a manifold with no boundary.
\item The disguised toric locus $\Kdt(G,H)$ is contractible.
\item We have $\dim \Kdt(G,H) = \dim \mc{S}_H + \dim \Fdt(G,H)$.
\end{enumerate}
\end{corollary}

\begin{proof}
By \Cref{lem:polyhedral_cones}(b), the set $\Fdt(G,H)$ is an open polyhedral cone, and hence, it is a contractible manifold without boundary. Parts (a), (b), and (c) are now immediate.
\end{proof}

\subsection{The disguised toric locus $\Kdt(G)$}
\label{subsec:Kdt_G}

For this section, we fix an E-graph $G=(V,E)$ with stoichiometric subspace $\mc{S}$, and assume that $\Kdt(G)\neq\emptyset$. Suppose further that
\begin{align}\label{eq:kinetic_equals_stoichiometric}
\text{the kinetic subspace of $(G,\vv\kappa)$ equals to $\mc{S}$ for all $\vv{\kk} \in \mathbb{R}_{>0}^{|E|}$.}    
\end{align}
In general, condition \eqref{eq:kinetic_equals_stoichiometric} is hard to verify, but it holds, for example, for all weakly reversible E-graphs by \Cref{thm:FH1977}.
As a consequence of \eqref{eq:kinetic_equals_stoichiometric}, whenever $\Kdt(G,H) \neq \emptyset$ for an E-graph $H$, we have $\mc{S} = \mc{S}_H$, where $\mc{S}_{H}$ denotes the stoichiometric subspace of the E-graph $H$. We also fix an $\xx_0 \in \mathbb{R}^n_{>0}$, and thereby we fix the positive stoichiometric class $(\xx_0 + \mc{S})_{>0}$ of $G$ (which equals to the positive stoichiometric class $(\xx_0 + \mc{S}_H)_{>0}$ of $H$ for any $H$ with $\Kdt(G,H) \neq \emptyset$). Our goal is to show that $\Kdt(G)$ is homeomorphic to the direct product of $(\xx_0+\mc{S})_{>0}$ and $\Fdt(G)$. To this end, we define the map 
\begin{align*}
\begin{split}
\Psi\colon (\xx_0 + \mc{S})_{>0} \times \Fdt(G) &\to \Kdt(G), \\
\Psi(\vv x, \vv\beta) &= (\beta_{\vv y \to \vv y'} \xx^{-\yy})_{\yy \rightarrow \yy' \in E}.
\end{split}
\end{align*}
It is immediate from the definitions of the sets $\Fdt(G)$ and $\Kdt(G)$ that the rate constant vector $(\beta_{\vv y \to \vv y'} \xx^{-\yy})_{\yy \rightarrow \yy' \in E}$ is indeed in $\Kdt(G)$ for all $(\vv x, \vv\beta) \in (\xx_0 + \mc{S})_{>0} \times \Fdt(G)$. Further, whenever $\xx$ and $\vv\kappa$ are related via $\Psi(\xx,\vv\beta)= \vv\kk$ (for some $\vv\beta$), the point $\xx$ is the unique positive equilibrium of $(G,\vv\kk)$ in $(\xx_0 + \mc{S})_{>0}$. Since for every $\vv\kk \in \Kdt(G)$ there exists a unique positive equilibrium $\xx^*_{\vv\kk}$ of $(G,\vv\kk)$ in $(\xx_0 + \mc{S})_{>0}$ by \Cref{thm:nice_dynamics_disg}(a), we obtain that $\Psi$ is a bijection, and its inverse is given by
\begin{align*}
\begin{split}
\Psi^{-1}\colon \Kdt(G) &\to  (\xx_0 + \mc{S})_{>0} \times \Fdt(G), \\
\Psi^{-1}(\vv\kk) &= (\xx^*_{\vv\kk},(\kk_{\yy\to\yy'}(\xx^*_{\vv\kk})^{\yy})_{\yy \rightarrow \yy' \in E}).
\end{split}
\end{align*}
The following result tells us that the map $\Psi$ is, in fact, a homeomorphism. As a direct consequence of this, we obtain the main result of this section, \Cref{cor:Kdt_G} below, which states that the disguised toric locus $\Kdt(G)$ is a contractible manifold with boundary, whose topology in many ways is reflected by the topology of $\Fdt(G)$. 
We note that part (b) of \Cref{cor:Kdt_G} is a sharpening of the already known fact that $\Kdt(G)$ is connected \cite{craciun:deshpande:jin:2024b}, since every contractible manifold is simply connected.

\begin{theorem}\label{thm:homeo_Psi}
For any E-graph $G$ that satisfies \eqref{eq:kinetic_equals_stoichiometric}, the map $\Psi$ is a homeomorphism between $(\xx_0 + \mc{S})_{>0} \times \Fdt(G)$ and $\Kdt(G)$.
\end{theorem}
\begin{proof}
Clearly, $\Psi$ is continuous. Further, as a consequence of the implicit function theorem, the map $\phi\colon \Kdt(G) \to (\xx_0 + \mc{S})_{>0}$, defined by $\phi(\vv\kappa) = \xx^*_{\vv\kappa}$, is continuous, because $\xx^*_{\vv\kappa}$ is a nondegenerate equilibrium within $(\xx_0+\mc{S})_{>0}$ by \Cref{thm:nice_dynamics_disg}(c). Hence, the continuity of the inverse map $\Psi^{-1}$ also follows.
\end{proof}

\begin{corollary}
\label{cor:Kdt_G}
Assume that $\Kdt(G)\neq \emptyset$ and \eqref{eq:kinetic_equals_stoichiometric} is satisfied. Then the following statements hold.
\begin{enumerate}[label={\rm(\alph*)}]
\item The disguised toric locus $\Kdt(G)$ is a manifold with boundary.
\item The disguised toric locus $\Kdt(G)$ is contractible.
\item We have $\dim \Kdt(G) = \dim \mc{S} + \dim \Fdt(G)$.
\item The manifold interior of $\Kdt(G)$ is $\Kdt(G,G^{\max})$. 
\item We have $\dim\Kdt(G)=\dim\Kdt(G,G^{\max})$.
\item The sets $\Kdt(G,G_\sigma^{\max})$ give a topological stratification\footnote{Here, we use the term \emph{topological stratification} in the sense of \cite{stacks}*{\href{https://stacks.math.columbia.edu/tag/09Y0}{Tag 09Y0}}.} of $\Kdt(G)$, where $\sigma$ ranges over all faces of $\overline{\Fdt(G)}$ with $\sigma\cap\rrpp^{|E|}\neq\emptyset$ (see \Cref{sec:disg_flux} for the definition of $G_\sigma^{\max}$).
\end{enumerate}
\end{corollary}

\begin{proof}
By \Cref{lem:polyhedral_cones}(a), the set $\Fdt(G)$ is obtained by removing some faces of the pointed polyhedral cone $\overline{\Fdt(G)}$, and hence, it is a contractible manifold with boundary. Parts (a), (b), and (c) are now immediate, and part (d) follows from \Cref{thm:FdtG_is_clFdtGGmax} together with \Cref{thm:homeo_Psi_H} applied to $H=G^{\max}$. Part (e) follows immediately from part (d). 
Part (f) follows from \Cref{thm:homeo_Psi_H} together with the fact that $\Fdt(G,G_{\sigma}^{\max})$ gives a topological stratification of $\Fdt(G)$.
\end{proof}

We conclude this section by providing an example that violates condition \eqref{eq:kinetic_equals_stoichiometric} and for which $\dim \Kdt(G) \neq \dim \mc{S} + \dim \Fdt(G)$. Consider the one-dimensional E-graph $G$ given by $\sf{0} \leftarrow  \sf{X} \to 2\sf{X}$. The dimension of the stoichiometric subspace is one, but the dimension of the kinetic subspace is zero (when $\kk_{\sf{X} \to \sf{0}} = \kk_{\sf{X} \to 2\sf{X}}$) or one (when $\kk_{\sf{X} \to \sf{0}} \neq \kk_{\sf{X} \to 2\sf{X}}$). Further, both $\Kdt(G)$ and $\Fdt(G)$ are one-dimensional.

\section{A flux-based procedure for computing the disguised toric locus}
\label{sec:algorithm}

Computationally, we can obtain a semialgebraic description of $\Kdt(G)$ through the process of \emph{quantifier elimination}, which is implemented in, e.g.,  \texttt{Mathematica} or \texttt{Redlog}. However, general-purpose algorithms for quantifier elimination tend to be very computationally expensive, especially for nonlinear problems like this one.
The authors of \cite{brustenga:craciun:sorea:2022} make use of the matrix--tree theorem to reduce the complexity of the quantifier elimination problem, but the resulting algorithm is still only feasible for very small networks.

Instead, we propose a new procedure for computing $\Kdt(G)$ using fluxes and the theory developed in \Cref{sec:G_max,sec:disg_flux,sec:topology}. The advantage is that we can now separate the problem into a linear part and a nonlinear part, which often proves to be computationally cheaper overall. 
We take the following three-step approach for an E-graph $G=(V,E)$.

\begin{enumerate}
    \item[1)] \textbf{\boldmath Calculate $G^{\max}$.}\\
    We iterate on the edges. For an edge $\yy \to \yy'$ of $G^{\comp}$, we check if there exists a weakly reversible E-graph $H=(V_H,E_H)\subseteq G^{\comp}$ such that $\yy \to \yy'\in E_H$ and for which $G$ admits a realization with respect to $H$. This is a linear feasibility problem, as it is enough to show that there exist $\vv\kk\in \rrpp^{|E|}$ and $\vv\lambda\in\rrpp^{|E_H|}$ that satisfy \eqref{eq:DE}, and check that $H$ is a weakly reversible graph; both are linear conditions.

    \item[2)] \textbf{\boldmath Calculate $\Fdt (G)$.}\\
    As in the proof of \Cref{lem:polyhedral_cones}, we build the polyhedral cone $\mathcal{X}(G)$ of pairs $(\vv \beta,\vv \gamma) \in \rrp^{|E|} \times \rrp^{|E^{\max}|}$ that satisfy \eqref{eq:DE-f} and \eqref{eq:VB-f} in \Cref{lem:semialgebraic_F} with $H=G^{\max}$. We then compute its projection to $\rrp^{|E|}$, e.g., via Fourier--Motzkin elimination \cite{ziegler:1995}*{Chapter 1}, and take its positive part.

    \item[3)] \textbf{\boldmath Calculate $\Kdt (G)$ from $\Fdt (G)$.}\\
    We calculate $\Kdt(G)$ from $\Fdt(G)$ by solving a nonlinear quantifier elimination problem. In particular, we want to find for which $\vv\kk \in \rrpp^{|E|}$ does there exist an $\xx \in \rrpp^n$ such that $\vv\beta = (\kk_{\yy\to\yy'}\xx^{\yy})_{\yy\to\yy'\in E}$ is in $\Fdt(G)$.
\end{enumerate}

The first step is a linear feasibility problem \cite{szederkenyi:hangos:tuza:2012}. The second step is a polyhedral geometry problem, and the third step is a quantifier elimination problem. Even when the last step is too computationally expensive, the intermediate object $\Fdt(G)$ can be very useful, for instance, for computing the dimension of $\Kdt(G)$ using \Cref{cor:Kdt_G}, or for {\em sampling} this space as shown in \Cref{subsec:4d}.

We use the above-sketched three-step approach to calculate the disguised toric loci of several examples in \Cref{sec:examples}. Our computations are performed in \texttt{Mathematica}, and the codes can be found in the GitHub repository \cite{boros_github}.

By \Cref{thm:G_max_enough_F}, when computing the disguised toric flux cone $\Fdt(G)$, it suffices to investigate the sets $\Fdt(G,H)$ for $H \subseteq G^{\max}$. In fact, we have seen in \Cref{thm:FdtG_is_clFdtGGmax} that $\Fdt(G)$ equals the positive part of the closure of $\Fdt(G,G^{\max})$. The consequence of the following lemma is that sometimes we can omit some edges of $G^{\max}$ to obtain a graph $H'$ for which $\Fdt(G,H')=\Fdt(G,G^{\max})$. The advantage is that we have to deal with a smaller-dimensional linear problem. We will make use of this in \Cref{subsec:LVA_rev,subsec:B-T_13,subsec:4d}.

\begin{lemma} 
\label{lem:rank-1}
Suppose an E-graph $H$ has three distinct vertices $\yy_1$, $\yy_2$, $\yy_3$ on a line (in this order, see \Cref{fig:H_and_H'}), and $\yy_1 \to \yy_3 \in H$. Let $H'$ be the E-graph that is obtained from $H$ by deleting $\yy_1 \to \yy_3$, and adding the three reactions $\yy_1 \rightleftarrows \yy_2 \to \yy_3$ (if they were not already in $H$).  Then, for any $\vv\gamma \in \Ft(H)$ there exists a $\vv\gamma' \in \Ft(H')$ such that $(H,\vv\gamma) \triangleq (H',\vv\gamma')$.
\end{lemma}
\begin{proof}
Let $a>0$ and $b>0$ denote the lengths of the vectors $\yy_2 - \yy_1$ and $\yy_3 - \yy_2$, respectively. With $E_{H'}$ denoting the edge set of $H'$, let us define $\vv\gamma'\in\rrpp^{|E_{H'}|}$ by
\begin{align*}
    \gamma'_{\yy_1 \to \yy_2} &= \gamma_{\yy_1 \to \yy_2} + \tfrac{a+b}{a}\gamma_{\yy_1 \to \yy_3}, \\
    \gamma'_{\yy_2 \to \yy_1} &= \gamma_{\yy_2 \to \yy_1} + \tfrac{b}{a}\gamma_{\yy_1 \to \yy_3}, \\
    \gamma'_{\yy_2 \to \yy_3} &= \gamma_{\yy_2 \to \yy_3} + \gamma_{\yy_1 \to \yy_3},
\end{align*}
and $\vv\gamma'$ agrees with $\vv\gamma$ on every edge different from these three. It is straightforward to check that $\vv\gamma'$ is indeed in $\Ft(H')$ and that $(H,\vv\gamma) \triangleq (H',\vv\gamma')$.
\end{proof}

\begin{figure}
    \centering
    \begin{tikzpicture}[scale=1]

\tikzset{mybullet/.style={inner sep=1.2pt,outer sep=2pt,draw,fill,Egraphcolor,circle}};
\tikzset{myarrow/.style={arrows={-stealth},very thick,Egraphcolor}};

\node[mybullet]  (0) at (0,0) {};
\node[mybullet]  (3) at (3,0) {};
\node[mybullet]  (5) at (5,0) {};

\draw [decorate,decoration={brace,amplitude=10pt,mirror}] (0.south east) -- (3.south west) node[midway,below=10pt] {$a$};
\draw [decorate,decoration={brace,amplitude=10pt,mirror}] (3.south east) -- (5.south west) node[midway,below=10pt] {$b$};

\node[yshift=-15pt] at (0) {$\yy_1$}; 
\node[yshift=-15pt] at (3) {$\yy_2$}; 
\node[yshift=-15pt] at (5) {$\yy_3$}; 

\draw[myarrow,transform canvas={yshift=4pt}]    (0) to node[above] {\fluxcolor{$\gamma_{\yy_1 \to \yy_3}$}} (5);
\node at (2.5,1.25) {$H$};

\node at (-0.75,0) {\phantom{0}};

\begin{scope}[shift={(7,0)}]
    
\node[mybullet]  (0) at (0,0) {};
\node[mybullet]  (3) at (3,0) {};
\node[mybullet]  (5) at (5,0) {};

\node[yshift=-15pt] at (0) {$\yy_1$}; 
\node[yshift=-15pt] at (3) {$\yy_2$}; 
\node[yshift=-15pt] at (5) {$\yy_3$}; 

\draw[myarrow,transform canvas={yshift=3pt}]  (0) to node[above] {\fluxcolor{$\tfrac{a+b}{a}\gamma_{\yy_1 \to \yy_3}$}} (3);
\draw[myarrow,transform canvas={yshift=-3pt}] (3) to node[below] {\fluxcolor{$\tfrac{b}{a}\gamma_{\yy_1 \to \yy_3}$}} (0);
\draw[myarrow]                                (3) to node[above] {\fluxcolor{$\gamma_{\yy_1 \to \yy_3}$}} (5);

\node at (2.5,1.25) {$H'$};

\node at (5.75,0) {\phantom{0}};

\end{scope}

\draw[rounded corners,lightgray] (current bounding box.south west) rectangle (current bounding box.north east);

\end{tikzpicture}
    \caption{Illustration of the idea of the proof of \Cref{lem:rank-1}.}
    \label{fig:H_and_H'}
\end{figure}

Note that \Cref{lem:rank-1} can equivalently be formulated in terms of rate constants instead of fluxes.
Systematically finding even smaller subgraphs $H'$ such that $\Fdt(G,H')=\Fdt(G,G^{\max})$ is an interesting direction for future work.
\section{Examples}
\label{sec:examples}

In this section, we analyze several examples from the reaction networks literature to illustrate the practical applicability of the theory presented in this paper.

We describe the E-graphs via drawing them in $\rr^n$, except in \Cref{subsec:4d}, where $n=4$. We number the edges of the E-graphs and then index the associated rate constants and flux values accordingly, e.g.,  $\kappa_i$ (or $\lambda_i$) and $\beta_i$ (or $\gamma_i$) denote the rate constant and the flux value on the $i$th edge, respectively. As in \Cref{lem:semialgebraic_F}, we denote by (DE-f) and (VB-f) the set of equations that describe dynamically equal fluxes and vertex-balanced fluxes, respectively. The calculations are available in the Mathematica Notebook \cite{boros_github}.

For the examples in \Cref{subsec:square_rev,subsec:square_parallelogram,subsec:LVA_rev,subsec:B-T_13,subsec:4d}, we find that the disguised toric locus is a full-dimensional semialgebraic set in $\rrpp^{|E|}$ that is a proper subset of $\rrpp^{|E|}$. In order to have an impression of their size, we provide a numerical estimate of the proportion of the simplex $\{\vv\kk \in \rrpp^{|E|} \st \sum_{i=1}^{|E|}\kk_i=1\}$ that belongs to these sets via sampling from a uniform distribution.

Before we dive into the examples, we recall that the \emph{deficiency} of an E-graph $G$ is a nonnegative integer (see \Cref{def:dfc} below). Provided $G$ is weakly reversible, the deficiency is precisely the codimension of the toric locus $\Kt(G)$, see \cite{craciun:dickenstein:shiu:sturmfels:2009}.
\begin{definition} \label{def:dfc}
For an E-graph $G = (V,E)$, its \emph{deficiency} is $\delta = |V| - \ell - \dim\mc{S}$, where $\ell$ is the number of connected components of $G$, and $\mc{S}$ is the stoichiometric subspace of $G$.
\end{definition}

\subsection{Reversible square}
\label{subsec:square_rev}

\begin{align}\label{tikz:square_rev}
\begin{aligned}
\begin{tikzpicture}[scale=2]

\tikzset{mybullet/.style={inner sep=1.2pt,outer sep=4pt,draw,fill,Egraphcolor,circle}};
\tikzset{myarrow/.style={arrows={-stealth},very thick,Egraphcolor}};

\draw [step=1, gray, very thin] (0,0) grid (1.25,1.25);
\draw [ -, black] (0,0)--(1.25,0);
\draw [ -, black] (0,0)--(0,1.25);

\node[mybullet]  (0) at (0,0) {};
\node[mybullet]  (X) at (1,0) {};
\node[mybullet] (XY) at (1,1) {};
\node[mybullet]  (Y) at (0,1) {};

\draw[myarrow,transform canvas={yshift=-2pt}]  (0) to node[below] {\tiny $1$} (X);
\draw[myarrow,transform canvas={yshift=2pt}]  (X) to node[above] {\tiny $5$} (0);
\draw[myarrow,transform canvas={xshift=2pt}]  (X) to node[right] {\tiny $2$} (XY);
\draw[myarrow,transform canvas={xshift=-2pt}]  (XY) to node[left] {\tiny $6$} (X);
\draw[myarrow,transform canvas={yshift=2pt}] (XY) to node[above] {\tiny $3$} (Y);
\draw[myarrow,transform canvas={yshift=-2pt}] (Y) to node[below] {\tiny $7$} (XY);
\draw[myarrow,transform canvas={xshift=-2pt}]  (Y) to node[left] {\tiny $4$} (0);
\draw[myarrow,transform canvas={xshift=2pt}]  (0) to node[right] {\tiny $8$} (Y);

\node[right] at (1.1,1.2) {$G$};

\node at (0.5,-0.75) {$\begin{aligned}
\frac{\mathrm{d}x_1}{\mathrm{d}t} &= \kappa_1   - \kappa_5 x_1  + \kappa_7 x_2 - \kappa_3 x_1 x_2 \\
\frac{\mathrm{d}x_2}{\mathrm{d}t} &= \kappa_2 x_1 - \kappa_6 x_1x_2 + \kappa_8   - \kappa_4 x_2
\end{aligned}$};
\node at (-1,0.5) {\phantom{0}};

\begin{scope}[shift={(2.5,0)}]

\draw [step=1, gray, very thin] (0,0) grid (1.25,1.25);
\draw [ -, black] (0,0)--(1.25,0);
\draw [ -, black] (0,0)--(0,1.25);

\node[mybullet]  (0) at (0,0) {};
\node[mybullet]  (X) at (1,0) {};
\node[mybullet] (XY) at (1,1) {};
\node[mybullet]  (Y) at (0,1) {};

\draw[myarrow,transform canvas={yshift=-2pt}]  (0) to node[below] {\tiny $1$} (X);
\draw[myarrow,transform canvas={yshift=2pt}]  (X) to node[above] {\tiny $5$} (0);
\draw[myarrow,transform canvas={xshift=2pt}]  (X) to node[right] {\tiny $2$} (XY);
\draw[myarrow,transform canvas={xshift=-2pt}]  (XY) to node[left] {\tiny $6$} (X);
\draw[myarrow,transform canvas={yshift=2pt}] (XY) to node[above] {\tiny $3$} (Y);
\draw[myarrow,transform canvas={yshift=-2pt}] (Y) to node[below] {\tiny $7$} (XY);
\draw[myarrow,transform canvas={xshift=-2pt}]  (Y) to node[left] {\tiny $4$} (0);
\draw[myarrow,transform canvas={xshift=2pt}]  (0) to node[right] {\tiny $8$} (Y);
\draw[myarrow,transform canvas={xshift=1.4,yshift=-1.4pt}]  (0) to node[pos=0.35, below] {\tiny $9$} (XY);
\draw[myarrow,transform canvas={xshift=-1.4pt,yshift=1.4pt}]  (XY) to node[pos=0.65, left] {\tiny $10$} (0);
\draw[myarrow,transform canvas={xshift=1.4pt, yshift=1.4pt}]  (X) to node[pos=0.35, right] {\tiny $11$} (Y);
\draw[myarrow,transform canvas={xshift=-1.4pt, yshift=-1.4pt}]  (Y) to node[pos=0.65, below] {\tiny $12$} (X);

\node[right] at (1.1,1.2) {$G^{\max}$};
\node at (1.75,0.5) {\phantom{0}};

\end{scope}

\draw[rounded corners,lightgray] (current bounding box.south west) rectangle (current bounding box.north east);

\end{tikzpicture}
\end{aligned}
\end{align}

For all $\vv \kappa \in \rrpp^8$, the mass-action differential equation associated with the reversible square $G$ (shown in \eqref{tikz:square_rev}) has a unique positive equilibrium, and this equilibrium is globally asymptotically stable. This follows by combining the Deficiency-One Theorem \cite{feinberg:1987}, permanence \cites{simon:1995, craciun:nazarov:pantea:2013, gopalkrishnan:miller:shiu:2014, boros:hofbauer:2020, anderson:cappelletti:kim:nguyen:2020}, and the exclusion of a periodic solution \cite{boros:hofbauer:2022}*{Section 4}. Further, the E-graph $G$ is known to have a full-dimensional disguised toric locus $\Kdt(G)$, but its explicit description has not been available \cite{craciun:deshpande:jin:2024a}*{Example 3.11}. Below, we calculate the disguised toric flux cone $\Fdt(G) = (\overline{\Fdt(G, G^{\max})})_{>0}$, which allows us to derive an explicit formula for $\Kdt(G)$. Recall that for the rate constants in $\Kdt(G)$, besides global asymptotic stability, it also follows that the Horn--Jackson function \eqref{eq:H-J} is a global Lyapunov function, providing us with a finer understanding of the dynamics.

The set of equilibrium fluxes is
\begin{align*}
\Feq(G) = \{\vv\beta \in \rrpp^8 \st \beta_1-\beta_5 = \beta_3 - \beta_7 \text{ and } \beta_2-\beta_6 = \beta_4 - \beta_8\}.    
\end{align*}
By definition, $\vv\beta \in \Fdt(G)$ if and only if there exists a $\vv \gamma \in \mathbb{R}^{12}_{\geq0}$ such that
\begin{gather}
\begin{alignedat}{4}
\beta_1&= \gamma_1 + \gamma_9,    & \qquad \beta_2&= \gamma_2 + \gamma_{11}, & \qquad \beta_3 &= \gamma_3 + \gamma_{10}, & \qquad \beta_4 &= \gamma_4 + \gamma_{12}, \\
\beta_5&= \gamma_5 + \gamma_{11}, & \qquad \beta_6&= \gamma_6 + \gamma_{10}, & \qquad \beta_7 &= \gamma_7 + \gamma_{12}, & \qquad \beta_8 &= \gamma_8 + \gamma_9;
\end{alignedat}
\tag{DE-f} \\
\begin{alignedat}{2}
  \gamma_1 + \gamma_8 + \gamma_9    &= \gamma_4 + \gamma_5 + \gamma_{10}, & \qquad \gamma_2 + \gamma_5 + \gamma_{11} &= \gamma_1 + \gamma_6 + \gamma_{12}, \\
  \gamma_3 + \gamma_6 + \gamma_{10} &= \gamma_2 + \gamma_7 + \gamma_9,    & \qquad \gamma_4 + \gamma_7 + \gamma_{12} &= \gamma_3 + \gamma_8 + \gamma_{11}.
\end{alignedat}
\tag{VB-f}
\end{gather}
Note that the way we handle \emph{all} subgraphs of $G^{\max}$ at once is that we allow some coordinates of $\vv\gamma$ to vanish (alternatively, we could require that $\vv\gamma \in \rrpp^{12}$ and get $\Fdt(G,G^{\max})$, and then $\Fdt(G)=(\overline{\Fdt(G,G^{\max})})_{>0}$, see \Cref{thm:FdtG_is_clFdtGGmax}).
Further, notice that any of the four equations on vertex balancing is a consequence of the other three. Additionally, under $\vv \beta \in \Feq(G)$, another two equations are redundant. Hence, we have nine independent linear equations. Solving those for $\gamma_1, \ldots, \gamma_9$ yields
\begin{align*}
\begin{alignedat}{3}
  \gamma_1 &= \beta_4 + \beta_5 - \beta_8 + \gamma_{10} - \gamma_{11} - \gamma_{12},           & \qquad \gamma_2 &= \beta_2 - \gamma_{11}, & \qquad \gamma_5 &= \beta_5 - \gamma_{11}, \\
  \gamma_8 &= \beta_4 + \beta_5 - \beta_1 + \gamma_{10} - \gamma_{11} - \gamma_{12},           & \qquad \gamma_3 &= \beta_3 - \gamma_{10}, & \qquad \gamma_6 &= \beta_6 - \gamma_{10}, \\
  \gamma_9 &= \beta_1 - \beta_5 - \beta_2 + \beta_6 - (\gamma_{10} - \gamma_{11} - \gamma_{12}), & \qquad \gamma_4 &= \beta_4 - \gamma_{12}, & \qquad \gamma_7 &= \beta_7 - \gamma_{12}.
\end{alignedat}
\end{align*}
Thus, a $\vv\beta \in \Feq(G)$ is in $\Fdt(G)$ if and only if there exist $\gamma_{10}, \gamma_{11}, \gamma_{12}$ such that
\begin{gather*}
0 \leq \gamma_{10} \leq \min(\beta_3, \beta_6), \quad 0 \leq \gamma_{11} \leq \min(\beta_2, \beta_5), \quad 0 \leq \gamma_{12} \leq \min(\beta_4, \beta_7), \\
\max(\beta_1,\beta_8) - \beta_4 - \beta_5 \leq \gamma_{10} - \gamma_{11} - \gamma_{12} \leq \beta_1 - \beta_5 -\beta_2 + \beta_6.
\end{gather*}
Consequently, a $\vv \beta \in \Feq(G)$ is in $\Fdt(G)$ if and only if
\begin{align}
\label{eq:square_rev_a}  \max(\beta_1, \beta_8) - \beta_4 - \beta_5      &\leq \min(\beta_3, \beta_6), \\
\label{eq:square_rev_b} -\min(\beta_2, \beta_5) - \min(\beta_4, \beta_7) &\leq \beta_1 - \beta_5 - \beta_2 + \beta_6. 
\end{align}
Provided $\vv \beta \in \Feq(G)$, one finds that the inequalities \eqref{eq:square_rev_a} and \eqref{eq:square_rev_b} are equivalent to
\begin{align}
\label{eq:square_rev_c}    (\beta_1 - \beta_8)(\beta_3 - \beta_6) &\leq (\beta_2 + \beta_5)(\beta_4 + \beta_7), \\
\label{eq:square_rev_d}    (\beta_2 - \beta_5)(\beta_4 - \beta_7) &\leq (\beta_1 + \beta_8)(\beta_3 + \beta_6).    
\end{align}
In fact, under $\vv \beta \in \Feq(G)$, formula \eqref{eq:square_rev_a} is equivalent to \eqref{eq:square_rev_c}, and \eqref{eq:square_rev_b} is equivalent to \eqref{eq:square_rev_d}. Hence,
\begin{align*}
    \Fdt(G) = \{\vv\beta \in \Feq(G) \st \text{\eqref{eq:square_rev_c} and \eqref{eq:square_rev_d} hold}\}.
\end{align*}
We emphasize that, although \eqref{eq:square_rev_c} and \eqref{eq:square_rev_d} describe a nonlinear object in $\rrpp^8$, its intersection with $\Feq(G)$ is linear.

It remains to find the disguised toric locus $\Kdt(G)$ (which is homeomorphic to $\rrpp^2 \times \Fdt(G)$, see \Cref{thm:homeo_Psi}). In other words, we aim to figure out for which $\vv \kappa \in \Keq(G) = \rrpp^8$ does there exist an equilibrium $\xx^* \in \rrpp^2$ for which
\begin{align*}
\begin{aligned}
    (\kappa_1 - \kappa_8)(\kappa_3 x^*_1 x^*_2 - \kappa_6 x^*_1 x^*_2) &\leq (\kappa_2 x^*_1 + \kappa_5 x^*_1)(\kappa_4 x^*_2 + \kappa_7 x^*_2), \\
    (\kappa_2 x^*_1 - \kappa_5 x^*_1)(\kappa_4 x^*_2 - \kappa_7 x^*_2) &\leq (\kappa_1 + \kappa_8)(\kappa_3 x^*_1 x^*_2 + \kappa_6 x^*_1 x^*_2).
\end{aligned} 
\end{align*}
Since both inequalities can be simplified by $x^*_1x^*_2$, it follows that
\begin{align*}
\Kdt(G)=\left\{\vv\kk \in \rrpp^8 \;\middle|\; \begin{aligned}(\kappa_1 - \kappa_8)(\kappa_3 - \kappa_6) &\leq (\kappa_2 + \kappa_5)(\kappa_4 + \kappa_7),\\
(\kappa_2 - \kappa_5)(\kappa_4 - \kappa_7) &\leq (\kappa_1 + \kappa_8)(\kappa_3 + \kappa_6)\end{aligned}\right\}. 
\end{align*}

Numerical simulation shows that approximately 83.3\% of the simplex $\{\vv\kk \in \rrpp^8 \st \sum_{i=1}^8 \kk_i = 1\}$ belongs to $\Kdt(G)$. Observe that, despite $(G,\vv\kappa)$ being globally asymptotically stable for all $\vv\kappa \in \rrpp^8$, we have found that $\Kdt(G) \subsetneq \rrpp^8$. We shed more light on this outcome via the examples in \Cref{subsec:square_parallelogram}.

Finally, we remark that the toric flux cone $\Ft(G)$ and the toric locus $\Kt(G)$ are given by
\begin{align*}
    \Ft(G) &= \{ \vv \beta \in \Feq(G) \st \beta_1 - \beta_5 = \beta_2 - \beta_6 \}, \\
    \Kt(G) &= \{ \vv \kappa \in \rrpp^8 \st K_1 K_3 = K_2 K_4\},
\end{align*}
where
\begin{alignat*}{2}
    K_1 &= \kappa_3 \kappa_4 (\kappa_2 + \kappa_5) + \kappa_5 \kappa_6 (\kappa_4 + \kappa_7), & \quad \quad K_2 &= \kappa_6 \kappa_7 (\kappa_1 + \kappa_8) + \kappa_1 \kappa_4 (\kappa_3 + \kappa_6), \\
    K_3 &= \kappa_7 \kappa_8 (\kappa_2 + \kappa_5) + \kappa_1 \kappa_2 (\kappa_4 + \kappa_7), & \quad \quad K_4 &= \kappa_2 \kappa_3 (\kappa_1 + \kappa_8) + \kappa_5 \kappa_8 (\kappa_3 + \kappa_6).
\end{alignat*}
The formula in $\Kt(G)$ was obtained by the application of the matrix-tree theorem \cite{craciun:dickenstein:shiu:sturmfels:2009}. Since the deficiency of $G$ is one, the toric locus $\Kt(G)$ is a codimension-one set in $\rrpp^8$. For comparison, the disguised toric locus $\Kdt(G)$ is a codimension-zero set.

\subsection{Square vs.\ parallelogram}
\label{subsec:square_parallelogram}

\begin{align}\label{tikz:square_parallelogram}
\begin{aligned}
\begin{tikzpicture}[scale=2]

\tikzset{mybullet/.style={inner sep=1.2pt,outer sep=4pt,draw,fill,Egraphcolor,circle}};
\tikzset{myarrow/.style={arrows={-stealth},very thick,Egraphcolor}};

\draw [step=1, gray, very thin] (0,0) grid (1.25,1.25);
\draw [ -, black] (0,0)--(1.25,0);
\draw [ -, black] (0,0)--(0,1.25);

\node[mybullet]  (0) at (0,0) {};
\node[mybullet]  (X) at (1,0) {};
\node[mybullet] (XY) at (1,1) {};
\node[mybullet]  (Y) at (0,1) {};

\draw[myarrow]                                 (0) to node[below] {\tiny $1$} (X);
\draw[myarrow,transform canvas={xshift=2pt}]   (X) to node[right] {\tiny $2$} (XY);
\draw[myarrow,transform canvas={xshift=-2pt}] (XY) to node[left]  {\tiny $5$} (X);
\draw[myarrow]                                (XY) to node[above] {\tiny $3$} (Y);
\draw[myarrow,transform canvas={xshift=-2pt}]  (Y) to node[left]  {\tiny $4$} (0);
\draw[myarrow,transform canvas={xshift=2pt}]   (0) to node[right] {\tiny $6$} (Y);

\node[right] at (1,1.2) {$G_1$};

\node[right] at (-0.3,-.85) {$\begin{aligned}
\frac{\mathrm{d}x_1}{\mathrm{d}t} &= \kappa_1 - \kappa_3 x_1x_2 \\
\frac{\mathrm{d}x_2}{\mathrm{d}t} &= \kappa_2 x_1 - \kappa_5 x_1x_2  + \kappa_6 - \kappa_4 x_2 
\end{aligned}$};

\begin{scope}[shift={(1.93,0)}]

\draw [step=1, gray, very thin] (0,0) grid (1.25,1.25);
\draw [ -, black] (0,0)--(1.25,0);
\draw [ -, black] (0,0)--(0,1.25);

\node[mybullet]  (0) at (0,0) {};
\node[mybullet]  (X) at (1,0) {};
\node[mybullet] (XY) at (1,1) {};
\node[mybullet]  (Y) at (0,1) {};

\draw[myarrow]                                 (0) to node[below] {\tiny $1$} (X);
\draw[myarrow,transform canvas={xshift=2pt}]   (X) to node[right] {\tiny $2$} (XY);
\draw[myarrow,transform canvas={xshift=-2pt}] (XY) to node[left] {\tiny $5$} (X);
\draw[myarrow]                                (XY) to node[above] {\tiny $3$} (Y);
\draw[myarrow,transform canvas={xshift=-2pt}]  (Y) to node[left] {\tiny $4$} (0);
\draw[myarrow,transform canvas={xshift=2pt}]   (0) to node[right] {\tiny $6$} (Y);
\draw[myarrow,transform canvas={xshift=1.4,yshift=-1.4pt}]  (0) to node[pos=0.5, below] {\tiny $7$} (XY);
\draw[myarrow,transform canvas={xshift=-1.4pt,yshift=1.4pt}]  (XY) to node[pos=0.5, left] {\tiny $8$} (0);

\node[right] at (1,1.2) {$G_1^{\max}$};

\end{scope}

\begin{scope}[shift={(3.86,-2)}]

\draw [step=1, gray, very thin] (0,0) grid (1.25,3.25);
\draw [ -, black] (0,0)--(1.25,0);
\draw [ -, black] (0,0)--(0,3.25);

\node[mybullet]   (Y) at (0,1) {};
\node[mybullet]   (X) at (1,0) {};
\node[mybullet] (X2Y) at (1,2) {};
\node[mybullet]  (3Y) at (0,3) {};

\draw[myarrow]                                 (Y) to node[below] {\tiny $1$} (X);
\draw[myarrow,transform canvas={xshift=2pt}]   (X) to node[right] {\tiny $2$} (X2Y);
\draw[myarrow,transform canvas={xshift=-2pt}] (X2Y) to node[left] {\tiny $5$} (X);
\draw[myarrow]                                (X2Y) to node[above] {\tiny $3$} (3Y);
\draw[myarrow,transform canvas={xshift=-2pt}]  (3Y) to node[left] {\tiny $4$} (Y);
\draw[myarrow,transform canvas={xshift=2pt}]   (Y) to node[right] {\tiny $6$} (3Y);

\node[right] at (1,3.2) {$G_2$};

\node[right] at (-2,0) {$\begin{aligned}
\frac{\mathrm{d}x_1}{\mathrm{d}t} &= \kappa_1 x_2 - \kappa_3 x_1x_2^2 \\
\frac{\mathrm{d}x_2}{\mathrm{d}t} &= - \kappa_1 x_2 + \kappa_3 x_1x_2^2 +2(\kappa_2 x_1 - \kappa_5 x_1 x_2^2  + \kappa_6 x_2 - \kappa_4 x_2^3) 
\end{aligned}$};

\end{scope}

\begin{scope}[shift={(5.79,-2)}]

\draw [step=1, gray, very thin] (0,0) grid (1.25,3.25);
\draw [ -, black] (0,0)--(1.25,0);
\draw [ -, black] (0,0)--(0,3.25);

\node[mybullet]   (Y) at (0,1) {};
\node[mybullet]   (X) at (1,0) {};
\node[mybullet] (X2Y) at (1,2) {};
\node[mybullet]  (3Y) at (0,3) {};

\draw[myarrow]                                 (Y) to node[below] {\tiny $1$} (X);
\draw[myarrow,transform canvas={xshift=2pt}]   (X) to node[right] {\tiny $2$} (X2Y);
\draw[myarrow,transform canvas={xshift=-2pt}] (X2Y) to node[left] {\tiny $5$} (X);
\draw[myarrow]                                (X2Y) to node[above] {\tiny $3$} (3Y);
\draw[myarrow,transform canvas={xshift=-2pt}]  (3Y) to node[left] {\tiny $4$} (Y);
\draw[myarrow,transform canvas={xshift=2pt}]   (Y) to node[right] {\tiny $6$} (3Y);
\draw[myarrow,transform canvas={xshift=1.4,yshift=-1.4pt}]  (Y) to node[pos=0.5, below] {\tiny $7$} (X2Y);
\draw[myarrow,transform canvas={xshift=-1.4pt,yshift=1.4pt}]  (X2Y) to node[pos=0.5, left] {\tiny $8$} (Y);

\node[right] at (1,3.2) {$G_2^{\max}$};

\end{scope}

\draw[rounded corners,lightgray] (current bounding box.south west) rectangle (current bounding box.north east);

\end{tikzpicture}
\end{aligned}
\end{align}

For all $\vv \kappa \in \rrpp^6$, each of the mass-action differential equations associated to $G_1$ and $G_2$ (shown in \eqref{tikz:square_parallelogram}) has a unique positive equilibrium \cite{feinberg:1987}. This equilibrium is globally asymptotically stable for the square $G_1$ (by a similar argument as for the reversible square $G$ in \eqref{tikz:square_rev}). However, the parallelogram $G_2$ admits a supercritical Bautin bifurcation \cite{boros:hofbauer:2022}. Thus, it admits bistability in the following sense: there exists a $\vv \kappa \in \rrpp^6$ for which the asymptotically stable positive equilibrium of $(G_2,\vv\kappa)$ is surrounded by an asymptotically stable limit cycle (and the two stable objects are separated by a repelling limit cycle). At the same time, the E-graph $G_2$ can be obtained by applying an affine transformation to the E-graph $G_1$, and hence, $\Kt(G_1)=\Kt(G_2)$ and $\Kdt(G_1)=\Kdt(G_2)$, see \cite{haque:satriano:sorea:yu:2023}*{Theorem 3.8}. Consequently, although $(G_1,\vv\kappa)$ is globally asymptotically stable for all $\vv\kappa \in \rrpp^6$, the disguised toric locus $\Kdt(G_1)$ is a proper subset of $\rrpp^6$ because there exists an affine transformation of $G_1$ that is not globally asymptotically stable for all rate constants. Referring back to \Cref{subsec:square_rev}, the same reason prevents the reversible square from being disguised vertex-balanced for all rate constants.

Below, we calculate the disguised toric flux cone $\Fdt(G_i)$, which allows us to derive an explicit formula for $\Kdt(G_i)$ for $i=1,2$. The set of equilibrium fluxes is
\begin{align*}
\Feq(G_i) = \{\vv\beta \in \rrpp^6 \st \beta_1=\beta_3 \text{ and } \beta_2+\beta_6 = \beta_4 + \beta_5\}.    
\end{align*}
By definition, a $\vv\beta \in \Feq(G_i)$ is in $\Fdt(G_i)$ if and only if there exists a $\vv\gamma \in \mathbb{R}^{8}_{\geq0}$ such that
\begin{gather}
\beta_1 = \gamma_1 + \gamma_7, \quad \beta_2 = \gamma_2, \quad \beta_3 = \gamma_3 + \gamma_8, \quad \beta_4 = \gamma_4, \quad \beta_5 = \gamma_5 + \gamma_8, \quad \beta_6 = \gamma_6 + \gamma_7;
\tag{DE-f} \\
\gamma_4 +\gamma_8= \gamma_1 + \gamma_6 + \gamma_7, \qquad \gamma_1 + \gamma_5 = \gamma_2, \qquad \gamma_2 + \gamma_7 = \gamma_3 + \gamma_5 + \gamma_8, \qquad \gamma_3 + \gamma_6 = \gamma_4.
\tag{VB-f}
\end{gather}
As in \Cref{subsec:square_rev} above, provided $\vv\beta \in \Feq(G_i)$, three of the ten equations are redundant. Solving the seven independent linear equations for $\gamma_1, \ldots, \gamma_7$ yields
\begin{align*}
\begin{alignedat}{4}
\gamma_1 &=  \beta_4 - \beta_6 + \gamma_8, & \qquad \gamma_2 &= \beta_2, & \qquad \gamma_3 &= \beta_1 - \gamma_8, & \qquad \gamma_7 &=  \beta_1 + \beta_6 - \beta_4 - \gamma_8, \\
\gamma_6 &= \beta_4 - \beta_1 + \gamma_8, & \qquad \gamma_4 &= \beta_4, & \qquad \gamma_5 &= \beta_5 - \gamma_8. &&    
\end{alignedat}
\end{align*}
Thus, a $\vv \beta \in \Feq(G_i)$ is in $\Fdt(G_i)$ if and only if there exists a $\gamma_8\geq0$ such that
\begin{align*}
    \max(\beta_1,\beta_6) -\beta_4 \leq \gamma_8 \leq \min(\beta_1, \beta_5, \beta_1 + \beta_6 - \beta_4).
\end{align*}
A brief calculation then yields
\begin{align*}
    \Fdt(G_i) = \{\vv\beta \in \Feq(G_i) \st |\beta_4 - \beta_6| \leq \beta_1 \leq \beta_4 + \beta_5 \},
\end{align*}
see \cite{boros_github}. Finally, one finds that
\begin{align*}
    \Kdt(G_i) = \{\vv \kappa \in \rrpp^6 \st (1-\tfrac{\kappa_6}{\kappa_1}) (1-\tfrac{\kappa_5}{\kappa_3}) \leq \tfrac{\kappa_2\kappa_4}{\kappa_1\kappa_3} \leq (1+\tfrac{\kappa_6}{\kappa_1}) (1+\tfrac{\kappa_5}{\kappa_3})\}.
\end{align*}
Hence, the codimension of the semialgebraic set $\Kdt(G_i)$ in $\rrpp^6$ is zero. In fact, numerical simulation shows that approximately 58.3\% of the simplex $\{\vv\kk \in \rrpp^6 \st \sum_{i=1}^6 \kk_i = 1\}$ belongs to $\Kdt(G)$. For comparison, we remark that the toric locus $\Kt(G_i)$ and the toric flux cone $\Ft(G_i)$ are given by
\begin{align*}
    \Ft(G_i) &= \{ \beta \in \Feq(G_i) \st \beta_1 = \beta_4 - \beta_6 \},\\
    \Kt(G_i) &= \{ \kappa \in \rrpp^6 \st \tfrac{\kappa_2\kappa_4}{\kappa_1\kappa_3} = (1+\tfrac{\kappa_6}{\kappa_1}) (1+\tfrac{\kappa_5}{\kappa_3})\},
\end{align*}
i.e., $\Ft(G_i)$ is a codimension-one cone in $\Feq(G_i)$ and $\Kt(G_i)$ is a codimension-one semialgebraic set in $\rrpp^6$, this is in line with the fact that the deficiency of $G_i$ equals one.

We conclude this subsection by remarking that the theory developed in this paper also applies to E-graphs whose stoichiometric subspace is a proper subspace of $\rr^n$ (here, $\rr^n$ is where the vertices live). For instance, consider the lifted parallelogram $G_3$ (shown in \eqref{tikz:parallelogram_lifted}), and note that $x_1+x_2+2x_3$ is conserved. The E-graph $G_3$ admits bistability that is inherited from the planar parallelogram, see \cite{boros:hofbauer:2022}*{Section 3.2}, \cite{banaji:boros:hofbauer:2022}*{Theorem 1}, or \cite{banaji:boros:hofbauer:2025}*{Theorem 3.2}. Since the lifted parallelogram $G_3$ is obtained by an affine transformation from the planar parallelogram $G_2$, its disguised toric locus $\Kdt(G_3)$ equals $\Kdt(G_2)$, see \cite{haque:satriano:sorea:yu:2023}.
\begin{align} \label{tikz:parallelogram_lifted}
\begin{aligned}
\begin{tikzpicture}[scale=1.5]

\begin{axis}[view={150}{35}, axis line style=white, width=4.5cm, height=4.5cm, ticks=none, xmin=0, xmax=3.3, ymin=0, ymax=3.3, zmin=0, zmax=3.3]

    \tikzset{bullet/.style={inner sep=1pt,outer sep=1.5pt,draw,fill,Egraphcolor,circle}};
    \tikzset{myarrow/.style={arrows={-stealth},thick,Egraphcolor}};
    \tikzset{mygrid/.style={very thin,gray!50}};
    \tikzset{myaxis/.style={thick,gray}};    
    \newcommand{\ab}{3.2};    
 
    \draw[mygrid] (1,0,\ab) -- (1,0,0) -- (1,\ab,0);
    \draw[mygrid] (2,0,\ab) -- (2,0,0) -- (2,\ab,0);
    \draw[mygrid] (3,0,\ab) -- (3,0,0) -- (3,\ab,0);
    \draw[mygrid] (0,1,\ab) -- (0,1,0) -- (\ab,1,0);
    \draw[mygrid] (0,2,\ab) -- (0,2,0) -- (\ab,2,0);
    \draw[mygrid] (0,3,\ab) -- (0,3,0) -- (\ab,3,0);
    \draw[mygrid] (0,\ab,1) -- (0,0,1) -- (\ab,0,1);
    \draw[mygrid] (0,\ab,2) -- (0,0,2) -- (\ab,0,2);
    \draw[mygrid] (0,\ab,3) -- (0,0,3) -- (\ab,0,3);
        
    \draw[myaxis] (0,0,0) -- (\ab,0,0);    
    \draw[myaxis] (0,0,0) -- (0,\ab,0);    
    \draw[myaxis] (0,0,0) -- (0,0,\ab);    
    
    \node[bullet] (P1) at (0,1,1) {};
    \node[bullet] (P2) at (1,0,1) {};
    \node[bullet] (P3) at (1,2,0) {};
    \node[bullet] (P4) at (0,3,0) {};
    
    \node at (1.5,0,2.5) {\footnotesize $G_3$};
    
    \draw[myarrow]  (P1) to node[yshift=4pt] {\tiny $1$} (P2);
    
    \draw[myarrow,transform canvas={xshift=-1pt,yshift=-1pt}] (P2) to node[xshift=-4pt] {\tiny $2$} (P3);
    \draw[myarrow,transform canvas={xshift=1pt,yshift=1pt}] (P3) to node[xshift=4pt] {\tiny $5$} (P2);
    
    \draw[myarrow] (P3) to node[yshift=-4pt] {\tiny $3$} (P4);

    \draw[myarrow,transform canvas={xshift=1pt,yshift=1pt}] (P4) to node[xshift=4pt] {\tiny $4$} (P1);
    \draw[myarrow,transform canvas={xshift=-1pt,yshift=-1pt}]   (P1) to node[xshift=-4pt] {\tiny $6$} (P4);

\end{axis}

\node at (6.54,1.5) {$\begin{aligned}
\frac{\mathrm{d}x_1}{\mathrm{d}t} &=  \kappa_1 x_2 x_3 - \kappa_3 x_1 x_2^2 \\
\frac{\mathrm{d}x_2}{\mathrm{d}t} &= -\kappa_1 x_2 x_3 + \kappa_3 x_1 x_2^2 +2(\kappa_2 x_1 x_3 - \kappa_5 x_1 x_2^2 + \kappa_6 x_2 x_3 - \kappa_4 x_2^3) \\
\frac{\mathrm{d}x_3}{\mathrm{d}t} &= -\kappa_2 x_1 x_3 + \kappa_5 x_1 x_2^2 + \kappa_4 x_2^3 - \kappa_6 x_2 x_3
\end{aligned}$};

\draw[rounded corners,lightgray] (current bounding box.south west) rectangle (current bounding box.north east);
\end{tikzpicture}
\end{aligned}
\end{align}

\subsection{Reversible Lotka--Volterra autocatalator}
\label{subsec:LVA_rev}

\begin{align} \label{tikz:LVA_rev}
\begin{aligned}
\begin{tikzpicture}[scale=1.2]

\tikzset{mybullet/.style={inner sep=1.2pt,outer sep=2pt,draw,fill,Egraphcolor,circle}};
\tikzset{myarrow/.style={arrows={-stealth},very thick,Egraphcolor}};

\draw [step=1, gray, very thin] (0,0) grid (3.25,2.25);
\draw [ -, black] (0,0)--(3.25,0);
\draw [ -, black] (0,0)--(0,2.25);

\node[mybullet]  (0) at (0,0) {};
\node[mybullet]  (Y) at (0,1) {};
\node[mybullet] (2X) at (2,0) {};
\node[mybullet] (XY) at (1,1) {};
\node[mybullet] (2Y) at (0,2) {};
\node[mybullet] (3X) at (3,0) {};

\draw[myarrow,transform canvas={yshift= 2pt}] (2X) to node[above] {\tiny $1$} (3X);
\draw[myarrow,transform canvas={yshift=-2pt}] (3X) to node[below] {\tiny $2$} (2X);
\draw[myarrow,transform canvas={xshift=-1.4pt,yshift=-1.4pt}] (XY) to node[below] {\tiny $3$} (2Y);
\draw[myarrow,transform canvas={xshift= 1.4pt,yshift= 1.4pt}] (2Y) to node[above] {\tiny $4$} (XY);
\draw[myarrow,transform canvas={xshift= 2pt}] (Y) to node[right] {\tiny $5$} (0);
\draw[myarrow,transform canvas={xshift=-2pt}] (0) to node[left] {\tiny $6$} (Y);
\node at (2.5,1.5) {$G$};

\node at (1.5,-1) {$\begin{aligned}
\frac{\mathrm{d}x_1}{\mathrm{d}t} &= \kappa_1 x_1^2 - \kappa_2 x_1^3 - \kappa_3 x_1 x_2 + \kappa_4 x_2^2 \\
\frac{\mathrm{d}x_2}{\mathrm{d}t} &= \kappa_3 x_1 x_2  - \kappa_4 x_2^2 - \kappa_5 x_2 + \kappa_6 
\end{aligned}$};

\begin{scope}[shift={(4,0)}]
    
\draw [step=1, gray, very thin] (0,0) grid (3.25,2.25);
\draw [ -, black] (0,0)--(3.25,0);
\draw [ -, black] (0,0)--(0,2.25);

\node[mybullet]  (0) at (0,0) {};
\node[mybullet]  (Y) at (0,1) {};
\node[mybullet] (2X) at (2,0) {};
\node[mybullet] (XY) at (1,1) {};
\node[mybullet] (2Y) at (0,2) {};
\node[mybullet] (3X) at (3,0) {};

\draw[myarrow,transform canvas={yshift= 2pt}] (2X) to node[xshift=-5pt,yshift=4pt] {\tiny $1$} (3X);
\draw[myarrow,transform canvas={yshift=-2pt}] (3X) to node[below] {\tiny $2$} (2X);
\draw[myarrow,transform canvas={xshift=-1.4pt,yshift=-1.4pt}] (XY) to node[below] {\tiny $3$} (2Y);
\draw[myarrow,transform canvas={xshift= 1.4pt,yshift= 1.4pt}] (2Y) to node[above] {\tiny $4$} (XY);
\draw[myarrow,transform canvas={xshift= 2pt}] (Y) to node[xshift=4pt] {\tiny $5$} (0);
\draw[myarrow,transform canvas={xshift=-2pt}] (0) to node[left] {\tiny $6$} (Y);
\draw[myarrow,transform canvas={xshift= 2pt}] (2Y) to node[xshift=4pt] {\tiny $7$} (Y);
\draw[myarrow,transform canvas={xshift=-2pt}] (Y) to node[left] {\tiny $8$} (2Y);

\draw[myarrow] (XY) to node[yshift=-3pt] {\tiny $14$} (Y);
\draw[myarrow] (XY) to node[right] {\tiny $13$} (0);
\draw[myarrow] (XY) to node[left]  {\tiny $12$} (2X);
\draw[myarrow] (XY) to node[yshift=-4pt] {\tiny $11$} (3X);
\draw[myarrow] (2X) to node[below] {\tiny $9$} (0);
\draw[myarrow] (2Y) to node[above] {\tiny $10$} (3X);

\node at (2.5,1.5) {$H_1$};

\end{scope}

\begin{scope}[shift={(8,0)}]
    
\draw [step=1, gray, very thin] (0,0) grid (3.25,2.25);
\draw [ -, black] (0,0)--(3.25,0);
\draw [ -, black] (0,0)--(0,2.25);

\node[mybullet]  (0) at (0,0) {};
\node[mybullet]  (Y) at (0,1) {};
\node[mybullet] (2X) at (2,0) {};
\node[mybullet] (XY) at (1,1) {};
\node[mybullet] (2Y) at (0,2) {};
\node[mybullet] (3X) at (3,0) {};

\draw[myarrow,transform canvas={yshift= 2pt}] (2X) to node[above] {\tiny \fluxcolor{$\beta_1+\varepsilon$}} (3X);
\draw[myarrow,transform canvas={yshift=-2pt}] (3X) to node[below] {\tiny \fluxcolor{$\beta_1+\varepsilon$}} (2X);
\draw[myarrow,transform canvas={xshift=-1.4pt,yshift=-1.4pt}] (XY) to node[pos=0.4, below, sloped] {\tiny \fluxcolor{$\beta_3+\frac{\varepsilon}{2}$}} (2Y);
\draw[myarrow,transform canvas={xshift= 1.4pt,yshift= 1.4pt}] (2Y) to node[above, sloped] {\tiny \fluxcolor{$\beta_3+\varepsilon$}} (XY);
\draw[myarrow,transform canvas={xshift= 2pt}] (Y) to node[right] {\tiny \fluxcolor{$\beta_5+\frac{\varepsilon}{2}$}} (0);
\draw[myarrow,transform canvas={xshift=-2pt}] (0) to node[left] {\tiny \fluxcolor{$\beta_5+\varepsilon$}} (Y);
\draw[myarrow] (Y) to node[left] {\tiny \fluxcolor{$\frac{\varepsilon}{2}$}} (2Y);

\draw[myarrow] (XY) to node[xshift=4pt,yshift=4pt] {\tiny \fluxcolor{$\frac{\varepsilon}{2}$}} (2X);
\draw[myarrow] (2X) to node[below] {\tiny \fluxcolor{$\frac{\varepsilon}{2}$}} (0);

\node at (2.5,1.5) {$(H_2,\fluxcolor{\vv\gamma})$};
\node at (3.5,2.2) {\phantom{0}};
\end{scope}

\draw[rounded corners,lightgray] (current bounding box.south west) rectangle (current bounding box.north east);

\end{tikzpicture}
\end{aligned}
\end{align}

The E-graph $G$ shown in \eqref{tikz:LVA_rev} is known as the reversible Lotka--Volterra autocatalator (LVA) \cite{simon:1992}. By a direct analysis, Simon proved that, provided $\frac{\kappa_2\kappa_4\kappa_6}{\kappa_1\kappa_3\kappa_5} > 1$ holds, there is a unique positive equilibrium, and this equilibrium is globally asymptotically stable. The analysis of the associated mass-action differential equation is much more delicate when $\frac{\kappa_2\kappa_4\kappa_6}{\kappa_1\kappa_3\kappa_5} < 1$. Indeed, the E-graph that is obtained from $G$ by omitting the reactions $3\X_1 \to 2\X_1$ and $2\X_2 \to \X_1+\X_2$ is known to admit a supercritical Bogdanov--Takens bifurcation (and hence, a fold bifurcation of equilibria, a supercritical Andronov--Hopf bifurcation, and even a homoclinic bifurcation), see \cite{banaji:boros:hofbauer:2024b}*{Theorem 33}; and it follows from \cite{banaji:boros:hofbauer:2025}*{Theorem 3.2} that $G$ inherits these bifurcations. Furthermore, $G$ admits 5 positive equilibria: e.g.,  there are 3 sinks and 2 saddles when $(\kappa_1, \kappa_2, \kappa_3, \kappa_4, \kappa_5, \kappa_6) = (2, \frac{1}{10}, 3, 1, 3, 1)$.

Below, we calculate the disguised toric flux cone $\Fdt(G)$ and the disguised toric locus $\Kdt(G)$. What we find can be seen as an alternative proof for Simon's global stability result. A short calculation shows that 
\begin{alignat*}{2}
    \Feq(G) &= \{\vv\beta \in \rrpp^6 \st \beta_2-\beta_1 = \beta_4-\beta_3 = \beta_6-\beta_5\}, & \quad\quad \Keq(G) &= \rrpp^6, \\
     \Ft(G) &= \{\vv\beta \in \Feq(G) \st \beta_2-\beta_1 = 0\}, & \quad\quad \Kt(G) &= \{\vv\kappa \in \Keq(G)\st \tfrac{\kappa_2\kappa_4\kappa_6}{\kappa_1\kappa_3\kappa_5} = 1 \}.
\end{alignat*}
Next, we compute $\Fdt(G)$. Consider the graph $H_1$ in \eqref{tikz:LVA_rev}, and notice that it is obtained from the maximal graph $G^{\max}$ by omitting the reactions $3\X_1 \to \0$, $2\X_2 \to 2\X_1$, $\0 \rightleftarrows 2\X_2$. By \Cref{lem:rank-1}, we have $\Fdt(G,G^{\max}) = \Fdt(G,H_1)$, and hence, by \Cref{thm:FdtG_is_clFdtGGmax}, we find that $\Fdt(G) = (\overline{\Fdt(G, H_1)})_{>0}$. It is straightforward to see that for any $\vv\beta \in \rrpp^{6}$ and for any $\vv\gamma \in \mathbb{R}^{14}_{\geq0}$ for which $(G,\vv\beta)\triangleq (H_1,\vv\gamma)$ holds, we have $\beta_1 = \gamma_1 - 2 \gamma_9$ and $\beta_2 = \gamma_2$. Further, assuming that $(H_1, \vv\gamma)$ is vertex-balanced, we have $\gamma_2 = \gamma_1 + \gamma_{10} + \gamma_{11}$, which implies $\beta_2 - \beta_1 = 2\gamma_9 + \gamma_{10} + \gamma_{11} \geq 0$. Hence, 
\begin{align*}
\Fdt(G) \subseteq \{\vv\beta \in \Feq(G) \st \beta_2 - \beta_1 \geq 0\}.
\end{align*}
On the other hand, assuming $\vv\beta \in \Feq(G)$ fulfills $\beta_2-\beta_1>0$ and letting $\varepsilon = \beta_2 - \beta_1$, the flux $\vv\gamma \in \rrpp^9$ on $H_2$ that is shown in magenta in \eqref{tikz:LVA_rev} is a vertex-balanced flux for which $(G,\vv\beta)\triangleq (H_2,\vv\gamma)$ holds. Consequently,
\begin{align*}
    \Fdt(G) = \{\vv\beta \in \Feq(G) \st \beta_2 - \beta_1 \geq 0\}.
\end{align*}
Next, we find an explicit description of $\Kdt(G)$. For $\vv\kappa \in \rrpp^6$ and an equilibrium $\xx^* \in \rrpp^2$ of $(G,\vv\kappa)$, let
\begin{align*}
    a^*_{\vv\kappa} = (\kappa_2 x_1^* - \kappa_1)(x_1^*)^2 = (\kappa_4 x_2^* - \kappa_3 x_1^*)x_2^* = \kappa_6 - \kappa_5 x_2^*.
\end{align*}
It is straightforward to see that
\begin{align*}
    \sgn a^*_{\vv\kappa} = \sgn ( \tfrac{\kappa_2\kappa_4\kappa_6}{\kappa_1\kappa_3\kappa_5} - 1 ),
\end{align*}
and therefore,
\begin{align*}
    \Kdt(G) = \{\vv\kappa \in \rrpp^6 \st \tfrac{\kappa_2\kappa_4\kappa_6}{\kappa_1\kappa_3\kappa_5} \geq 1 \}.
\end{align*}
Hence, we have revealed that all those systems for which Simon proved global asymptotic stability in \cite{simon:1992} are disguised vertex-balanced systems, implying that the Horn--Jackson function \eqref{eq:H-J} is a global Lyapunov function (while the approach by Simon is independent of the Horn--Jackson function). Since planar reversible systems are permanent \cites{simon:1995, craciun:nazarov:pantea:2013}, global asymptotic stability of the unique positive equilibrium follows for every $\vv\kappa \in \Kdt(G)$.

Finally, notice that precisely 50\% of the simplex $\{\vv\kk \in \rrpp^6 \st \sum_{i=1}^6 \kk_i = 1\}$ belongs to $\Kdt(G)$.

\subsection{A network with a subcritical Bogdanov--Takens bifurcation}
\label{subsec:B-T_13}

\begin{align} \label{tikz:B-T_13}
\begin{aligned}
\begin{tikzpicture}[scale=1.2]

\tikzset{mybullet/.style={inner sep=1.2pt,outer sep=3.5pt,draw,fill,Egraphcolor,circle}};
\tikzset{myarrow/.style={arrows={-stealth},very thick,Egraphcolor}};

\draw [step=1, gray, very thin] (0,0) grid (3.25,1.25);
\draw [ -, black] (0,0)--(3.25,0);
\draw [ -, black] (0,0)--(0,1.25);

\node[mybullet]  (0) at (0,0) {};
\node[mybullet]  (X) at (1,0) {};
\node[mybullet] (2X) at (2,0) {};
\node[mybullet] (3X) at (3,0) {};
\node[mybullet] (XY) at (1,1) {};

\draw[myarrow]  (X) to node[below] {\tiny $1$} (0);
\draw[myarrow]  (0) to node[left] {\tiny $2$} (XY);
\draw[myarrow] (XY) to node[right] {\tiny $3$} (2X);
\draw[myarrow,transform canvas={yshift=2pt}] (2X) to node[above] {\tiny $4$} (3X);
\draw[myarrow,transform canvas={yshift=-2pt}] (3X) to node[below] {\tiny $5$} (2X);
\node at (2.5,1.2) {$G$};

\node at (2,-1) {$\begin{aligned}
\frac{\mathrm{d}x_1}{\mathrm{d}t} &= -\kappa_1 x_1 + \kappa_2 + \kappa_3 x_1 x_2 + \kappa_4 x_1^2 - \kappa_5 x_1^3 \\
\frac{\mathrm{d}x_2}{\mathrm{d}t} &=               \kappa_2 - \kappa_3 x_1 x_2 
\end{aligned}$};

\begin{scope}[shift={(4,0)}]

\draw [step=1, gray, very thin] (0,0) grid (3.25,1.25);
\draw [ -, black] (0,0)--(3.25,0);
\draw [ -, black] (0,0)--(0,1.25);

\node[mybullet]  (0) at (0,0) {};
\node[mybullet]  (X) at (1,0) {};
\node[mybullet] (2X) at (2,0) {};
\node[inner sep=1.2pt,outer sep=6pt,draw,fill,Egraphcolor,circle] (3X) at (3,0) {};
\node[mybullet] (XY) at (1,1) {};

\draw[myarrow]  (X) to node[below] {\tiny $1$} (0);
\draw[myarrow,transform canvas={xshift=-1.4pt, yshift=1.4pt}]  (0) to node[left] {\tiny $2$} (XY);
\draw[myarrow] (XY) to node[right] {\tiny $3$} (2X);
\draw[myarrow,transform canvas={yshift=2pt}] (2X) to node[above left] {\tiny $4$} (3X);
\draw[myarrow,transform canvas={yshift=-2pt}] (3X) to node[below left] {\tiny $5$} (2X);
\draw[myarrow,transform canvas={yshift=2pt}]  (X) to node[above] {\tiny $6$} (2X);
\draw[myarrow,transform canvas={yshift=-2pt}] (2X) to node[below] {\tiny $7$}  (X);
\draw[myarrow,transform canvas={xshift=1.4pt, yshift=-1.4pt}] (XY) to node[below] {\tiny $8$} (0);
\draw[myarrow] (XY) to node[left] {\tiny $9$} (X);
\draw[myarrow] (XY) to node[above] {\tiny $10$} (3X);

\node at (2.5,1.2) {$H_1$};

\end{scope}

\begin{scope}[shift={(8,0)}]

\draw [step=1, gray, very thin] (0,0) grid (3.25,1.25);
\draw [ -, black] (0,0)--(3.25,0);
\draw [ -, black] (0,0)--(0,1.25);

\node[mybullet]  (0) at (0,0) {};
\node[mybullet]  (X) at (1,0) {};
\node[mybullet] (2X) at (2,0) {};
\node[mybullet] (3X) at (3,0) {};
\node[mybullet] (XY) at (1,1) {};

\draw[myarrow]  (X) to node[below] {\tiny \fluxcolor{$\beta_2$}}  (0);
\draw[myarrow]  (0) to node[left]  {\tiny \fluxcolor{$\beta_2$}} (XY);
\draw[myarrow] (XY) to node[right] {\tiny \fluxcolor{$\beta_2$}} (2X);
\draw[myarrow,transform canvas={yshift=2pt}]  (2X) to node[above] {\tiny \fluxcolor{$\beta_5$}}            (3X);
\draw[myarrow,transform canvas={yshift=-2pt}] (3X) to node[below] {\tiny \fluxcolor{$\beta_5$}}            (2X);
\draw[myarrow,transform canvas={yshift=2pt}]   (X) to node[above,xshift=-5pt]      {\tiny \fluxcolor{$\beta_2 - \beta_1$}}  (2X);
\draw[myarrow,transform canvas={yshift=-2pt}] (2X) to node[below]      {\tiny \fluxcolor{$2\beta_2 - \beta_1$}}  (X);

\node at (2.5,1.2) {$(H_2,\fluxcolor{\vv\gamma})$};
\node at (3.5,1.5) {\phantom{0}};
\end{scope}

\draw[rounded corners,lightgray] (current bounding box.south west) rectangle (current bounding box.north east);

\end{tikzpicture}
\end{aligned}
\end{align}

The E-graph obtained from $G$ in \eqref{tikz:B-T_13} by omitting the reaction $3\X_1 \to 2\X_1$ is known to admit a subcritical Bogdanov--Takens bifurcation (and hence, a fold bifurcation of equilibria, a subcritical Andronov--Hopf bifurcation, and even a homoclinic bifurcation), see \cite{banaji:boros:hofbauer:2024b}*{Theorem 33}. It follows from \cite{banaji:boros:hofbauer:2025}*{Theorem 3.2} that $G$ inherits these bifurcations. Furthermore, $G$ admits 3 positive equilibria: e.g.,  there are 2 sinks and 1 saddle when $(\kappa_1, \kappa_2, \kappa_3, \kappa_4, \kappa_5) = (7, 1, 2, 7, 2)$.

Below, we compute the disguised toric flux cone $\Fdt(G)$ and the disguised toric locus $\Kdt(G)$. A short calculation shows that 
\begin{alignat*}{2}
    \Feq(G) &= \{\beta \in \rrpp^5 \st \beta_2 = \beta_3 \text{ and } \beta_2+\beta_3+\beta_4=\beta_1+\beta_5 \}, & \quad\quad \Keq(G) &= \rrpp^5, \\
     \Ft(G) &= \emptyset, & \quad\quad \Kt(G) &= \emptyset.
\end{alignat*}
Next, we compute $\Fdt(G)$. Consider the graph $H_1$ in \eqref{tikz:B-T_13}, and notice that it is obtained from the maximal graph $G^{\max}$ by omitting the reactions $2\X_1 \to \0$, $3\X_1 \to \0$, $\X_1 \rightleftarrows 3\X_1$. By \Cref{lem:rank-1}, we have $\Fdt(G,G^{\max}) = \Fdt(G,H_1)$, and hence, by \Cref{thm:FdtG_is_clFdtGGmax}, we find that $\Fdt(G) = (\overline{\Fdt(G, H_1)})_{>0}$. It is straightforward to see that for any $\vv\beta \in \rrpp^{5}$ and for any $\vv\gamma \in \mathbb{R}^{10}_{\geq0}$ for which $(G,\vv\beta)\triangleq (H_1,\vv\gamma)$ holds, we have $\beta_1 = \gamma_1 - \gamma_6$ and $\beta_2 = \gamma_2$. Further, assuming that $(H_1, \vv\gamma)$ is vertex-balanced, we have $\gamma_2 = \gamma_1 + \gamma_8$, which implies $\beta_2 - \beta_1 = \gamma_6 + \gamma_8 \geq 0$. Hence, 
\begin{align*}
\Fdt(G) \subseteq \{\vv\beta \in \Feq(G) \st \beta_2 \geq \beta_1\}.
\end{align*}
On the other hand, assuming $\vv\beta \in \Feq(G)$ fulfills $\beta_2\geq\beta_1$, the flux $\vv\gamma$ on $H_2$ that is shown in magenta in \eqref{tikz:B-T_13} is a vertex-balanced flux for which $(G,\vv\beta) \triangleq (H_2,\vv\gamma)$ holds (in the limit case $\beta_2 = \beta_1$, $H_2$ has one fewer edge). Consequently,
\begin{align*}
    \Fdt(G) = \{\vv\beta \in \Feq(G) \st \beta_2 \geq \beta_1\}.
\end{align*}
Hence, a $\vv\kappa \in \rrpp^5$ is in $\Kdt(G)$ if and only if $(G,\vv\kk)$ has an equilibrium $\xx^* \in \rrpp^2$ with $x_1^* \leq \tfrac{\kappa_2}{\kappa_1}$. A short Mathematica calculation gives
\begin{align*}
\Kdt(G) = \{\vv\kappa \in \rrpp^5 \st \tfrac{\kappa_5}{\kappa_1} \geq \tfrac{\kappa_1^2}{\kappa_2^2} + \tfrac{\kappa_4}{\kappa_2}\}.
\end{align*}
Numerical simulation shows that approximately 35.4\% of the simplex $\{\vv\kk \in \rrpp^5 \st \sum_{i=1}^5 \kk_i = 1\}$ belongs to $\Kdt(G)$.

Note that, for all $\vv\kappa \in \rrpp^5$, $(G,\vv\kappa)$ is dynamically equal to a planar, weakly reversible mass-action system (this can be seen by adding the reactions $\X_1 \stackrel{\varepsilon}{\leftarrow} 2\X_1 \stackrel{\varepsilon}{\to} 3\X_1$), and hence, it is permanent. Therefore, for all $\vv\kappa \in \Kdt(G)$, the unique positive equilibrium of $(G,\vv\kappa)$ is globally asymptotically stable. The more involved dynamics described in the first paragraph of this subsection can only occur for $\vv \kappa \in \rrpp^5$ with $\tfrac{\kappa_5}{\kappa_1} < \tfrac{\kappa_1^2}{\kappa_2^2} + \tfrac{\kappa_4}{\kappa_2}$.

\subsection{Basic clock mechanism}

\begin{align} \label{tikz:circadian}
\begin{aligned}
\begin{tikzpicture}[scale=1.2]

\begin{axis}[view={105}{30}, axis line style=white, width=5.5cm, height=5.5cm, ticks=none, xmin=0, xmax=1.9, ymin=0, ymax=1.9, zmin=0, zmax=1.5]

    \tikzset{bullet/.style={inner sep=1pt,outer sep=1.5pt,draw,fill,Egraphcolor,circle}};
    \tikzset{myarrow/.style={arrows={-stealth},thick,Egraphcolor}};
    \tikzset{mygrid/.style={very thin,gray!50}};
    \tikzset{myaxis/.style={thick,gray}};    
    \newcommand{\ab}{2.5};    
 
    \draw[mygrid] (1,0,\ab) -- (1,0,0) -- (1,\ab,0);
    \draw[mygrid] (2,0,\ab) -- (2,0,0) -- (2,\ab,0);
    \draw[mygrid] (0,1,\ab) -- (0,1,0) -- (\ab,1,0);
    \draw[mygrid] (0,2,\ab) -- (0,2,0) -- (\ab,2,0);
    \draw[mygrid] (0,\ab,1) -- (0,0,1) -- (\ab,0,1);
    \draw[mygrid] (0,\ab,2) -- (0,0,2) -- (\ab,0,2);
        
    \draw[myaxis] (0,0,0) -- (\ab,0,0);    
    \draw[myaxis] (0,0,0) -- (0,\ab,0);    
    \draw[myaxis] (0,0,0) -- (0,0,\ab);    
    
    \node[bullet] (P1) at (0,0,0) {};
    \node[bullet] (P2) at (1,0,0) {};
    \node[bullet] (P3) at (0,1,0) {};
    \node[bullet] (P4) at (1,1,0) {};
    \node[bullet] (P5) at (0,0,1) {};
    
    \node [above left]  at (P1) {\small $\0$};
    \node [below right] at (P2) {\small $\mathsf{T}$};
    \node [above right] at (P3) {\small $\mathsf{P}$};
    \node [right] at (P4) {\small $\mathsf{T}+\mathsf{P}$};
    \node [above right] at (P5) {\small $\mathsf{C}$};
    \node[right] at (0,1.25,1.25) {$G$};
    
    \draw[myarrow,transform canvas={xshift=-1.5pt,yshift=1pt}]  (P1) to node[xshift=-2pt,yshift=3pt] {\tiny $1$} (P2);
    \draw[myarrow,transform canvas={xshift=1.5pt,yshift=-1pt}] (P2) to node[xshift=0pt,yshift=-5pt] {\tiny $2$} (P1);
    
    \draw[myarrow,transform canvas={xshift=1.5pt,yshift=1.5pt}] (P1) to node[xshift=12pt,yshift=4pt] {\tiny $3$} (P3);
    \draw[myarrow,transform canvas={xshift=-1.5pt,yshift=-1.5pt}]  (P3) to node[xshift=12pt,yshift=-5pt] {\tiny $4$} (P1);
    
    \draw[myarrow,transform canvas={xshift=1.5pt,yshift=1pt}] (P4) to node[xshift=2pt,yshift=6pt] {\tiny $5$} (P5);
    \draw[myarrow,transform canvas={xshift=-1.5pt,yshift=-1pt}] (P5) to node[xshift=-4pt,yshift=0pt] {\tiny $6$} (P4);

    \draw[myarrow] (P5) to node[xshift=-3pt] {\tiny $7$} (P1);
   
\end{axis}

\node at (2,-1) {$\begin{aligned}
\frac{\mathrm{d}[\sf{T}]}{\mathrm{d}t} &= \kappa_1 - \kappa_2 [\sf{T}] - \kappa_5 [\sf{T}] [\sf{P}] + \kappa_6 [\sf{C}] \\
\frac{\mathrm{d}[\sf{P}]}{\mathrm{d}t} &= \kappa_3 - \kappa_4 [\sf{P}] - \kappa_5 [\sf{T}] [\sf{P}] + \kappa_6 [\sf{C}] \\
\frac{\mathrm{d}[\sf{C}]}{\mathrm{d}t} &= \kappa_5 [\sf{T}] [\sf{P}] - \kappa_6 [\sf{C}] - \kappa_7 [\sf{C}]
\end{aligned}$};

\begin{scope}[shift={(5,0)}]

\begin{axis}[view={105}{30}, axis line style=white, width=5.5cm, height=5.5cm, ticks=none, xmin=0, xmax=1.9, ymin=0, ymax=1.9, zmin=0, zmax=1.5]

    \tikzset{bullet/.style={inner sep=1pt,outer sep=1.5pt,draw,fill,Egraphcolor,circle}};
    \tikzset{myarrow/.style={arrows={-stealth},thick,Egraphcolor}};
    \tikzset{mygrid/.style={very thin,gray!50}};
    \tikzset{myaxis/.style={thick,gray}};    
    \newcommand{\ab}{2.5};    
 
    \draw[mygrid] (1,0,\ab) -- (1,0,0) -- (1,\ab,0);
    \draw[mygrid] (2,0,\ab) -- (2,0,0) -- (2,\ab,0);
    \draw[mygrid] (0,1,\ab) -- (0,1,0) -- (\ab,1,0);
    \draw[mygrid] (0,2,\ab) -- (0,2,0) -- (\ab,2,0);
    \draw[mygrid] (0,\ab,1) -- (0,0,1) -- (\ab,0,1);
    \draw[mygrid] (0,\ab,2) -- (0,0,2) -- (\ab,0,2);
        
    \draw[myaxis] (0,0,0) -- (\ab,0,0);    
    \draw[myaxis] (0,0,0) -- (0,\ab,0);    
    \draw[myaxis] (0,0,0) -- (0,0,\ab);    
    
    \node[style={inner sep=1pt,outer sep=3.5pt,draw,fill,Egraphcolor,circle}] (P1) at (0,0,0) {};
    \node[bullet] (P2) at (1,0,0) {};
    \node[bullet] (P3) at (0,1,0) {};
    \node[style={inner sep=1pt,outer sep=3.5pt,draw,fill,Egraphcolor,circle}] (P4) at (1,1,0) {};
    \node[bullet] (P5) at (0,0,1) {};
    
    \node [above left]  at (P1) {\small $\0$};
    \node [below right] at (P2) {\small $\mathsf{T}$};
    \node [above right] at (P3) {\small $\mathsf{P}$};
    \node [right] at (P4) {\small $\mathsf{T}+\mathsf{P}$};
    \node [above right] at (P5) {\small $\mathsf{C}$};
    \node at (0,1.25,1.25) {$(G^{\max},\fluxcolor{\vv\gamma})$};
    
    \draw[myarrow,transform canvas={xshift=-1.5pt,yshift=1pt}]  (P1) to node[xshift=-4pt,yshift=3pt] {\tiny \textcolor{magenta}{$\beta_2$}} (P2);
    \draw[myarrow,transform canvas={xshift=1.5pt,yshift=-1pt}] (P2) to node[xshift=2pt,yshift=-5pt] {\tiny \textcolor{magenta}{$\beta_2$}} (P1);
    
    \draw[myarrow,transform canvas={xshift=1.5pt,yshift=1.5pt}] (P1) to node[xshift=12pt,yshift=4pt] {\tiny \textcolor{magenta}{$\beta_4$}} (P3);
    \draw[myarrow,transform canvas={xshift=-1.5pt,yshift=-1.5pt}]  (P3) to node[xshift=12pt,yshift=-5pt] {\tiny \textcolor{magenta}{$\beta_4$}} (P1);
    
    \draw[myarrow,transform canvas={xshift=1.5pt,yshift=1pt}] (P4) to node[sloped,xshift=-5pt,yshift=4pt] {\tiny \textcolor{magenta}{$\beta_6+\beta_7$}} (P5);
    \draw[myarrow,transform canvas={xshift=-1.5pt,yshift=-1pt}] (P5) to node[xshift=-4pt,yshift=0pt] {\tiny \textcolor{magenta}{$\beta_6$}} (P4);

    \draw[myarrow] (P5) to node[xshift=-5pt] {\tiny \textcolor{magenta}{$\beta_7$}} (P1);
    \draw[myarrow] (P1) to node[xshift=-5pt] {\tiny \textcolor{magenta}{$\beta_7$}} (P4);

\end{axis}
\node at (4,4) {\phantom{0}};
\end{scope}

\draw[rounded corners,lightgray] (current bounding box.south west) rectangle (current bounding box.north east);

\end{tikzpicture}
\end{aligned}
\end{align}

The E-graph $G$ in \eqref{tikz:circadian} is a simplified model of the circadian clock mechanism in \textit{Drosophila}, where $\sf{P}$ and $\sf{T}$ refer to period (PER) and timeless (TIM) proteins, respectively, while $\sf{C}$ stands for the PER-TIM compound, see \cites{johnston:pantea:donnell:2016, leloup:goldbeter:1999}. Though $G$ is not weakly reversible, for all $\vv \kappa \in \rrpp^7$ there exists a $\vv \lambda \in \rrpp^8$ such that $\vv{f}_{(G,\vv\kappa)} = \vv{f}_{(G^{\max},\vv\lambda)}$ ($G^{\max}$ is shown on the right in \eqref{tikz:circadian}). Hence, $(G,\vv \kappa)$ is dynamically equal to a weakly reversible mass-action system for all $\vv \kappa \in \rrpp^7$, and therefore, $\Keq(G) = \rrpp^7$, see \cite{boros:2019}. It has been shown in \cite{craciun:deshpande:jin:2025b}*{Example 5.3} that $\Kdt(G)$ is a seven-dimensional semialgebraic set in $\rrpp^7$. Below, we argue that, in fact, $\Kdt(G) = \rrpp^7$.

Note that $\Feq(G) = \{ \vv\beta \in \rrpp^7 \st \beta_1 - \beta_2 = \beta_3 - \beta_4 = \beta_5 - \beta_6 = \beta_7 \}$. Further, the E-graph $G^{\max}$ endowed with the flux $\vv\gamma$ shown in magenta is dynamically equal to that of $G$ endowed with the equilibrium flux $\vv\beta$. Since $\vv\gamma$ is vertex-balanced, we conclude that $\Fdt(G) = \Feq(G)$. Hence, by \Cref{prop:Fdt_iff_Kdt}(c), we have $\Kdt(G) = \Keq(G) = \rrpp^7$, i.e., the pair $(G,\vv\kappa)$ is disguised vertex-balanced for all $\vv\kappa \in \rrpp^7$. Since $G^{\max}$ has a single connected component, $(G,\vv\kappa)$ is globally asymptotically stable for all $\vv\kappa \in \rrpp^7$.

\subsection{Tetrahedron}
\label{subsec:tetrahedron}

\begin{align} \label{tikz:tetrahedron}
\begin{aligned}
\begin{tikzpicture}[scale=1.2]

\begin{axis}[view={110}{55}, axis line style=white, width=5.5cm, height=6cm, ticks=none, xmin=0, xmax=3, ymin=0, ymax=3, zmin=0, zmax=3.5]

    \tikzset{bullet/.style={inner sep=1pt,outer sep=1.5pt,draw,fill,Egraphcolor,circle}};
    \tikzset{myarrow/.style={arrows={-stealth},thick,Egraphcolor}};
    \tikzset{mygrid/.style={very thin,gray!50}};
    \tikzset{myaxis/.style={thick,gray}};    
    \newcommand{\ab}{2.75};    
 
    \draw[mygrid] (1,0,\ab) -- (1,0,0) -- (1,\ab,0);
    \draw[mygrid] (2,0,\ab) -- (2,0,0) -- (2,\ab,0);
    \draw[mygrid] (0,1,\ab) -- (0,1,0) -- (\ab,1,0);
    \draw[mygrid] (0,2,\ab) -- (0,2,0) -- (\ab,2,0);
    \draw[mygrid] (0,\ab,1) -- (0,0,1) -- (\ab,0,1);
    \draw[mygrid] (0,\ab,2) -- (0,0,2) -- (\ab,0,2);
        
    \draw[myaxis] (0,0,0) -- (\ab,0,0);    
    \draw[myaxis] (0,0,0) -- (0,\ab,0);    
    \draw[myaxis] (0,0,0) -- (0,0,\ab);    
    
    \node[bullet] (P1) at (0,0,0) {};
    \node[style={inner sep=1pt,outer sep=2.5pt,draw,fill,Egraphcolor,circle}] (P2) at (1,0,0) {};
    \node[bullet] (P3) at (2,0,0) {};
    \node[bullet] (P4) at (0,2,0) {};
    \node[bullet] (P5) at (0,1,1) {};
    \node[bullet] (P6) at (0,0,2) {};
    
    \node [above left]  at (P1) {\small $\0$};
    \node [below right]  at (P2) {\small $\X_1$};
    \node [below right]  at (P3) {\small $2\X_1$};
    \node [above right] at (P4) {\small $2\X_2$};
    \node [above right] at (P5) {\small $\X_2+\X_3$};
    \node [above right] at (P6) {\small $2\X_3$};
    \node at (1.5,1.5,0) {$G$};
    
    \draw[myarrow,transform canvas={xshift=-1.5pt,yshift=0.5pt}]  (P1) to node[xshift=-2pt,yshift=2pt] {\tiny $1$} (P2);
    \draw[myarrow,transform canvas={xshift=1.5pt,yshift=-0.5pt}] (P2) to node[xshift=5pt,yshift=4pt] {\tiny $2$} (P1);
    
    \draw[myarrow,transform canvas={xshift=-1.5pt,yshift=0.5pt}] (P2) to node[xshift=-2pt,yshift=2pt] {\tiny $3$} (P3);
    \draw[myarrow,transform canvas={xshift=1.5pt,yshift=-0.5pt}]  (P3) to node[xshift=2pt,yshift=-2pt] {\tiny $4$} (P2);
    
    \draw[myarrow,transform canvas={xshift=-1pt,yshift=1pt}] (P2) to node[xshift=2pt,yshift=6pt] {\tiny $5$} (P5);
    \draw[myarrow,transform canvas={xshift=1pt,yshift=-1pt}] (P5) to node[xshift=2pt,yshift=-2pt] {\tiny $6$} (P2);

    \draw[myarrow,transform canvas={xshift=.5pt,yshift=1.5pt}] (P4) to node[above] {\tiny $7$} (P5);
    \draw[myarrow,transform canvas={xshift=-.5pt,yshift=-1.5pt}]  (P5) to node[below] {\tiny $8$} (P4);

    \draw[myarrow,transform canvas={xshift=.5pt,yshift=1.5pt}] (P5) to node[above] {\tiny $9$} (P6);
    \draw[myarrow,transform canvas={xshift=-.5pt,yshift=-1.5pt}]  (P6) to node[below] {\tiny $10$} (P5);

\end{axis}

\node at (8,2) {$\begin{aligned}
\frac{\mathrm{d}x_1}{\mathrm{d}t} &= \kappa_1 + (\kappa_3 - \kappa_2- \kappa_5) x_1 - \kappa_4 x_1^2 + \kappa_6 x_2 x_3 \\ 
\frac{\mathrm{d}x_2}{\mathrm{d}t} &= \kappa_5 x_1 - \kappa_7 x_2^2 + (\kappa_8 -\kappa_9 - \kappa_6) x_2 x_3 + \kappa_{10} x_3^2 \\
\frac{\mathrm{d}x_3}{\mathrm{d}t} &= \kappa_5 x + \kappa_7 x_2^2 - (\kappa_8 -\kappa_9 + \kappa_6) x_2 x_3 - \kappa_{10} x_3^2 \\
\end{aligned}$};

\node at (-.25,2) {\phantom{0}};
\node at (11.75,2) {\phantom{0}};

\draw[rounded corners,lightgray] (current bounding box.south west) rectangle (current bounding box.north east);

\end{tikzpicture}
\end{aligned}
\end{align}

Motivated by the network in \cite{johnston:siegel:szederkenyi:2013:mathbiosci}*{Fig. 4 (b)}, let us consider the E-graph $G$ (in $\rr^3$) shown in \eqref{tikz:tetrahedron}. Since $G$ is weakly reversible, $\Keq(G) = \rrpp^{10}$, see \cite{boros:2019}. Further, since $G$ is of deficiency two, the toric locus $\Kt(G)$ has codimension two. Indeed, $\Kt(G) = \{\vv\kappa \in \rrpp^{10} \st \tfrac{\kappa_1}{\kappa_2} = \tfrac{\kappa_3}{\kappa_4} \text{ and } \tfrac{\kappa_7}{\kappa_8} = \tfrac{\kappa_9}{\kappa_{10}}\}$. Below, we show that the disguised toric locus $\Kdt(G)$ equals $\rrpp^{10}$.

The set of equilibrium fluxes is
\begin{align*}
    \Feq(G) = \{\vv\beta \in \rrpp^{10} \st \beta_1 - \beta_2 + \beta_3 - \beta_4 = 0,\, \beta_5 - \beta_6 = 0,\, \beta_7 - \beta_8 + \beta_9 - \beta_{10} = 0\}.
\end{align*}
Note that a $\vv\beta \in \Feq(G)$ is in $\Fdt(G,G)$ if and only if there exists a $\vv \gamma \in \rrpp^{10}$ such that
\begin{gather}
\begin{alignedat}{3}
  \beta_1 &= \gamma_1, & \qquad \beta_6    &= \gamma_6,    & \qquad \beta_3 - \beta_2 - \beta_5  &= \gamma_3 - \gamma_2 - \gamma_5, \\
  \beta_4 &= \gamma_4, & \qquad \beta_7    &= \gamma_7,    & \qquad \beta_8 - \beta_9 - \beta_6  &= \gamma_8 - \gamma_9 - \gamma_6, \\
  \beta_5 &= \gamma_5, & \qquad \beta_{10} &= \gamma_{10}, & \qquad \beta_8 - \beta_9 + \beta_6  &= \gamma_8 - \gamma_9 + \gamma_6;
\end{alignedat}
\tag{DE-f} \\
\begin{alignedat}{3}
  \gamma_1 &= \gamma_2, & \qquad \gamma_7    &= \gamma_8, & \qquad \gamma_2 + \gamma_3 + \gamma_5 &= \gamma_1 + \gamma_4 + \gamma_6, \\
  \gamma_4 &= \gamma_3, & \qquad \gamma_{10} &= \gamma_9, & \qquad \gamma_8 + \gamma_9 + \gamma_6 &= \gamma_7 + \gamma_{10} + \gamma_5.
\end{alignedat}
\tag{VB-f}
\end{gather}
It is straightforward to see that for all $\vv\beta \in \Feq(G)$ there exists a unique $\vv\gamma \in \rrpp^{10}$ that satisfies all 15 linear equations, namely,
\begin{align*}
    \gamma_1 = \gamma_2 = \beta_1, \quad\quad \gamma_3 = \gamma_4 = \beta_4, \quad\quad \gamma_5 = \gamma_6 = \beta_5, \quad\quad \gamma_7 = \gamma_8 = \beta_7, \quad\quad \gamma_9 = \gamma_{10} = \beta_{10}.  
\end{align*}
Hence, $\Fdt(G,G) = \Feq(G)$. Since in general $\Fdt(G,G) \subseteq \Fdt(G) \subseteq \Feq(G)$, it follows that $\Fdt(G) = \Feq(G)$. Consequently, by \Cref{prop:Fdt_iff_Kdt}(c), we have $\Kdt(G) = \Keq(G) = \rrpp^{10}$. Furthermore, since $G$ has only a single connected component, we conclude that, for all $\vv\kappa \in \rrpp^{10}$, the unique positive equilibrium is globally asymptotically stable \cite{anderson:2011a}.

Finally, we argue that $\Kdt(G, G) = \rrpp^{10}$ (and hence, $\Kdt(G) = \rrpp^{10}$) can be seen directly. Indeed, for $\vv\kappa \in \rrpp^{10}$, let $\xx^*$ be a positive equilibrium (we do not need to know a priori that it is unique) and define $\vv\lambda \in \rrpp^{10}$ by
\begin{alignat*}{5}
    \lambda_1 &= \kappa_1, & \quad\quad \lambda_3 &= \kappa_4 x_1^*, & \quad\quad \lambda_5 &= \kappa_5, & \quad\quad \lambda_7 &= \kappa_7, & \quad\quad \lambda_9 &= \kappa_{10}\tfrac{x_3^*}{x_2^*}, \\
    \lambda_2 &= \kappa_1\tfrac{1}{x_1^*}, & \quad\quad \lambda_4 &= \kappa_4, & \quad\quad \lambda_6 &= \kappa_5 \tfrac{x_1^*}{x_2^*x_3^*}, & \quad\quad \lambda_8 &= \kappa_7\tfrac{x_2^*}{x_3^*}, & \quad\quad \lambda_{10} &= \kappa_{10}.
\end{alignat*}
Since $\tfrac{\lambda_1}{\lambda_2} = \tfrac{\lambda_3}{\lambda_4}$ and $\tfrac{\lambda_7}{\lambda_8} = \tfrac{\lambda_9}{\lambda_{10}}$, we have $\vv\lambda \in \Kt(G)$. Further, a direct calculation shows that $(G,\vv\kappa)$ and $(G,\vv\lambda)$ are dynamically equal. Hence, indeed $\Kdt(G) = \rrpp^{10}$.

\subsection{A four-dimensional example}
\label{subsec:4d}

\begin{align} \label{tikz:4d}
\begin{aligned}
\begin{tikzpicture}[scale=1.9]

\tikzset{myarrow/.style={arrows={-Latex},thick}};

\node[outer sep=5pt]  (0) at (0,0)  {$\0$};
\node  (X2) at (1,0)  {$\X_2$};
\node (2X2) at (2,0)  {$2\X_2$};
\node  (X1) at (1,0.75)  {$\X_1$};
\node (2X1) at (2,1.5)  {$2\X_1$};
\node  (X3) at (1,-0.75) {$\X_3$};
\node (2X3) at (2,-1.5) {$2\X_3$};
\node[outer sep=5pt]  (X0) at (3,0)  {$\X_0$};

\draw[myarrow, transform canvas={xshift=-1.4pt,yshift=1.4pt}]  (0) to node[left]  {\tiny $1$} (X1);
\draw[myarrow, transform canvas={xshift=1.4pt,yshift=-1.4pt}]  (X1) to node[right] {\tiny $2$} (0);
\draw[myarrow, transform canvas={yshift=2pt}]                  (0) to node[above] {\tiny $3$} (X2);
\draw[myarrow, transform canvas={yshift=-2pt}]                 (X2) to node[below] {\tiny $4$} (0);
\draw[myarrow, transform canvas={xshift=1.4pt,yshift=1.4pt}]   (0) to node[right] {\tiny $5$} (X3);
\draw[myarrow, transform canvas={xshift=-1.4pt,yshift=-1.4pt}] (X3) to node[left]  {\tiny $6$} (0);

\draw[myarrow, transform canvas={xshift=-1.8pt,yshift=-0.9pt}]  (X0) to node[left]  {\tiny $9$}  (2X1);
\draw[myarrow, transform canvas={xshift= 1.8pt,yshift= 0.9pt}] (2X1) to node[right] {\tiny $10$}  (X0);
\draw[myarrow, transform canvas={yshift=-2pt}]                  (X0) to node[below] {\tiny $11$} (2X2);
\draw[myarrow, transform canvas={yshift=2pt}]                  (2X2) to node[above] {\tiny $12$}  (X0);
\draw[myarrow, transform canvas={xshift= 1.8pt,yshift=-0.9pt}]  (X0) to node[right] {\tiny $13$} (2X3);
\draw[myarrow, transform canvas={xshift=-1.8pt,yshift=0.9pt}]  (2X3) to node[left]  {\tiny $14$}  (X0);

\draw[myarrow,transform canvas={yshift=2pt}] (0) to[out=25, in=155]  node[above] {\tiny $7$}  (X0);
\draw[myarrow,transform canvas={yshift=-2pt}] (X0) to[out=155, in=25] node[below] {\tiny $8$}  (0);

\node at (0.5,1.25) {$G$};

\node at (1.35,-1.6) {$\begin{aligned}
    \dot{x}_1 &= \kappa_1 - \kappa_2 x_1 + 2(\kappa_9    x_0 - \kappa_{10} x_1^2) \\
    \dot{x}_2 &= \kappa_3 - \kappa_4 x_2 + 2(\kappa_{11} x_0 - \kappa_{12} x_2^2) \\
    \dot{x}_3 &= \kappa_5 - \kappa_6 x_3 + 2(\kappa_{13} x_0 - \kappa_{14} x_3^2) \\
    \dot{x}_0 &= \kappa_7 - \kappa_8 x_0 - (\kappa_9 + \kappa_{11} + \kappa_{13})x_0 + \kappa_{10} x_1^2 + \kappa_{12} x_2^2 + \kappa_{14} x_3^2
\end{aligned}$};

\begin{scope}[shift={(3.5,0)}]
    
\node[outer sep=5pt]  (0) at (0,0)  {$\0$};
\node  (X2) at (1,0)  {$\X_2$};
\node (2X2) at (2,0)  {$2\X_2$};
\node  (X1) at (1,.75)  {$\X_1$};
\node (2X1) at (2,1.5)  {$2\X_1$};
\node  (X3) at (1,-.75) {$\X_3$};
\node (2X3) at (2,-1.5) {$2\X_3$};
\node[outer sep=5pt]  (X0) at (3,0)  {$\X_0$};

\draw[myarrow, transform canvas={xshift=-1.4pt,yshift=1.4pt}]  (0) to node[left]  {\tiny $1$} (X1);
\draw[myarrow, transform canvas={xshift=1.4pt,yshift=-1.4pt}]  (X1) to node[right] {\tiny $2$} (0);
\draw[myarrow, transform canvas={yshift=2pt}]                  (0) to node[above] {\tiny $3$} (X2);
\draw[myarrow, transform canvas={yshift=-2pt}]                 (X2) to node[below] {\tiny $4$} (0);
\draw[myarrow, transform canvas={xshift=1.4pt,yshift=1.4pt}]   (0) to node[right] {\tiny $5$} (X3);
\draw[myarrow, transform canvas={xshift=-1.4pt,yshift=-1.4pt}] (X3) to node[left]  {\tiny $6$} (0);

\draw[myarrow, transform canvas={xshift=-1.8pt,yshift=-0.9pt}]  (X0) to node[left]  {\tiny $9$}  (2X1);
\draw[myarrow, transform canvas={xshift= 1.8pt,yshift= 0.9pt}] (2X1) to node[right] {\tiny $10$}  (X0);
\draw[myarrow, transform canvas={yshift=-2pt}]                  (X0) to node[below] {\tiny $11$} (2X2);
\draw[myarrow, transform canvas={yshift=2pt}]                  (2X2) to node[above] {\tiny $12$}  (X0);
\draw[myarrow, transform canvas={xshift= 1.8pt,yshift=-0.9pt}]  (X0) to node[right] {\tiny $13$} (2X3);
\draw[myarrow, transform canvas={xshift=-1.8pt,yshift=0.9pt}]  (2X3) to node[left]  {\tiny $14$}  (X0);

\draw[myarrow] (X1) to node[left]  {\tiny $15$} (2X1);
\draw[myarrow] (X2) to node[above] {\tiny $16$} (2X2);
\draw[myarrow] (X3) to node[right] {\tiny $17$} (2X3);

\draw[myarrow] (X0) to[out=135, in=0] node[above] {\tiny $18$}  (X1);
\draw[myarrow, bend left] (X0) to     node[below] {\tiny $19$}  (X2);
\draw[myarrow] (X0) to[out=225, in=0] node[below] {\tiny $20$}  (X3);

\draw[myarrow,transform canvas={yshift=2pt}] (0) to[out=25, in=155]  node[above] {\tiny $7$}  (X0);
\draw[myarrow,transform canvas={yshift=-2pt}] (X0) to[out=155, in=25] node[below] {\tiny $8$}  (0);

\node at (0.5,1.25) {$H$};

\end{scope}

\draw[rounded corners,lightgray] (current bounding box.south west) rectangle (current bounding box.north east);

\end{tikzpicture}
\end{aligned}
\end{align}

The E-graph $G$ in \eqref{tikz:4d} first appeared in \cite{craciun:pantea:2008}*{Section 6}, where, using Species-Reaction graph theory \cite{craciun:feinberg:2006}, it has been shown that $G$ does not admit multiple positive equilibria. Since $G$ is weakly reversible, we can conclude that for all rate constants, there exists a unique positive equilibrium.

By definition, a $\vv\beta \in \rrpp^{14}$ is in $\Feq(G)$ if and only if
\begin{align*}
    \beta_{8} - \beta_7     &= \tfrac12(\beta_1 - \beta_2 + \beta_3 - \beta_4 + \beta_5 - \beta_6), \\
    \beta_{10} - \beta_9    &= \tfrac12(\beta_1 - \beta_2), \\
    \beta_{12} - \beta_{11} &= \tfrac12(\beta_3 - \beta_4), \\
    \beta_{14} - \beta_{13} &= \tfrac12(\beta_5 - \beta_6).  
\end{align*}
The toric flux cone $\Ft(G)$ and the toric locus $\Kt(G)$ are given by
\begin{align*}
    \Ft(G) &= \{ \vv\beta \in \Feq(G) \st \beta_1 - \beta_2 = \beta_3 - \beta_4 = \beta_5 - \beta_6 = 0\}, \\
    \Kt(G) &= \{ \vv\kappa \in \rrpp^{14} \st \tfrac{\kappa_1^2}{\kappa_2^2}\tfrac{\kappa_{10}}{\kappa_9} = \tfrac{\kappa_3^2}{\kappa_4^2}\tfrac{\kappa_{12}}{\kappa_{11}} = \tfrac{\kappa_5^2}{\kappa_6^2}\tfrac{\kappa_{14}}{\kappa_{13}} = \tfrac{\kappa_7}{\kappa_8}\}.
\end{align*}
Note that the graph $H$ (depicted on the right in \eqref{tikz:4d}) is obtained from $G^{\max}$ by omitting $\0 \to 2\X_i$ for $i=1,2,3$. By \Cref{lem:rank-1}, we have $\Fdt(G,G^{\max}) = \Fdt(G,H)$.  
Following the same steps as in \Cref{subsec:square_rev,subsec:square_parallelogram,subsec:tetrahedron} (but with $H$ instead of $G^{\max}$), and solving the resulting linear quantifier elimination problem, we find that a $\vv \beta \in \Feq(G)$ is in $\Fdt(G)$ if and only if
\begin{align} \label{eq:4d}
    \beta_1 - \max(\beta_2,\tfrac12(\beta_1+\beta_2)) + \beta_3 - \max(\beta_4,\tfrac12(\beta_3+\beta_4)) + \beta_5 - \max(\beta_6,\tfrac12(\beta_5+\beta_6)) + \beta_7 \geq 0,
\end{align}
see \cite{boros_github} for more details. Albeit it might be challenging to find an explicit description of the disguised toric locus $\Kdt(G)$, it follows from \Cref{cor:Kdt_G} and the formula \eqref{eq:4d} that $\dim \Kdt(G) = 14$. For comparison, recall that $ \dim \Kt(G) = 11$ (since $G$ has a deficiency of 3). 

We conclude this section by noting that a numerical experiment shows that approximately 62.6\% of the simplex $\{\vv\kk \in \rrpp^{14} \st \sum_{i=1}^{14} \kk_i = 1\}$ belongs to $\Kdt(G)$. We have determined this by numerically computing the unique positive equilibrium for a randomly picked rate constant, and then verifying whether the corresponding flux vector satisfies \eqref{eq:4d}, see \cite{boros_github}. This example highlights the strength of the approach that we presented in this paper: the disguised toric flux cone (which is the solution of a linear problem) can give valuable information about the disguised toric locus (which is the solution of a nonlinear problem).
\section{Discussion}
\label{sec:discussion}

In this paper, we introduced a flux-based framework for analyzing the disguised toric locus of a reaction network.
Our results build upon and advance prior work \cites{craciun:jin:sorea:2024, craciun:deshpande:jin:2025a}. 
Our first main result, \Cref{thm:homeo_Psi}, establishes a homeomorphism relating the disguised toric locus $\Kdt(G)$ of a reaction network $G$ with its disguised toric flux cone $\Fdt(G)$. This result provides a practical computational approach. Since $\Fdt(G)$ is a polyhedral cone, its linear structure often allows for an explicit characterization of $\Kdt(G)$, which is typically more complex.

The flux cone $\Fdt(G)$ decomposes into flux cones $\Fdt(G,H)$, each corresponding to an E-graph $H \subseteq G^{\comp}$ that admits a vertex-balanced realization. This decomposition can nevertheless be challenging to interpret, since the number of admissible subgraphs $H$ may be very large.
Our second main result, \Cref{thm:FdtG_is_clFdtGGmax} and \Cref{cor:Kdt_G}, shows the existence of a maximal weakly reversible realization graph $G^{\max}$ such that every admissible realization $H$ is a subgraph of $G^{\max}$ and $\Fdt(G,G^{\max})$ coincides with the manifold interior of $\Fdt(G)$.
Combined with the homeomorphism in our first main result, this yields an analogous relation between $\Kdt(G,G^{\max})$ and $\Kdt(G)$.

Furthermore, \Cref{sec:algorithm} outlines how $G^{\max}$ can be computed efficiently via a linear feasibility problem. Once $G^{\max}$ is obtained, evaluating $\Fdt(G, G^{\max})$ allows us to determine $\Fdt(G)$, and consequently $\Kdt(G)$.
Finally, in \Cref{sec:examples}, we illustrated the applicability of our approach through several biologically relevant models, including the square and the parallelogram systems, the reversible Lotka–Volterra autocatalator, and a basic clock mechanism.

The approaches developed in this work suggest several directions for future research. 
One promising direction is to further explore the relationship between the topology of $\Kdt(G)$ and the polyhedral geometry of $\Fdt(G)$, with the goal of characterizing how $\Kdt(G)$ is embedded within the positive orthant, and what subgraphs $H\subseteq G^{\max}$ contribute to its boundary.
Another potential direction involves exploiting changes of coordinates. While this work focuses on the disguised toric locus, which requires two related systems to be dynamically equal, this condition can be relaxed by asking for equality up to a change of coordinates. Systems related in this way still share many dynamical properties, including stability and multistationarity. Hence, investigating the size and structure of this larger set may yield a deeper insight into the dynamics of mass-action systems.

\section*{Acknowledgments}
We acknowledge the POSTECH Center for Mathematical Machine Learning and its Applications (CM2LA) in Pohang, South Korea, for hosting a Seorabeol Research Station workshop in July 2024, where the initial steps of this research were taken.
We also thank an anonymous referee for useful comments, including the suggestion to investigate the boundary stratification of the disguised toric locus.
\section*{Funding}

BB was supported by the National Research, Development and Innovation Office, Hungary, grant RRF-2.3.1-21-2022-00006. GC was supported in part by the US National Science Foundation, grant DMS-2051568. DR was supported in part by the Fulbright Program.
OH was partially funded by the European Union (Grant Agreement no.~101044561, POSALG). Views and opinions expressed are those of the authors only and do not necessarily reflect those of the European Union or European Research Council (ERC). Neither the European Union nor ERC can be held responsible for them.

\bibliographystyle{amsalpha} 
\bibliography{references}

\end{document}